\documentclass[reqno,11pt]{amsart}
\usepackage[shortlabels]{enumitem}
\usepackage{hyperref}
\usepackage[
    backend=biber,       
    style=numeric,     
    doi=true,            
    url=true,            
    giveninits=true,      
]{biblatex}
\usepackage[nameinlink,capitalise]{cleveref}
\usepackage[margin=1in]{geometry}
\usepackage{tikz-cd}
\usepackage{bm}
\usepackage{physics}
\usepackage{geometry}
\usepackage{amsfonts}
\usepackage{appendix}

\renewcommand*{\bibnamedash}{\rule[0.4ex]{3em}{0.6pt}\space}
\makeatletter
\AtEveryBibitem{%
  \ifnameundef{author}
    {\global\undef\bbx@lasthash}
    {\iffieldequals{fullhash}{\bbx@lasthash}
       {\savefield{fullhash}{\bbx@lasthash}%
        \renewbibmacro*{author}{\bibnamedash}}
       {\savefield{fullhash}{\bbx@lasthash}}}%
}
\makeatother
\renewbibmacro*{in:}{%
  \ifentrytype{article}{}{\printtext{\bibstring{in}\intitlepunct}}}

\usepackage[textsize=scriptsize,backgroundcolor=white]{todonotes}

\theoremstyle{plain}
\newtheorem{thm}{Theorem}
\newtheorem{cor}[thm]{Corollary}
\newtheorem{prop}[thm]{Proposition}
\newtheorem{lem}[thm]{Lemma}
\newtheorem{asmp}{Assumption}
\newtheorem*{exmp*}{Example}

\theoremstyle{definition}
\newtheorem{defn}[thm]{Definition}

\theoremstyle{remark}
\newtheorem{rk}[thm]{Remark}

\AddToHook{env/defn/begin}{\crefalias{thm}{defn}}
\AddToHook{env/prop/begin}{\crefalias{thm}{prop}}
\AddToHook{env/asmp/begin}{\crefalias{thm}{asmp}}
\AddToHook{env/lem/begin}{\crefalias{thm}{lem}}
\AddToHook{env/rk/begin}{\crefalias{thm}{rk}}
\AddToHook{env/cor/begin}{\crefalias{thm}{cor}}
\AddToHook{env/appendices/begin}{\crefalias{section}{appendix}}
\AddToHook{env/appendices/begin}{\crefalias{subsection}{appendix}}

\crefname{defn}{Definition}{Definitions}
\crefname{thm}{Theorem}{Theorems}
\crefname{asmp}{Assumption}{Assumptions}
\crefname{cor}{Corollary}{Corollaries}
\crefname{lem}{Lemma}{Lemmas}
\crefname{prop}{Proposition}{Propositions}
\crefname{rk}{Remark}{Remarks}

\allowdisplaybreaks
\DeclareMathOperator{\divr}{div}
\DeclareMathOperator{\vol}{vol}
\DeclareMathOperator{\dom}{dom}
\DeclareMathOperator{\spn}{span}
\DeclareMathOperator{\supp}{supp}
\DeclareMathOperator{\Id}{Id}
\DeclareMathOperator{\alg}{Alg}
\DeclareMathOperator{\Adj}{Ad}
\DeclareMathOperator{\adj}{ad}
\DeclareMathOperator{\Aut}{Aut}
\DeclareMathOperator{\proj}{proj}
\DeclareMathOperator{\ran}{ran}
\DeclareMathOperator{\degr}{deg}

\newcommand{\vertiii}[1]{{\left\vert\kern-0.25ex\left\vert\kern-0.25ex\left\vert #1 \right\vert\kern-0.25ex\right\vert\kern-0.25ex\right\vert}}

\title[Koopman--von Neumann Dynamics and Noncommutative Geometry]{Measure-free Koopman--von Neumann Dynamics and Noncommutative Geometry}
\author[Giannakis]{Dimitrios Giannakis}
\address{Department of Mathematics, Dartmouth College, Hanover, NH 03755, USA.}
\email{dimitrios.giannakis@dartmouth.edu}
\author[Montgomery]{Michael Montgomery}
\address{Department of Mathematics, Dartmouth College, Hanover, NH 03755, USA.}
\email{michael.r.montgomery@dartmouth.edu}

\begin{document}

\maketitle

\begin{abstract}
	The Koopman--von Neumann formulation of classical statistical dynamics maps the isometric evolution of probability densities in $L^1$ under the Liouville equation to a unitary evolution of quantum mechanical wavefunctions in $L^2$ generated by the antisymmetric part of the Liouville operator.
	This approach enables the use of Hilbert space techniques to model the statistical evolution of observables.
	However, being an $L^2$ method, it is not suitable for representing pointwise evolution along dynamical trajectories.
	Here, we propose a Koopman--von Neumann framework that replaces the $L^p$ spaces associated with a volume measure on the state space manifold, $X$, by a reproducing kernel Hilbert space (RKHS), $\mathcal H$, of continuous functions on $X$ that satisfy joint analyticity conditions with respect to the generator (vector field), $V$, of the dynamics and its RKHS adjoint.
	Our scheme is based on two amplification steps:
	First, a dilation of the antisymmetric part of $V$ to an essentially skew-adjoint operator, $L$, on the tensor product Hilbert space $\mathfrak H = \ell^2(\mathbb N) \otimes \mathcal H$.
	Second, a dilation of $L$ to an essentially skew-adjoint, free derivation, $\mathcal D$, acting on a weighted symmetric Fock space $\mathfrak F$ generated by $\mathfrak H$, with coalgebra structure.
	Under the condition that $\mathcal H$ is a finitely generated reproducing kernel Hilbert algebra, we show that the unitary evolution generated by $\mathcal D$ consistently recovers the (generally, non-unitary) Koopman evolution of observables in a dense subalgebra of $\mathcal H$ generated by jointly analytic vectors, over a time interval that is uniformly bounded away from zero.
	Upon iteration, this scheme consistently recovers the Koopman evolution in the uniform norm over any finite time interval.
	In a subsequent part of this work, we build a family of weak spectral triples $(\mathcal A_n, \mathfrak F, -i \mathcal D)$, $n \in \mathbb N$, wherein  $\mathcal A_n$ are non-abelian $*$-subalgebras of $B(\mathfrak F)$ generated by creation and annihilation operators and $-i \mathcal D$ plays the role of a Dirac operator (generally, with noncompact resolvent) inducing extended pseudometrics on the state spaces of $\mathcal A_n$.
	In contrast to standard canonical commutation relation algebras, the algebras $\mathcal A_n$ are generated by bounded operators.
	We describe how the evolution of quantum observables in $\mathcal A_n$ consistently represents the pointwise evolution of observables in $\mathcal H$.
	We also characterize aspects of the quantum pseudometrics on the state space manifold $X$ induced by $(\mathcal A_n, \mathfrak F, -i\mathcal D)$.
	Finally, we make connections with the theory of kernel mean embeddings of Borel probability measures.
\end{abstract}

\section{Introduction}
\label{sec:introduction}

Consider a dynamical flow $\Phi^t\colon X \to X$, $t \in \mathbb R$, on a smooth, orientable manifold $X$, generated by a smooth vector field $\vec V\colon X \to T X$.
It is well-known that the evolution $t \mapsto p_t$ of smooth probability densities $p_t$ under $\Phi^t$ is governed by the Liouville equation,
\begin{equation}
	\label{eq:liouville}
	\partial_t p_t = \mathcal L p_t \equiv - \divr(p_t \vec V).
\end{equation}
Making the change of variables $\psi_t = \sqrt{p_t}$, and assuming that $p_t$ is strictly positive, it follows from basic rules of calculus that $\psi_t$ satisfies the evolution equation
\begin{equation}
	\label{eq:kvn}
	\partial_t \psi_t = A \psi_t, \quad A \psi_t = - \frac{\vec V \cdot \nabla \psi_t + \divr(\psi_t\vec V)}{2},
\end{equation}
where $A$ is an antisymmetric operator with dense domain $C^\infty_c(X) \subset L^2(X)$,
\begin{equation*}
	\langle f, A g \rangle = - \langle A f, g \rangle, \quad \langle f, g \rangle := \int_X \bar f g \, d\vol, \quad f, g \in C^\infty_c(X).
\end{equation*}
Under appropriate conditions (e.g., \cite{StenglEtAl24}), $A$ extends to a skew-adjoint operator, $\tilde{A}$, on $L^2(X)$, leading to the unitary evolution
\begin{equation*}
	\psi_t = e^{t \tilde{A}} \psi_0.
\end{equation*}

The construction outlined above is a basic example of methods referred to interchangeably as Koopman--von Neumann or Hilbert space methods for classical statistical dynamics.
In essence, the isometric evolution of the probability density function $p_t$ in $L^1(X)$ is mapped into an equivalent unitary evolution of a wavefunction $\psi_t$ in $L^2(X)$, thus providing a mathematical link between classical and quantum dynamics.
Historically, this approach can be traced to work of Sch\"onberg \cite{Schonberg52,Schonberg53} from the 1950s and later work by  Loinger \cite{Loinger62}, Della Riccia and Wiener \cite{DellaRiciaWiener66}, and Sudarshan \cite{Sudarshan76}, among other authors.
Inspired by the work of Koopman and von Neumann \cite{Koopman31,KoopmanVonNeumann32} from the 1930s on the unitary evolution of observables under measure-preserving dynamics, Mauro \cite{Mauro02} studied complex-valued wavefunctions analogous to $\psi_t$ above, which he called Koopman--von Neumann waves.

Mauro's constructions and results have stronger connections with the line of work initiated by Sch\"onberg rather than the work of Koopman and von Neumann (which was seminal in its own way in the foundation of modern ergodic theory), but the term Koopman--von Neumann approach is now commonly used in reference to methods that use unitarily evolving quantum wavefunctions to capture the non-unitary evolution of classical probability densities \cite{Barandes25}.
In recent years, variants of this framework have been extended and applied in diverse areas, including statistical mechanics, quantum simulation, plasma dynamics, kinetic theory, and semiclassical dynamics; e.g., \cite{BondarEtAl12,BondarEtAl19,Joseph20,LinEtAl22,JosephEtAl23,StenglEtAl24,NovikauJoseph25}.
A relativistic version of the Koopman--von Neumann method was developed in \cite{Mezic23}.

\subsection{Motivation: Reproducing kernel Hilbert spaces for measure-free Koopman--von Neumann dynamics}
\label{sec:motivation}

A common aspect of many Koopman--von Neumann techniques (in fact, all such methods known to us) is that the wavefunction $\psi_t$ lies in an $L^2$ space associated with a volume measure on the state space manifold $X$.
A special property of the $L^2(X)$ inner product in that case is that the symmetric part, $S$, of the Liouville operator $\mathcal L$ is a multiplication operator by the negative divergence of the vector field,
\begin{equation*}
	S f = \frac{\vec V \cdot \nabla f - \divr(f \vec V)}{2} = - (\divr \vec V) f,
\end{equation*}
with $\mathcal L = S + A$ and $\langle f, S g \rangle = \langle S f, g \rangle$ on $C^\infty_c(X)$.
This implies $\mathcal L(\psi_t^2) = 2 \psi_t A \psi_t$, giving~\eqref{eq:kvn}.
However, $L^2$ spaces induced by volume measures on manifolds are not appropriate for representing the evolution of observables along individual dynamical trajectories (i.e., $x \mapsto f(\Phi^t(x))$), as they lack a notion of continuous pointwise evaluation of their elements.
Relatedly, the Koopman---von Neumann evolution~\eqref{eq:kvn} can only capture the evolution of probability measures that are absolutely continuous with respect to a fixed volume measure, but a general flow $\Phi^t$ does not have a canonical measure associated to it.

In ergodic theory, it is common to work in $L^p$ spaces associated with invariant measures, and for certain classes of dynamical systems (typically exhibiting some form of hyperbolicity) there is a distinguished choice such as a Sinai--Ruelle--Bowen (SRB) measure \cite{Young02,Blank17}.
However, invariant measures are oftentimes supported on null subsets of $X$ with respect to the volume measure (e.g., attractors under dissipative dynamics) and the associated $L^p$ spaces cannot capture the transient evolution of observables off of the invariant measure's support.
In Anosov dynamical systems, a powerful approach has been to study transfer operators in anisotropic Banach spaces of distributions adapted to the stable and unstable directions of the dynamics \cite{BlankEtAl02,GouezelLiverani06,BaladiTsujii07,GiuliettiEtAl13}, extending earlier work on expanding maps \cite{Ruelle89,GundlachLatushkin03} where the transfer operator acts as a smoothing operator on H\"older or Sobolev spaces.
Spectral analysis of transfer operators in such spaces yields important information such as Pollicott--Ruelle resonances that characterize mixing and rates of decay of correlations.
However, these spaces lack the requisite inner product structure for Hilbert space formulations of dynamics.

These considerations motivate the development of Koopman--von Neumann-type techniques in Hilbert spaces that do not depend on a choice of measure and allow reconstruction of observables along individual dynamical trajectories.
With these objectives in mind, reproducing kernel Hilbert spaces (RKHSs) are natural candidates for the Hilbert space:
An RKHS can be defined on an arbitrary set $X$ without invoking a measurable structure and, by construction, every such space consists of everywhere-defined functions with associated continuous pointwise evaluation functionals \cite{Aronszajn50,PaulsenRaghupathi16}.

The action of dynamical systems on RKHSs has a history of study using composition operators or weighted composition operators; e.g., \cite{Cowen83,ContrerasEtAl01,Le17,Zorboska18,IkedaEtAl22,IkedaEtAl22b,BoulleEtAl25,HartzTornes25,KohneEtAl25}.
For an RKHS $\mathcal H$ of complex-valued functions on $X$ with reproducing kernel $k\colon X \times X \to \mathbb C$ and a discrete-time dynamical map $\Phi \colon X \to X$, one way of approaching these constructions (focusing on the unweighted case for illustration) is to define a linear map $P \colon \tilde{\mathcal H} \to \mathcal H$ on the dense subspace $\tilde{\mathcal H} = \{ k_x: x \in X\} \subset \mathcal H$ spanned by the kernel sections $k_x := k(x, \cdot)$ as
\begin{equation*}
	P k_x = k_{\Phi(x)}.
\end{equation*}
Note that well-definition of $P$ generally requires that the kernel sections $k_x$ are linearly independent.
The adjoint of $P$ is then a composition operator, $P^* f = f \circ \Phi$, acting on functions $f$ in a (not necessarily dense) domain in $\mathcal H$.

Properties of $P$ such as closability, boundedness, normality, and existence of unitary extensions, are the outcome of an intricate interplay between the dynamics $\Phi$ and the reproducing kernel $k$.
In general, even closability (required for the composition operator $P^*$ to be densely defined) is non-trivial.
In a number of examples, boundedness of $P$ places severe restrictions on the dynamics; for example, when $k$ is a translation-invariant kernel on $\mathbb R^d$ induced by an analytic function, $P$ is bounded if and only if $\Phi$ is an affine map \cite{IkedaEtAl22b}.

For a continuous-time flow $\Phi^t\colon X \to X$ generated by a vector field $\vec V\colon X \to TX$ such that $V := \vec V \cdot \nabla$ and $V^*$ both have dense domains in $\mathcal H$, with $\dom(V^*)$ containing $\tilde{\mathcal H}$, we have
\begin{equation}
	\label{eq:rkhs_liouville}
	\partial_t k_{\Phi^t(x)} = V^* k_{\Phi^t(x)},
\end{equation}
with the left-hand side given by the weak limit $\langle \partial_t k_{\Phi^t(x)}, f \rangle_{\mathcal H} = \lim_{t\to 0} \langle k_{\Phi^t(x)} - k_x, f\rangle_{\mathcal H} / t$, $f \in \dom(V) $.
The above can be intuitively interpreted as an RKHS analog of the Liouville equation~\eqref{eq:liouville}.
However, the existence of an associated (semi)group of bounded operators on $\mathcal H$, analogous to the Perron-Frobenius semigroup of isometries on $L^1(X)$, is not guaranteed.

Even if~\eqref{eq:rkhs_liouville} has an associated semigroup of bounded operators, the Koopman--von Neumann approach as formulated in $L^p$ spaces does not directly translate to RKHSs:
First, for a general RKHS $\mathcal H$, the kernel sections $k_x$ need not have square roots in the same space.
Second, the symmetric part $(V + V^*)/2$ with respect to the RKHS inner product is not, in general, a multiplication operator.
This means that, even if well-defined as an element of $\mathcal H$, $\sqrt{k_{\Phi^t(x)}}$ does not evolve unitarily under a Koopman--von Neumann equation analogous to~\eqref{eq:kvn}.

\subsection{Our contributions}
\label{sec:contributions}

We propose a measure-free Koopman--von Neumann-type representation of continuous-time dynamical systems on manifolds.
Our approach is based on an RKHS, $\mathcal H$, of continuous functions on the state space manifold $X$, with a dense subspace, $\mathcal H_\infty$, on which the generator $V$ and its adjoint are jointly analytic.
Under these conditions, we build a dilation, $L$, of the antisymmetric part $(V - V^*)/2$ of the generator acting on the amplification $\mathfrak H = \ell^2(\mathbb N) \otimes \mathcal H$.
In \cref{thm:amplification}, we show that $L$ is an essentially skew-adjoint operator that contains sufficient information to reconstruct the Koopman evolution $t \mapsto f \circ \Phi^t$ of observables $f \in \mathcal H_\infty$under the dynamical flow $\Phi^t$ on $X$, for times $t$ in an $f$-dependent time interval, $\lvert t \rvert < t_f$, with $t_f > 0$ determined from the joint radius of convergence of the series
\begin{equation*}
	\sum_{n=0}^\infty \frac{t^n}{n!} V^n f, \qquad \sum_{n=0}^\infty \frac{t^n}{n!} (V^*)^n f,
\end{equation*}
in the norm topology of $\mathcal H$.

In more detail, the unitary evolution group generated by $L$ consistently recovers the non-unitary evolution of vectors $f \in \mathcal H_\infty$ generated by $V$ in the first component of a finite linear combination of tensors,\begin{equation}
	\label{eq:amplification_embedding}
	E_n f := \sum_{i=1}^n F_i e_i \otimes f,
\end{equation}
in the limit of $n \to \infty$, where $e_i$ are the canonical orthonormal basis vectors of $\ell^2(\mathbb N)$, and $F_i$ the Fibonacci numbers.
Examples of dynamical systems meeting the required analyticity conditions are systems on compact Lie groups generated by vector fields with polynomial coefficients in the Fourier basis.
This result opens the possibility to analyze the pointwise evolution of observables in RKHSs (including the ``off-attractor'' transient evolution) using spectral techniques for unitary evolution groups, analogously to spectral analysis of transfer and Koopman operators on $L^2$ in the measure-preserving case \cite{Walters81,EisnerEtAl15}.

In general, the evolution interval $t_f$ from \cref{thm:amplification} may be arbitrarily small for general analytic vectors $f$.
To overcome this limitation, we perform a further dilation of $L$ to a free derivation acting on a weighted symmetric Fock space, $\mathfrak F = F_w(\mathfrak H)$, generated by $\mathfrak H$.
For an appropriate choice of subconvolutive weights, $w\colon \mathbb N \to \mathbb R_{>0}$ (e.g., \cite{Feichtinger79,Grochenig07}), $\mathfrak F$ becomes a commutative Banach algebra with respect to the symmetric tensor product \cite{GiannakisEtAl26}.
By Gelfand duality, $\mathfrak F$ is isomorphic to a space of continuous complex-valued functions, $\widehat{\mathfrak F}$, on its character spectrum, $\sigma(\mathfrak F) \subset \mathfrak F^*$.
The function space $\widehat{\mathfrak F}$ has the structure of a ``reproducing kernel Hilbert algebra'' (RKHA) \cite{GiannakisMontgomery25}; that is, it is an RKHS with additional coalgebra structure, making it a commutative Banach algebra under the pointwise product of functions.
We lift $L$ to an essentially skew-adjoint operator, $\mathcal D$, on $\mathfrak F$, which acts as a free derivation with respect to the symmetric tensor product, $\mathcal D(\eta \vee \xi) = (\mathcal D \eta) \vee \xi + \eta \vee (\mathcal D \xi)$.
In \cref{thm:tensor_koopman}, we show that if $\mathcal H$ is a finitely generated unital RKHA, the unitary evolution generated by $\mathcal D$ reconstructs the Koopman evolution of observables in a dense subalgebra of $ \mathcal H$ for a time interval that is uniformly bounded away from 0.
As a corollary, iterating the evolutions obtained via this approach consistently approximates the Koopman evolution of observables with respect to the supremum norm over any finite time interval.

Another objective of this paper is to study the properties of the dynamical generator $\mathcal D$ from the perspective of noncommutative geometry \cite{Connes94,KhalkhaliMarcolli08}.
In Connes' noncommutative geometry \cite{Connes89}, a spectral triple, $(\mathbb A, \mathbb H, \mathbb D)$, consists of a unital $*$-algebra $\mathbb A \subseteq B(\mathbb H)$ of bounded operators acting on a Hilbert space, $\mathbb H$, and a self-adjoint operator, $\mathbb D$, with the properties that (i) the commutator $[\mathbb D, a]$ is a bounded operator on $\mathbb H$ for every $a \in \mathbb A$; (ii) $\mathbb D$ has compact resolvent.
The operator $\mathbb D$, called in this context Dirac operator, defines an extended pseudometric, $d\colon S(\mathbb A) \times S(\mathbb A) \to [0, \infty]$, on the state space $S(\mathbb A)$ of the $C^*$-algebra closure of $\mathbb A$,
\begin{equation*}
	d(\varphi, \psi) = \sup \{ \lvert \varphi(a) - \psi(a) \rvert: \text{$a \in \mathbb A$, $\lVert [\mathbb D, a] \rVert < 1$}\}.
\end{equation*}
The metric encodes analytical properties of the Dirac operator in the induced state space geometry of $S(\mathbb A)$.
Later on, Rieffel developed the theory of compact quantum metric spaces \cite{Rieffel99,Rieffel04}, where the commutator norm in the Connes metric is replaced by a Lipschitz seminorm that metrizes the weak$^*$ topology of $S(\mathbb A)$.
Sine the foundational works in \cite{Connes89,Rieffel99}, the theories of spectral triples and quantum metric spaces have been extended and generalized in many directions.
Examples include semifinite spectral triples \cite{CareyEtAl06a,CareyEtAl06b} and quantum locally compact metric spaces \cite{Latremoliere13}.

In our framework, the RKHA $\mathfrak F$ supports a family of finitely-generated $*$-algebras $\mathcal A_n \subset B(\mathfrak F)$, $n \in \mathbb N$, generated by creation and annihilation operators induced by vectors of the form $E_n f$, for $f$ in a finite-dimensional subspace $\mathcal S \subset \mathcal H$ and the Fibonacci embeddings $E_n$ from~\eqref{eq:amplification_embedding}.
The generators of $\mathcal A_n$ obey a weighted form of the canonical commutation relations (CCRs), and in contrast to generators of standard CCR algebras, they are bounded operators.
One of our main results is that $(\mathcal A_n, \mathfrak F, -i\overline{\mathcal D})$ defines a weak spectral triple, with the self-adjoint operator closure $-i\overline{\mathcal D}$ playing the role of a Dirac operator with noncompact resolvent.
Moreover, as $n \to \infty$, the unitary evolution generated by $\mathcal D$ reconstructs the Koopman evolution of observables in a dense subalgebra of $\mathcal H$ for the same uniform time interval as in \cref{thm:tensor_koopman}; see \cref{thm:state_evolution_convergence,thm:spectral_triple_uniform_convergence}.

The family of weak spectral triples $(\mathcal A_n, \mathfrak F, -i \mathcal D)$ has an associated embedding, $\Psi \colon X \to S(\mathcal A)$, of the state space manifold into the state space of a $C^*$-algebra $\mathcal A \subset B(\mathfrak F)$ that contains $\bigcup_{n \in \mathbb N}\mathcal A_n$ as a $*$-subalgebra.
Pulling back the corresponding pseudometric $d$ under $\Psi$ defines a pseudometric $d_X \colon X \times X \to [0, \infty]$ that depends on the Dirac operator $-i \mathcal D$, and thus the vector field $V$ of the dynamical system.
To characterize the behavior of $d_X$ we make two observations:
First, the embedding $\Psi$ induces a map $\widehat{\Psi} \colon \mathcal A \to C_b(X)$ that is an algebra homomorphism.
Second, for any $n \in \mathbb N$, $\widehat{\Psi}$ intertwines the Lie bracket $\adj_{\mathcal D} := [\mathcal D, \cdot]$ on $\mathcal A_n$ with the dynamical vector field $V$,
\begin{equation*}
	\widehat{\Psi} \circ \adj_{\mathcal D} = V \circ \widehat{\Psi}.
\end{equation*}
Using this property, we show that $d_X$ is bounded by the metric induced by a Lipschitz norm associated with $V$; see \cref{cor:quantum_metrics}.

In a final contribution, we extend the embedding $\Psi$ to the space of Borel probability measures on $X$, in a construction that is similar in spirit to the kernel mean embedding of probability measures in RKHS theory \cite{SriperumbudurEtAl10}*{and references therein}.
We describe how this extended embedding recovers the dynamical evolution of expectation values of observables in $\mathcal H$ with respect to arbitrary Borel probability measures on the state space manifold, including singular measures that cannot be represented by the standard Koopman--von Neumann method.

\subsection{Organization of the paper}

In \cref{sec:analyticity} we define the dynamical system under study, state our joint analyticity assumptions, and lay out basic results that stem from them.
This section also contains examples on compact and locally compact Lie groups that meet our basic assumptions.
In \cref{sec:first_amplification}, we describe our first amplification step utilizing the Fibonacci embedding from~\eqref{eq:amplification_embedding}.
This is followed by our second amplification to the weighted Fock space and the associated weak spectral triples, presented in \cref{sec:fock_space}.
In \cref{sec:state_space_embedding_and_metric}, we examine approaches for embedding the state space manifold $X$ into the state space, $S(\mathcal A)$, of the $C^*$-algebra $\mathcal A$, including the embedding $\Psi$ that induces a $C^*$-algebra homomorphism into $C_b(X)$.
We also study the properties of the induced noncommutative metric $d_X$ on state space associated with $\Psi$.
In \cref{sec:discussion} we give an overview of kernel mean embedding of Borel probability measures and our related approach for reconstructing the statistical evolution of observables through embeddings into $S(\mathcal A)$.
We conclude with a discussion drawing analogies between our measure-free approach and the Koopman--von Neumann method in $L^p$ spaces.

The paper contains four appendices:
\Cref{app:rkha} collects various basic properties of RKHAs.
\Cref{app:Fock_RKHA} gives an overview of the construction of the weighted Fock space $\mathfrak F$, summarizing results from \cite{GiannakisEtAl26}.
In \Cref{app:alt_first_amplification} we present an alternative first amplification approach to the Fibonacci scheme of \cref{sec:first_amplification}.
In particular, \Cref{thm:alt_amplification} in that appendix is analogous to \cref{thm:amplification}, and all results in the paper that depend on that theorem could be formulated using the amplification approach in \Cref{app:alt_first_amplification}.
For the reasons stated in \cref{rk:alt_first_amplification}, the Fibonacci amplification scheme is our method of choice, but we include the alternative approach in \cref{app:alt_first_amplification} as we believe it is of independent interest.
\Cref{app:ideal} describes the construction of a commutative quotient $C^*$-algebra of $\mathcal A$ that provides an alternative construction of $\Psi$ as an embedding into the state space of the quotient algebra.

All Hilbert spaces in this paper will be over the complex numbers, and all inner products will be taken conjugate-linear in the first argument.
For $n \in \mathbb N$, $[n]$ shall denote the set $\{1, \ldots, n \}$.

\section{Dynamical system and joint analyticity}
\label{sec:analyticity}

We consider a dynamical flow, $\Phi^t \colon X \to X$, $t \in \mathbb R$, on a $C^\infty$ manifold $X$, generated by a $C^\infty$ vector field $\vec V \colon X \to T X$.
The vector field induces a derivation $V \colon C^\infty(X) \to C^\infty(X)$ by $V f = \vec V \cdot \nabla f$.
We fix an RKHS $\mathcal H \subset C^\infty(X)$ with reproducing kernel $k \colon X \times X \to \mathbb C$ and associated dense subspace $\tilde{\mathcal H} = \spn\{k_x : x \in X\} \subset \mathcal H$ spanned by the kernel sections $k_x = k(x,\cdot)$.
As an unbounded operator from $\mathcal H$ to itself, define the domain of $V$ to be
\begin{equation*}
	\dom(V) = \{f \in \mathcal H : Vf \in \mathcal H\}.
\end{equation*}

Throughout this paper, we will make the standing assumption that the reproducing kernel $k$ is strictly positive-definite; that is, for every finite set $ \{ x_i \}_{i=1}^n \subset X$ consisting of distinct points $x_i$ and nonzero numbers $c_1, \ldots, c_n \in \mathbb C$, we have $\sum_{i,j=1}^n \overline{c_i} k(x_i, x_j) c_j > 0$.
This is equivalent to linear independence of any finite set $\{ k_{x_i} \}_{i=1}^n \subseteq \mathcal H$ whenever $x_1, \ldots, x_n \in X$ are distinct.
Since the manifold $X$ is uncountable as a set, linear independence of the kernel sections means that $\mathcal H$ is necessarily infinite-dimensional.
Moreover, the canonical feature map $\kappa\colon X \to \mathcal H$,
\begin{equation}
	\label{eq:canonical_feature_map}
	\kappa(x) = k_x,
\end{equation}
is injective.

In many of the examples that we consider the reproducing kernel $k$ is, additionally, universal \cite{MicchelliEtAl06}.
This means that for any compact subset $Y \subseteq X$, $\mathcal H\rvert_Y$ is a dense subspace of $C(Y)$.
In general, universality and strict positive-definiteness  are independent notions, but they are equivalent in particular cases such as radial kernels on $\mathbb R^d$ \cite{SriperumbudurEtAl11}.
While universality is certainly a desirable property to have in applications (see, e.g., \cref{rk:uniform_approx} ahead) it is not necessary to carry out our constructions.
We refer the reader to \cite{PaulsenRaghupathi16,FerreiraMenegatto13,SteinwartChristmann08} for detailed expositions of RKHS theory.

\subsection{Transfer and Koopman operators on reproducing kernel Hilbert spaces}

For each $t \in \mathbb R$ we define, similarly to \cref{sec:motivation}, the densely defined linear map $P^t\colon \tilde{\mathcal H} \to \mathcal H$ by linear extension of
\begin{equation}
	\label{eq:transfer_op}
	P^t k_x = k_{\Phi^t(x)}.
\end{equation}
Note that linear independence of the kernel sections $k_x$ is important for the well-definition of $P^t$ through~\eqref{eq:transfer_op}.
By construction, the subspace $\tilde{\mathcal H}$ is invariant under $P^t$, so $ \{ P^t\}_{t\in \mathbb R}$ forms a group of densely defined operators under composition.
While $P^t$ need not be closable, its adjoint is a Koopman (composition) operator by the flow.
Being the predual of a composition operator, we thus think of $P^t$ as an RKHS version of the transfer operator.
Variants of the following result are well-known in the literature; e.g., \cite{IkedaEtAl22,IkedaEtAl22b,BoulleEtAl25}.

\begin{lem}
	For $t \in \mathbb R$, the adjoint $U_t := (P^t)^* \colon \dom(U_t) \to \mathcal H$ has domain $\dom(U_t) = \{f \in \mathcal H: f \circ \Phi^t\ \in \mathcal H\} \subseteq \mathcal H$.
	Furthermore, $U_t$ acts as a composition operator on its domain, $U_t f = f \circ \Phi^t$.
\end{lem}

\begin{proof}
	Fix $t \in \mathbb R$ and suppose $f \circ \Phi^t \in \mathcal H$ for $f \in \mathcal H$.
	Then, for every $x \in X$,
	\begin{displaymath}
		\langle P^t k_x, f \rangle_{\mathcal H} = \langle k_{\Phi^t(x)}, f \rangle_{\mathcal H} = f(\Phi^t(x)) = (f \circ \Phi^t)(x) = \langle k_x, f \circ \Phi^t\rangle_{\mathcal H}.
	\end{displaymath}
	Now for every $g \in \tilde{\mathcal H}$ we have $g = \sum_{i=1}^n c_i k_{x_i}$ for some $n \in \mathbb N$, $c_i \in \mathbb C$, and $x_i \in X$, so
	\begin{equation*}
		\lvert\langle f, P^t g \rangle_{\mathcal H} \rvert = \left\lvert\sum_{i=1}^n \langle f, c_i P^t k_{x_i} \rangle_{\mathcal H}\right\rvert = \left\lvert\sum_{i=1}^n \langle f \circ \Phi^t, c_i k_{x_i}\rangle_{\mathcal H} \right\rvert = \lvert\langle f \circ \Phi^t, g \rangle_{\mathcal H} \rvert \leq  \lVert f \circ \Phi^t \rVert_{\mathcal H}\lVert g \rVert_{\mathcal H}.
	\end{equation*}
	This implies that $g \mapsto \langle f, P^t g\rangle_{\mathcal H}$ is a bounded linear functional on $\tilde{\mathcal H}$ with operator norm bounded by $\lVert f \circ \Phi^t \rVert_{\mathcal H}$, and thus that $f \in \dom(U_t)$.

	Now suppose that $f \in \dom(U_t)$.
	Then $\langle k_x, U_t f \rangle_{\mathcal H}$ is well-defined for every $x \in X$ and
	\begin{equation*}
		(U_t f)(x) = \langle k_x, U_t f\rangle_{\mathcal H} = \langle P^t k_x, f \rangle_{\mathcal H} = \langle k_{\Phi^t(x)}, f \rangle_{\mathcal H} = f(\Phi^t(x)).
	\end{equation*}
	Thus, we have $(U_t f)(x) = (f \circ \Phi^t)(x)$ which implies that $f \circ \Phi^t \in \dom(U_t)$ and that $U_t$ acts as a composition operator on its domain.
\end{proof}

Note that, unlike $P^t$, the composition operators $U_t$ need not be densely defined and, in general, do not form a group.
To make further progress, we will require that $V$ and its adjoint, $V^*$, on $\mathcal H$ satisfy a joint analyticity condition on a common, dense, invariant subspace of their domains.

\begin{asmp}
	\label{asmp:main}
	The following hold.
	\begin{enumerate}[(a)]
		\item $V$ is continuous as a map from $\mathcal H$ to $C(X)$ with the compact open topology.
		\item There exists a dense subspace $\mathcal H_\infty \subseteq \dom(V) \cap \dom(V^*)$ such that $V(\mathcal H_\infty) \subseteq \mathcal H_\infty$ and $V^*(\mathcal H_\infty) \subseteq \mathcal H_\infty$.
		\item $V$ and $V^*$ are jointly analytic on $\mathcal H_\infty$.
		      That is, for $V_1 = V$ and $V_2 = V^*$, and for every $f \in \mathcal H_\infty$ there exists a strictly positive radius of convergence $|z| < R_f$ for the series
		      $$ \sum_{n=0}^\infty \frac{z^n}{n!} \sum_{\epsilon \colon [n] \to [2]} \left\lVert V_{\epsilon(1)} \cdots V_{\epsilon(n)} f \right\rVert_{\mathcal H}.$$
	\end{enumerate}
\end{asmp}

\begin{rk}
	\Cref{asmp:main}(a) implies that $V$ is a closed operator, and the subspace $\tilde{\mathcal H} \subset \mathcal H$ spanned by the kernel sections lies in $\dom(V^*)$.
	The Liouville equation~\eqref{eq:rkhs_liouville} then follows as a consequence of this fact.
\end{rk}
\begin{proof}
	By \cref{asmp:main}(a), for every compact set $Y \subseteq X$ there is a constant $C_Y$ such that for every $f \in \mathcal H$
	$$\lVert Vf|_Y \rVert_\infty \leq C_Y \lVert f \rVert_\mathcal H.$$
	Hence, if $\{(f_i, Vf_i)\}_i \subset \mathcal H \times \mathcal H$ is a sequence in the graph of $V$ converging to $(f,g) \in \mathcal H \times \mathcal H$, then for all compact sets $Y \subseteq X$
	$$\lim_{i \to \infty} \lVert Vf_i|_Y -Vf|_Y\rVert_\infty \leq \lim_{i \to \infty} C_Y \lVert f_i - f \rVert_\mathcal H = 0.$$
	Since $\mathcal H \subset C(X)$, the kernel embedding map is continuous in the weak topology of $\mathcal H$.
	The uniform boundedness principle implies that $\sup_{x\in Y} \lVert k_x\rVert \leq C$, and so
	$$\lim_{i \to \infty} \lVert Vf_i|_Y - g|_Y\rVert_\infty \leq \lim_{i \to \infty} C \lVert Vf_i - g\rVert_\mathcal H = 0$$
	(i.e., $Vf_i$ converges to $g$ in the compact open topology).
	This means $Vf = g \in \mathcal H$ and $f \in \dom(V)$.

	To see why $k_x \in \dom(V^*)$, set $Y = \{x\}$ and observe that
	$$\lvert \langle k_x, Vf \rangle_{\mathcal H} \rvert =\lvert Vf(x)\rvert \leq C_Y \lVert f \rVert_{\mathcal H} \quad \text{for } f\in \dom(V),$$
	making $\dom(V) \ni f \mapsto \langle k_x, Vf \rangle_{\mathcal H}$ bounded.
\end{proof}

Since the unbounded operator $V$ is not necessarily skew-adjoint, and $U_t$ need not be strongly continuous (see \cref{sec:motivation}), many standard theorems on semigroups and Koopman operators do not apply.
The following proposition shows that exponentiation of $V$ recovers the Koopman operator on analytic vectors in $\mathcal H_\infty$.

\begin{prop}
	\label{prop:taylor_koopman}
	Suppose \cref{asmp:main} holds so that for any observable $f \in \mathcal H_\infty$ there exists a strictly positive radius of convergence $t_f$ of $\sum_{n=0}^\infty \frac{t^n}{n!} V^nf$.
	Then, for all evolution times $\lvert t \rvert < t_f$, we have
	\begin{equation*}
		U_t f = e^{t V} f \equiv \sum_{n=0}^\infty \frac{t^n}{n!} V^n f.
	\end{equation*}
\end{prop}

\begin{proof}
	Fix a point $x \in X$.
	Then, for $t \in (-t_f,t_f)$ and by absolute convergence of $e^{Vt}f$,
	$$\frac{d}{dt} e^{Vt}f(x) = \sum_{n=0}^\infty \frac{t^n}{n!} V^{n+1} f(x) = e^{Vt}Vf(x).$$
	Since $V \colon \mathcal H \to C(X)$ is continuous in the compact open topology, for every compact set $Y \subseteq X$ there is a constant $C_Y$ such that
	$$\lVert Vf|_Y \rVert_{\infty} \leq C_Y \lVert f \rVert_\mathcal H.$$
	Again by absolute convergence of $e^{Vt}f$ in $\mathcal H$,
	\begin{align*}
		\lim_{N\to \infty} \left\lVert \left.\left( \sum_{n=0}^N \frac{t^n}{n!}V^n (Vf)\right)\right|_Y- \left.
		V e^{Vt}f \right|_Y \right\rVert_\infty & =
		\lim_{N\to \infty} \left\lVert \left.
		V \left( \sum_{n=N+1}^\infty \frac{t^n}{n!}V^n f\right)\right|_Y \right\rVert_\infty                                                                      \\
		                                        & \leq \lim_{N \to \infty} C_Y \left\lVert \sum_{n=N+1}^\infty \frac{t^n}{n!}V^{n} f \right\rVert_\mathcal H = 0.
	\end{align*}
	This means $e^{Vt}f$ is in the domain of $V$ and $Ve^{Vt}f = e^{Vt}Vf \in \mathcal H$.

	Next, define the following functions, $f_t = e^{Vt}f$ and $g(t) = f_t(\Phi^{-t}(x))$.
	Observe that $g(0) = f(x)$ and
	\begin{align*}
		\frac{d}{dt} g(t) & = \left.\frac{df_\tau}{d\tau}\right|_{\tau=t}\left(\Phi^{-t}(x)\right) + df_t \left( \frac{d}{dt} \Phi^{-t}(x)\right) \\
		                  & = (Ve^{Vt}f)(\Phi^{-t}(x)) + df_t( -V_{\Phi^{-t}(x)} )                                                                \\
		                  & = (V f_t)(\Phi^{-t}(x)) - V_{\Phi^{-t}(x)} f_t                                                                        \\
		                  & =V_{\Phi^{-t}(x)} f_t - V_{\Phi^{-t}(x)} f_t                                                                          \\
		                  & =0.
	\end{align*}
	Since $g$ is a smooth function, $g(t) = f(x)$ for all $t \in (-t_f,t_f)$.
	Unpacking the definition of $g(t)$ we find
	$f(x) = (e^{Vt}f)(\Phi^{-t}(x))$ and so $f(\Phi^t(x)) = e^{Vt}f(x)$.
\end{proof}

\subsection{Transfer and Koopman operators on reproducing kernel Hilbert algebras}

With only \cref{asmp:main} we cannot guarantee that the Koopman operator $U_t$ is densely defined for any fixed time $t \in \mathbb R$.
More importantly, we wish for $e^{Vt}$ to uniquely determine $U_t$.
To address this question we will require that the RKHS $\mathcal H$ additionally possesses RKHA structure.
We begin by recalling the following definition from \cite{GiannakisMontgomery25}.

\begin{defn}
	\label{def:rkha}
	An RKHS $\mathcal H$ on a set $X$ with reproducing kernel $k\colon X \times X \to \mathbb C$ is said to be an RKHA if the map $\Delta\colon k_x \mapsto k_x \otimes k_x$  extends to a bounded linear map $\Delta \colon \mathcal H \to \mathcal H \otimes \mathcal H$.
\end{defn}

Since
\begin{equation*}
	\langle k_x, \Delta^*(f \otimes g)\rangle_{\mathcal H} = \langle \Delta k_x, f \otimes g \rangle_{\mathcal H \otimes \mathcal H} = \langle k_x \otimes k_x, f \otimes g\rangle_{\mathcal H \otimes \mathcal H} = \langle k_x, f\rangle_{\mathcal H} \langle k_x, g\rangle_{\mathcal H} = f(x)g(x),
\end{equation*}
it follows that $\Delta^*$ is a bounded linear map that implements pointwise multiplication.
Similarly, the map $\Delta_n \colon \mathcal H \to \mathcal H^{\vee n} \subset \mathcal H^{\otimes n}$ given by the unique linear extension of $\Delta_n k_x = k_x^{\vee n}$ is a bounded operator and hence $\Delta_n^* \colon \mathcal H^{\otimes n} \to \mathcal H$ implements $n$-fold pointwise multiplication and is bounded.
For an explicit operator norm bound observe that
$$\Delta_2 = \Delta, \quad \Delta_n = (\Delta \otimes \Id_{\mathcal H^{\otimes (n-1)}})\circ \Delta_{n-1}.$$

By virtue of these properties, every RKHA $\mathcal H$ is simultaneously a Hilbert function space and commutative algebra with respect to pointwise function multiplication.
Moreover, since
\begin{equation}
	\label{eq:rkha_kernel_bound}
	\lVert k_x \rVert_{\mathcal H} = \frac{\lVert k_x \otimes k_x \rVert_{\mathcal H \otimes \mathcal H}}{\lVert k_x \rVert_{\mathcal H}} = \frac{\lVert \Delta k_x \rVert_{\mathcal H}}{\lVert k_x \rVert_{\mathcal H}} \leq \lVert \Delta \rVert,
\end{equation}
and
\begin{equation}
	\label{eq:rkha_boundedness}
	\lvert f(x) \rvert = \lvert \langle k_x, f \rangle_{\mathcal H} \rvert \leq \lVert k_x \rVert_{\mathcal H} \lVert f \rVert_{\mathcal H},
\end{equation}
it follows that an RKHA contains only bounded functions.
Further discussion on RKHAs is included in \cref{app:rkha}.
Concrete examples will be given in \cref{sec:poly_examples}.
For a subset $\mathcal S$ of an RKHA $\mathcal H$, $\alg(\mathcal S) \subseteq \mathcal H$ will denote the smallest subalgebra of $\mathcal H$ containing $\mathcal S$.

\begin{asmp}
	\label{asmp:alg}
	The RKHS $\mathcal H$ is an RKHA.
	Moreover, there is a finite-dimensional subspace $\mathcal S \subset \mathcal H_\infty$ such that $\mathcal H = \overline{\alg(\mathcal S)}$.
\end{asmp}

\Cref{asmp:alg} is essentially a finite generation requirement for the RKHA $\mathcal H$.
This property will allow us to prove dense definition of $U_t$ for sufficiently small $\lvert t \rvert > 0$.

\begin{prop}\label{prop:algdomain}
	Under \cref{asmp:main,asmp:alg}, if $e^{Vt}f$ converges absolutely for $f \in \mathcal S$, then $e^{Vt}f$ converges absolutely for $f \in \alg(\mathcal S)$.
	Furthermore, $e^{Vt} \colon \alg(\mathcal S) \to \mathcal H$ is an algebra homomorphism.
\end{prop}

\begin{proof}
	It suffices to show that if $e^{Vt}f$ and $e^{Vt}g$ converge absolutely, then so does $e^{Vt}(f g)$.
	We may verify this directly using the derivation property of $V$ and boundedness of multiplication in RKHAs:
	\begin{align*}
		\sum_{n=0}^\infty \frac{|t|^n}{n!} \lVert V^n(f g) \rVert_{\mathcal H} & = \sum_{n=0}^\infty \frac{|t|^n}{n!} \left\lVert\sum_{k=0}^n \frac{n!}{k! (n-k)!} (V^k f)  (V^{n-k}g) \right\rVert_{\mathcal H}                                                     \\
		\\ &\leq C \sum_{n=0}^\infty \sum_{k=0}^n \frac{ |t|^n}{k! (n-k)!} \lVert V^k f\rVert_{\mathcal H} \lVert V^{n-k}g\rVert_{\mathcal H} \\
		                                                                       & = C \left( \sum_{n=0}^\infty \frac{|t|^n}{n!} \lVert V^n f \rVert_{\mathcal H} \right) \left( \sum_{m=0}^\infty \frac{|t|^m}{m!} \lVert V^m g \rVert_{\mathcal H} \right) < \infty.
	\end{align*}
	If we perform the same computation without norms or absolute values, we see that $e^{Vt}(fg) = (e^{Vt}f) (e^{Vt}g)$.
	Hence, $e^{Vt}$ is an algebra homomorphism.
\end{proof}

Now under \cref{asmp:alg} $e^{Vt}$ is densely defined for some time interval $t \in (-T,T)$, $T>0$, and $U_t|_{\alg(\mathcal S)} = e^{Vt}|_{\alg(\mathcal S)}$ since $U_t$ is also an algebra homomorphism.
The evolution of kernel sections $k_x$, and thus, by injectivity of the feature map~\eqref{eq:canonical_feature_map}, the pointwise dynamics on $X$, can then be recovered from $e^{Vt}$ by taking the adjoint.
Indeed, since $U_t$ is densely defined, $P^t$ is closable and
$$\text{span}\{k_x: x\in X\} \equiv \tilde{\mathcal H} \equiv \dom(P^t) \subseteq \text{dom}(U_t^*) = \text{dom}(\overline{P^t}) \subseteq \text{dom}\left( (e^{Vt})^*\right),$$
giving
\begin{equation*}
	k_{\Phi^t(x)} = P^t k_x = (e^{Vt})^* k_x.
\end{equation*}
Note, however, that in general $(e^{Vt})^* k_x$ need not be expressible as an application of $e^{V^* t}$ to $k_x$ since the kernel sections may not be in the domain of $e^{V^*t}$.

The goal of \cref{sec:first_amplification,sec:fock_space} ahead is to construct a unitary evolution group on an amplification of $\mathcal H$, in such a way that the non-unitary dynamics under $e^{tV}$ can be recovered.

\subsection{Vector fields with trigonometric polynomial coefficients}
\label{sec:poly_examples}

Before moving forward, consider the example of a smooth vector field $\vec V\colon X \to TX$ on the $d$-dimensional torus, $X = \mathbb T^d$, with the invariant frame $\left\{\frac{\partial\ }{\partial \theta_j}\colon X \to TX \right\}_{j=1}^d$,
\begin{equation}
	\label{eq:v_poly}
	\vec V(\theta) = \sum_{j=1}^d p_j(e^{\pm i \theta_1}, \ldots,e^{\pm i \theta_d}) \, \frac{\partial\ }{\partial \theta_j}, \quad \theta = (\theta_1, \ldots, \theta_d).
\end{equation}
Here, $p_j(e^{\pm i \theta_1},\ldots,e^{\pm i \theta_d})$ are real-valued polynomials and $\theta_j \in [0, 2\pi)$ are standard angle coordinates.
For the RKHS $\mathcal H$, let
\begin{equation*}
	k(\theta, \theta') = \sum_{m \in \mathbb Z^d} \omega^{-2}(m) e^{i m \cdot (\theta-\theta')},
\end{equation*}
where the weights $\omega \colon \mathbb Z^d \to \mathbb R_{>0}$ are normalized, inverse square-summable, and subconvolutive on the dual group $\widehat{\mathbb T^d} \cong \mathbb Z^d $,
\begin{equation}
	\label{eq:subconv}
	\quad \omega(0) = 1, \quad \sum_{m \in \mathbb Z^d} \omega^{-2}(m) < \infty, \quad (\omega^{-2} * \omega^{-2})(m) \leq C \omega^{-2}(m).
\end{equation}
Weights meeting these conditions include the subexponential family,
\begin{equation}
	\label{eq:subexp_weights}
	\omega(m) = e^{\tau \lvert m \rvert^p}, \quad \tau >0, \quad p \in (0,1),
\end{equation}
where $\lvert m \rvert^p := \lvert m_1 \rvert^p + \ldots + \lvert m_d \rvert^p $, among other examples from time--frequency analysis on locally compact abelian (LCA) groups \cite{Feichtinger79,Grochenig07}.
As shown in \cite{DasGiannakis23,DasEtAl23}, the corresponding RKHS $\mathcal H$ is a unital RKHA, with the constant function $1_X\colon X \to \mathbb C$, $1_X(x) = 1$, acting as the unit.
Clearly, $\alg(\{1_X,e^{\pm i \theta_1},\ldots,e^{\pm i \theta_d} \})$ is dense in $\mathcal H$ so \cref{asmp:alg} is satisfied.

This RKHA also satisfies \cref{asmp:main}(a, b) for any vector field with polynomial coefficients and with $\mathcal H_\infty = \text{span} \{\theta \mapsto e^{ij\cdot \theta} : j \in \mathbb Z^d\}$.
To see why, we filter $\mathcal H$ with bandlimited functions,
\begin{equation}
	\label{eq:h_bandlimited}
	\mathcal H_l = \left\{ f\in \mathcal H: \supp(\hat{f}) \subseteq [-l,l]^d \right\} \subset \mathcal H_\infty, \quad l \in \mathbb N_0,
\end{equation}
where $\hat f \in c_0(\mathbb Z^d)$ denotes the Fourier transform of $f \in \mathcal H$.
To verify \cref{asmp:main}(b) observe that multiplication by polynomials, $p_j(e^{\pm i \theta_1},\cdots, e^{\pm \theta_d})$, is bounded as an operator, $M_{p_j}$, from $\mathcal H$ to itself.
In particular, letting $H \ominus K$ denote the orthogonal complement of a subspace $K \subseteq H$ of a Hilbert space $H$, and $p = \max_j \deg p_j$, then
$$M_{p_j}, M_{p_j}^* \colon \mathcal H_l \ominus \mathcal H_{l'} \to \mathcal H_{l+p+1} \ominus \mathcal H_{l'-p-1} $$
are well-defined bounded linear maps for all $l,l' \in \mathbb N_0$, $l' \leq l$.
Moreover, the invariant vector fields are well-defined as operators $\frac{\partial}{\partial \theta_m} \colon \mathcal H_l \to \mathcal H_l$.
By~\eqref{eq:v_poly}, we have $\langle g, V f\rangle_{\mathcal H} =0$ for $g \in \mathcal H_l$ and $f \in \dom(V) \ominus \mathcal H_{l+p+1}$.
Thus, for every $l \in \mathbb N_0$ and $g \in \mathcal H_l$, $f \mapsto \langle g, V f \rangle_{\mathcal H}$
is a bounded linear functional on $\dom(V)$.
This proves that $\mathcal H_\infty \subset \text{dom}(V^*)$ and so we see that $V,V^* \colon \mathcal H_\infty \to \mathcal H_\infty$ satisfy \cref{asmp:main}(b).

For \cref{asmp:main}(a) we may directly compute
\begin{align*}
	(Vf)(\theta) & = \sum_{j=1}^d p_j(e^{\pm i \theta_1}, \ldots,e^{\pm i \theta_d})  \lim_{h_j \to 0} \frac{\langle k_{(\theta_1,\cdots, \theta_j+h_j,\cdots ,\theta_d)}, f \rangle_{\mathcal H} - \langle k_{(\theta_1,\cdots, \theta_j,\cdots ,\theta_d)}, f \rangle_{\mathcal H}}{h_j} \\
	             & = \left\langle \sum_{j=1}^d \overline{p_j(e^{\pm i \theta_1}, \ldots,e^{\pm i \theta_d})} \frac{\partial}{\partial \theta_j} k_\theta, f \right\rangle_{\mathcal H},
\end{align*}
provided the feature map $\theta \mapsto k_\theta$ is smooth and $\frac{\partial}{\partial \theta_j} k_\theta \in \mathcal H$.
Since the Mercer expansion of $k(\theta,\cdot)$ and its term-wise derivatives converge uniformly, then $\frac{\partial}{\partial \theta_j} k_\theta \in \mathcal H$ and
$$\left\lVert\left[ \theta' \mapsto \frac{\partial\ }{\partial \theta_j} k(\theta,\theta') \right]\right\rVert^2_\mathcal H = \left(\sum_{n=-\infty}^\infty e^{-2\tau |n|^p} \right)^{d-1} \left(\sum_{n=-\infty}^\infty n^2 e^{-2\tau |n|^p} \right) < \infty.$$
Since at every point $\theta$, $V$ is a finite linear combination of such terms by bounded functions we have
$$\sup_{\theta \in \mathbb T^d} \lVert \vec{V} \cdot \nabla_\theta k(\theta, \cdot)\rVert_{\mathcal H} = C<\infty,$$
and so
$$\lVert Vf \rVert_{\infty} = \sup_{\theta \in \mathbb T^d} \left\lvert \langle \vec{V}\cdot \nabla_\theta k(\theta,\cdot), f \rangle_{\mathcal H} \right\rvert \leq C \lVert f \rVert_{\mathcal H}$$
for some constant $C$.
Hence \cref{asmp:main}(a) is satisfied.

We next show that \cref{asmp:main}(c) is also satisfied for vector fields on $\mathbb T^d$ with trigonometric polynomial coefficients.

\begin{prop}\label{prop:example}
	The Fourier functions $e^{i m\cdot \theta}$ are jointly analytic vectors for the vector field $V$~\eqref{eq:v_poly} and its adjoint $V^*$ as unbounded operators on the RKHA $\mathcal H$ associated with any weight satisfying~\eqref{eq:subconv}.
\end{prop}

\begin{proof}
	Observe that the multiplication operators $M_{p_j}$ and their adjoints are graded maps
	\begin{equation*}
		M_{p_j}, M_{p_j}^* \colon \mathcal H_l \to \mathcal H_{l+d},
	\end{equation*}
	satisfying the operator norm bounds
	\begin{equation*}
		\lVert M_{p_j} \rvert_{\mathcal H_l} \rVert, \; \lVert M^*_{p_j} \rvert_{\mathcal H_l} \rVert \leq M := \max_j \lVert M_{p_j} \rVert,
	\end{equation*}
	uniformly in $l \in \mathbb N$.
	Moreover, $\frac{\partial\ }{\partial\theta_j} \colon \mathcal H_l \to \mathcal H_l$ has operator norm $l$.
	Together, these facts imply that $V, V^* \colon \mathcal H_l \to \mathcal H_{l+p}$ are also graded and satisfy
	$$\lVert V|_{\mathcal H_l} \rVert, \; \lVert V^*|_{\mathcal H_l} \rVert \leq d M (l+p).$$
	Gradedness also implies that $V^n$ and $(V^*)^n$ are well-defined linear maps from $\mathcal H_\infty$ to $\mathcal H$ for any $n \in \mathbb N$.
	Turning to the joint analyticity condition, fix $f \in \mathcal H_l$ and observe that the series
	\begin{align*}
		\sum_{n=0}^\infty \frac{|z|^n}{n!} \sum_{\epsilon \colon [n] \to [2]} \left\lVert V_{\epsilon(1)} \cdots V_{\epsilon(n)} f \right\rVert_{\mathcal H}
		 & \leq \sum_{n=0}^\infty \frac{|z|^n}{n!} 2^n \sup_{\epsilon \colon [n] \to [2]} \left\lVert V_{\epsilon(1)} \cdots V_{\epsilon(n)} f \right\rVert_{\mathcal H} \\
		 & \leq \sum_{n=0}^\infty \frac{|z|^n}{n!} 2^n \prod_{j=1}^n (dM(l+jp)) \lVert f \rVert_{\mathcal H}                                                             \\
		 & = \sum_{n=0}^\infty \frac{|z|^n}{n!}n! (2dMp)^n \prod_{j=1}^n \left( 1+ \frac{l}{jp}\right) \lVert f \rVert_{\mathcal H}                                      \\
		 & \leq \sum_{n=0}^\infty (2dMp|z|)^n \prod_{j=1}^n \exp \left(\frac{l}{pj}\right) \lVert f \rVert_{\mathcal H}                                                  \\
		 & = \sum_{n=0}^\infty (2dMp|z|)^n \exp \left(\frac{l}{p} \sum_{j=1}^n \frac{1}{j}\right) \lVert f \rVert_{\mathcal H}                                           \\
		 & \leq \sum_{n=0}^\infty (2dMp|z|)^n \exp \left(\frac{l}{p}(\ln(n)+1)\right) \lVert f \rVert_{\mathcal H}                                                       \\
		 & = \lVert f \rVert_{\mathcal H} e^{l/p} \sum_{n=0}^\infty n^{l/p}(2dMp|z|)^n                                                                                   \\
	\end{align*}
	has a radius of convergence $R_f = (2dMp)^{-1}$.
\end{proof}

The previous results provide a class of dynamical systems and RKHAs that satisfy \cref{asmp:main}.

\begin{rk}
	In particular cases, the Koopman operator $U_t$ may be defined for all $t \in \mathbb R$.
	For an example, consider the vector field $\vec V = \sin(\theta) \frac{d }{d\theta}$ on the unit circle.
	The evolution of $\phi_1(\theta) = e^{i\theta}$ under the corresponding flow, $\Phi^t\colon \mathbb T \to \mathbb T$, can be found analytically from the equivalent complex differential equation $\frac{dz}{dt} = \frac{1}{2}(z^2-1)$, $z(t) = e^{i\theta(t)}$.
	The solution is given by
	$$(U_t \phi_1)(\theta) =-\tanh(t/2)+ \sech^2(t/2)\sum_{n=1}^\infty \tanh^{n-1}(t/2) e^{in\theta}.$$
	One can directly verify this solution and show that $U_t\phi_1 \in \mathcal H$ for all $t \in \mathbb R$ where $\mathcal H$ is the RKHA built with the subexponential weights in \eqref{eq:subexp_weights}.
	Since $U_t$ is a composition operator and $\mathcal H$ is a Banach algebra, we have $U_t(fg) = (U_t f) (U_t g)$ for all $f, g \in \dom(U_t)$ such that $fg \in \dom(U_t)$.
	The latter, together with the fact that $\overline{U_t(f)} = U_t(\overline{f})$, implies that $\alg\{1_X, \phi_1, \phi_{-1}\} \subseteq \dom(U_t)$ for all $ t\in \mathbb R$.
	Thus, $U_t$ is densely defined for all $t \in \mathbb R$ in $\mathcal H$.

	The distinction between analytic vectors for $V$ and the domain of $U_t$ already appears in this example.
	By extracting the highest-order Fourier mode we get the bound
	$$\lVert V^n \phi_1 \rVert_\mathcal H = \left\lVert \frac{i^n}{(2i)^n}n! \phi_1^{n+1} + \cdots \right\rVert_\mathcal H \geq \frac{n!}{2^n} \omega(n+1),$$
	and so the sum $\sum_{n=0}^\infty \frac{t^n}{n!} \lVert V^n \phi_1 \rVert_\mathcal H$ diverges for $|t| \geq 2$.
	This means $e^{Vt}\phi_1$ is ill-defined for $|t|\geq 2$ even though $U_t \phi_1$ lies in $\mathcal H$ for all $t \in \mathbb R$.
\end{rk}

\subsection{Existence of locally compact examples} Examples in the locally compact case $X=\mathbb R^d$ with vector fields $V=\sum_{j=1}^d p_j(x_1,\cdots,x_d) \frac{\partial}{\partial x_j}$ with general polynomial coefficients pose significant challenges due to the possibility of finite-time blowup (e.g., $V= x^2 \frac{d}{dx}$ on $\mathbb R$).
However, there is a less satisfying source of locally compact examples.

Let $Y$ be a smooth, compact manifold, $\vec W \colon Y \to TY$ a smooth vector field, $\mathcal H \subset C^\infty(Y)$ an RKHA, $\mathcal H_\infty \subset \mathcal H$ a dense set of jointly analytic vectors, and $\mathcal S \subset \mathcal H_\infty$ a finite dimensional space satisfying \cref{asmp:main,asmp:alg}.
If $\vec W$ has a fixed point, $\vec W_p = 0$, then we may remove the point $p$ to obtain an example on a locally compact space $X= Y\backslash \{p\}$ with $\vec V =\vec W|_X$ and the original RKHA viewed as a space of functions on $X$, $\mathcal H \subset C^\infty(X)$.

\section{Fibonacci amplification}
\label{sec:first_amplification}

Given \cref{asmp:main}, the operators
\begin{equation}
	\label{eq:a_s_ops}
	A = \frac{V-V^*}{2}, \quad S=\frac{V+V^*}{2},
\end{equation}
are antisymmetric (resp.\ symmetric) on $\dom(A) = \dom(S) = \mathcal H_\infty$, and their domain is a dense subspace of analytic vectors in $\mathcal H$.
We dilate $A$ to essentially skew-adjoint operators $L$ and $L^+$ from which the action of $V$ and $V^*$, respectively, can be recovered.

To accomplish this, consider the amplification $\mathfrak H = l^2(\mathbb N) \otimes \mathcal H$.
Let $e_1, e_2, \ldots$ be the canonical orthonormal basis vectors of $\ell^2(\mathbb N)$, $e_i(j) = \delta_{ij}$, and $e_{i,j} \in B(\ell^2(\mathbb N))$ the corresponding matrix units, $e_{i,j} e_j = e_i$.
Define $L, L^+ \colon \mathfrak H_0 \to \mathfrak H $ on the dense domain $\mathfrak H_0 = C_c(\mathbb N) \otimes_\text{alg} \mathcal H_\infty \subset \mathfrak H$ as
\begin{equation}
	\label{eq:l_ops}
	\begin{aligned}
		L   & = \sum_{i=1}^\infty e_{i,i+1} \otimes S + e_{i,i} \otimes A - e_{i+1,i} \otimes S,   \\
		L^+ & = \sum_{i=1}^\infty - e_{i,i+1} \otimes S + e_{i,i} \otimes A + e_{i+1,i} \otimes S.
	\end{aligned}
\end{equation}
Let also $F_n$ denote the $n^{th}$ Fibonacci number (i.e., $F_0=0$, $F_1=1$, and $F_n=F_{n-1}+F_{n-2}$), and define $E_n \colon \mathcal H \to \mathfrak H$, $n \in \mathbb N$, as
\begin{equation*}
	E_nf = \sum_{i=1}^n F_i e_{i} \otimes f.
\end{equation*}
Define additionally $p_1\colon \ell^2(\mathbb N) \to \mathbb C$ as the orthogonal projection onto the first component.

Observe that $L$ and $L^+$ are represented by skew-symmetric tridiagonal matrices with entries $A$ and $\pm S$,
\begin{equation}
	\label{eq:l_mats}
	L =\left[
		\begin{array}{ccccc}
			A      & S      & 0      & 0      & \cdots \\
			-S     & A      & S      & 0      & \cdots \\
			0      & -S     & A      & S      & \cdots \\
			0      & 0      & -S     & A      & \cdots \\
			\vdots & \vdots & \vdots & \vdots & \ddots \\
		\end{array}\right],
	\quad
	L^+ =\left[
		\begin{array}{ccccc}
			-A     & S      & 0      & 0      & \cdots \\
			-S     & -A     & S      & 0      & \cdots \\
			0      & -S     & -A     & S      & \cdots \\
			0      & 0      & -S     & -A     & \cdots \\
			\vdots & \vdots & \vdots & \vdots & \ddots \\
		\end{array}\right].
\end{equation}
It is evident from this representation that $L$ and $L^+$ are antisymmetric operators that have $\mathfrak H_0$ as an invariant subspace.
Observe also that
\begin{equation*}
	(p_1 \otimes \Id) L E_n f =(S+A)f = Vf, \quad (p_1 \otimes \Id) L^+ E_n f = (S-A)f = V^*f,
\end{equation*}
for all $n >1$ and $f \in \mathcal H_\infty$.
Moreover, $E_1 \colon \mathcal H \hookrightarrow \mathfrak H$ is an isometric embedding and
\begin{equation*}
	(p_1 \otimes \Id) L  E_1  = A, \quad (p_1 \otimes \Id) L^+ E_1  = -A.
\end{equation*}
Thus, $L$ and $L^+$ are dilations of $A$ and $-A$, respectively.

\begin{lem}\label{lem:radiusA}
	Let $f \in \mathcal H_\infty$ be jointly analytic with respect to $V$ and $V^*$ with radius of convergence for the series
	$$ \sum_{n=0}^\infty \frac{z^n}{n!} \sum_{\epsilon \colon [n] \to [2]} \left\lVert V_{\epsilon(1)} \cdots V_{\epsilon(n)} f \right\rVert_{\mathcal H}, \quad z \in \mathbb C,$$
	$R_f > 0$.
	Then $f$ is also jointly analytic with respect to $A$ and $S$ with a radius of convergence greater than or equal to $R_f$.
\end{lem}

\begin{proof}
	Let $A_1=A$ and $A_2=S$ and consider
	\begin{align*}
		\sum_{n=0}^\infty \frac{z^n}{n!} \sum_{\epsilon \colon [n] \to [2]} \left\lVert A_{\epsilon(1)} \cdots A_{\epsilon(n)} f \right\rVert_{\mathcal H} & =\sum_{n=0}^\infty \frac{z^n}{n!} \sum_{\epsilon \colon [n] \to [2]} \left\lVert \left(\frac{V\pm_{\epsilon(1)} V^*}{2}\right) \cdots \left(\frac{V\pm_{\epsilon(n)} V^*}{2}\right) f \right\rVert_{\mathcal H} \\
		                                                                                                                                                   & \leq \sum_{n=0}^\infty \frac{z^n}{n!} \frac{1}{2^n}\sum_{\epsilon \colon [n] \to [2]} \sum_{\delta \colon [n] \to [2]} \left\lVert V_{\delta(1)} \cdots V_{\delta(n)} f \right\rVert_{\mathcal H}               \\
		                                                                                                                                                   & = \sum_{n=0}^\infty \frac{z^n}{n!} \sum_{\epsilon \colon [n] \to [2]} \left\lVert V_{\epsilon(1)} \cdots V_{\epsilon(n)} f \right\rVert_{\mathcal H},
	\end{align*}
	which converges for $|z| < R_f$.
\end{proof}

Next, we will need the following results on analytic vectors, originally due to Nelson \cite{Nelson59}.

\begin{thm}[\cite{Nelson59}*{Lemma~5.1}, \cite{ReedSimon75}*{\S X.2, Theorem~X.39}] The following hold for a densely defined symmetric (resp.\ skew-symmetric) operator on a Hilbert space.
	\label{thm:nelson}
	\begin{enumerate}[(a)]
		\item If $\dom(T)$ contains a dense set of analytic vectors, then $T$ is essentially self-adjoint (resp.\ skew-adjoint).
		\item If $T$ is closed, then it is self-adjoint (resp.\ skew-adjoint) if and only if $\dom(T)$ has a dense set of analytic vectors.
		      The set of analytic vectors also forms a core for $T$ (i.e., the closure of $T$ restricted to its analytic vectors is $T$).
	\end{enumerate}
\end{thm}

By \cref{asmp:main,thm:nelson}, $A$ and $S$ from~\eqref{eq:a_s_ops} are essentially skew-adjoint (resp.\ self-adjoint) operators on the RKHS $\mathcal H$.
In addition, we have:

\begin{thm}
	\label{thm:amplification}
	Under \cref{asmp:main}, the unbounded operators $L$ and $L^+$ from~\eqref{eq:l_ops} have the following properties:
	\begin{enumerate}[(a)]
		\item $L$ and $L^+$ are essentially skew-adjoint on $\mathfrak H_0$.
		\item For every $n\in \mathbb N$, $k < n$, and $f \in \mathcal H_\infty$,
		      \begin{equation*}
			      (p_1 \otimes \Id) L^k E_nf = V^kf, \quad (p_1 \otimes \Id) (L^+)^k E_nf = (V^*)^kf.
		      \end{equation*}
	\end{enumerate}
\end{thm}

\begin{proof}
	By symmetry, it suffices to prove this for $L$.
	For part (a) we will exhibit a dense family of analytic vectors implying, by \cref{thm:nelson}(a), essential skew adjointness of $L$.

	We need to analyze the combinatorial growth of terms due to the connectivity of $L$ as an infinite matrix.
	Let $A_1 = A$ and $A_2 = S$.
	Precisely, the number of terms of the form $A_{\epsilon(1)} \cdots A_{\epsilon(n)}f$, in the $k^{th}$ entry of $L^n(e_l \otimes f)$ is counted by the adjacency matrix for the graph,
	$$\Gamma =
		\begin{tikzpicture}
			\clip (-.2,-.1) rectangle (6.5,.7);
			\draw[black, very thick] (0,0)--(5,0);
			\draw[black, fill=black] (0,0) circle (.6mm);
			\draw[black, fill=black] (1,0) circle (.6mm);
			\draw[black, fill=black] (2,0) circle (.6mm);
			\draw[black, fill=black] (3,0) circle (.6mm);
			\draw[black, fill=black] (4,0) circle (.6mm);
			\draw[black, fill=black] (5.5,0) circle (.2mm);
			\draw[black, fill=black] (5.75,0) circle (.2mm);
			\draw[black, fill=black] (6,0) circle (.2mm);

			\draw[black, very thick] (0,0) to[out=120,in=180] (0,.5) to[out=0, in=70] (0,0);
			\draw[black, very thick] (1,0) to[out=120,in=180] (1,.5) to[out=0, in=70] (1,0);
			\draw[black, very thick] (2,0) to[out=120,in=180] (2,.5) to[out=0, in=70] (2,0);
			\draw[black, very thick] (3,0) to[out=120,in=180] (3,.5) to[out=0, in=70] (3,0);
			\draw[black, very thick] (4,0) to[out=120,in=180] (4,.5) to[out=0, in=70] (4,0);
		\end{tikzpicture}$$
	Since the operator norm of this adjacency matrix is given by $\lVert \Gamma \rVert = 3$ (use functional calculus and the fact $\lVert \Gamma -1 \rVert =2$), we have $\lvert \Gamma^n(k,l) \rvert \leq \lVert \Gamma \rVert^n = 3^n$.
	Then the vectors $e_l \otimes f \in l^2(\mathbb N) \otimes \mathcal H_\infty$ are analytic for $L$ with a radius of convergence greater than or equal to $R_f/3$ when the radius of convergence for $f$ is $R_f$ as an analytic vector of $V$ and $V^*$ (see \cref{lem:radiusA}):
	\begin{align*}
		\sum_{n=0}^\infty \frac{|z|^n}{n!} \lVert L^n (e_l \otimes f) \rVert_{\mathcal H}
		 & \leq \sum_{n=0}^\infty \frac{|z|^n}{n!} \left(\sum_{k=1}^{l+n} |\Gamma^n(k,l)| \right) \max_{\epsilon\colon [n] \to [2]} \left\lVert A_{\epsilon(1)} \cdots A_{\epsilon(n)} f \right\rVert_{\mathcal H} \\
		 & \leq \sum_{n=0}^\infty \frac{|z|^n}{n!} (l+n) 3^n \sum_{\epsilon \colon [n] \to [2]} \left\lVert A_{\epsilon(1)} \cdots A_{\epsilon(n)} f \right\rVert_{\mathcal H}                                     \\
		 & \leq \sum_{n=0}^\infty \frac{(3|z|)^n}{n!} (l+n) \sum_{\epsilon \colon [n] \to [2]} \left\lVert V_{\epsilon(1)} \cdots V_{\epsilon(n)} f \right\rVert_{\mathcal H}.
	\end{align*}
	Therefore, $\mathfrak H_0 = C_c(\mathbb N) \otimes_\text{alg}  \mathcal H_\infty$ is a dense family of analytic vectors for $L$.

	For part (b) consider the formal vector $E f = \sum_{i=1}^\infty F_i e_{i} \otimes f \in \mathcal H^{\mathbb N}_\infty$ and its behavior under $L$.
	We claim that
	\begin{equation*}
		E V f = L E f
	\end{equation*}
	as formal vectors and matrices.
	Indeed, we have
	\begin{align*}
		L E f
		 & = \sum_{i=1}^\infty \left( F_{i+1} e_i \otimes Sf + F_i e_i \otimes Af - F_i e_{i+1} \otimes Sf \right)                              \\
		 & =F_2 e_1 \otimes Sf + F_1 e_1 \otimes Af  + \sum_{i=2}^\infty   F_{i+1} e_i \otimes Sf + F_i e_i \otimes Af - F_{i-1} e_i \otimes Sf \\
		 & = e_1 \otimes (A+S)f +\sum_{i=2}^\infty\left( (F_{i+1} -F_{i-1}) e_i \otimes Sf + F_i e_i \otimes Af \right)                         \\
		 & = F_1e_1 \otimes Vf +\sum_{i=2}^\infty  \left(F_i e_i \otimes Sf + F_i e_i \otimes Af \right)                                        \\
		 & = EVf.
	\end{align*}
	Since the infinite matrix $L$ is banded, when the map $E$ is truncated, we have $(p_1 \otimes \Id) L^k E_n f = (p_1 \otimes \Id) L^k E f= V^k f$ as formal vectors provided $k < n$.
\end{proof}

By \cref{thm:amplification} and Stone's theorem on 1-parameter unitary evolution groups \cite{Stone32}, we deduce that the closures $\overline{L} = -L^*$ and $\overline{L^+} = -(L^+)^* $ are skew-adjoint operators that generate strongly continuous groups of unitary operators, $e^{t \overline{L}}$ and $e^{-t \overline{L}}$, respectively, built using the Borel functional calculus.
Moreover, restricted to analytic vectors $u \in \mathfrak{H}_0$ these unitaries satisfy $e^{t\overline{L}} u = e^{tL} u$ and $e^{t \overline{L^+}}u = e^{t L^+} u$, where the right-hand sides are given by Taylor series.
We now compare $e^{tL}$ (resp.\ $e^{tL^+}$) applied to $E_n f$ and the vector $\sum_{k=0}^n \frac{t^k}{k!} V^k f$ (resp.\ $\sum_{k=0}^n \frac{t^k}{k!} (V^*)^k f$).
This bound is nearly identical to the one used to prove analyticity of $e_l \otimes f$.

\begin{prop}
	\label{prop:phi_evo}
	Under \cref{asmp:main}, let $R_f$ be the radius of convergence for the joint analyticity of $f$ with respect to $V$ and $V^*$.
	Let $L$, $L^+$, and $E_n$ be as above and $|t| \leq R_f/3$.
	Then the following limits vanish.
	$$\lim_{n \to \infty} \left\lVert (p_1 \otimes \Id) e^{Lt} E_nf - \sum_{k=0}^{n-1} \frac{t^k}{k!} V^k f \right\rVert_{\mathcal H} = 0, \quad \lim_{n \to \infty} \left\lVert (p_1 \otimes \Id) e^{L^+t} E_nf - \sum_{k=0}^{n-1} \frac{t^k}{k!} (V^*)^k f \right\rVert_{\mathcal H} = 0.$$
\end{prop}
\begin{proof}
	Let $A_1 = A$ and $A_2 = S$.
	We have
	\begin{align*}
		\left\lVert (p_1 \otimes \Id) e^{Lt} E_nf - \sum_{k=0}^{n-1} \frac{t^k}{k!} V^k f \right\rVert_{\mathcal H}
		 & = \left\lVert \sum_{k=n}^\infty \frac{t^k}{k!} (p_1 \otimes \Id)L^k E_n f \right\rVert_{\mathcal H}                                                                                                    \\
		 & \leq  \sum_{k=n}^\infty \frac{|t|^k}{k!} \left(\sum_{l=1}^{k} |\Gamma^k(1,l)|\right) \max_{\epsilon \colon [k] \to [2]} \left\lVert  A_{\epsilon(1)} \cdots A_{\epsilon(k)}f \right\rVert_{\mathcal H} \\
		 & \leq  \sum_{k=n}^\infty \frac{k(3|t|)^k}{k!} \sum_{\epsilon \colon [k] \to [2]} \left\lVert  A_{\epsilon(1)} \cdots A_{\epsilon(k)}f \right\rVert_{\mathcal H}.
	\end{align*}
	Since this is the tail of an absolutely convergent sum, provided $3|t| < R_f$, we have the result
	\begin{equation*}
		\lim_{n \to \infty} \left\lVert (p_1 \otimes \Id) e^{Lt} E_nf - \sum_{k=0}^n \frac{t^k}{k!} V^k f \right\rVert_{\mathcal H} = 0.
	\end{equation*}
	The limit for $L^+$ follows by symmetry.
\end{proof}

By an $\varepsilon/2$ trick and due to analyticity of $f$, we also have well-definedness of $e^{Vt} f$ and $e^{V^*t} f$ for $|t| < R_f$ and convergence of
$$\lim_{n \to \infty} \left\lVert (p_1 \otimes \Id) e^{Lt} E_nf - e^{Vt} f \right\rVert_{\mathcal H} = 0, \quad \lim_{n \to \infty} \left\lVert (p_1 \otimes \Id) e^{L^+t} E_nf - e^{V^*t} f \right\rVert_{\mathcal H} = 0$$
for $|t| < R_f/3$.
As a result, by \cref{prop:taylor_koopman}, we have
\begin{equation}
	\label{eq:koopman_recovery_from_l}
	\lim_{n \to \infty} \left\lVert (p_1 \otimes \Id) e^{Lt} E_nf - U_t f \right\rVert_{\mathcal H} = 0,
\end{equation}
so the dilation of $A$ to $L$ recovers the Koopman evolution $U_t f$ in $\mathcal H$, again up to time $|t| < R_f/3$.
Note that we do not have an analogous result for the transfer operator evolution under $P^t$ since our proof of \cref{prop:taylor_koopman} relied on the fact that $V$ satisfies the Leibniz rule (whereas, in general, $V^*$ does not).

Another aspect of \cref{prop:phi_evo} and related results such as~\eqref{eq:koopman_recovery_from_l} worth keeping in mind is that the radius of convergence $R_f$ need not have a nonzero lower bound over $f \in \mathcal H_\infty$.
In other words, the evolution time $t$ for which $e^{tL}$ faithfully recovers the true Koopman evolution $U_t$ on $\mathcal H_\infty$ may be arbitrarily small.
To overcome this limitation, in \cref{sec:fock_space} we will perform an additional amplification to a weighted Fock space with commutative Banach algebra structure, on which $L$ induces a free derivation.
Under \cref{asmp:alg}, this will allow us to consistently recover the Koopman evolution of observables in finitely generated, dense subalgebras of $\mathcal H$ for evolution times that are uniformly bounded away from 0.

\begin{rk}
	\label{rk:diagonalizable_l}
	In the special case that $V$ is antisymmetric as an unbounded operator on $\mathcal H$, we have
	\begin{equation}
		\label{eq:l_v_skew_adj}
		L = -L^+ =\Id \otimes V.
	\end{equation}
	As a result, if $\phi \in \mathcal H_\infty$ is an eigenvector of $V$ then $\eta \otimes \phi$ is an eigenvector of $L$ for any nonzero $\eta \in \ell^2(\mathbb N)$.
	In particular, if $V$ is diagonalizable with an orthonormal basis of eigenvectors $\phi_j$ for $\mathcal H$, then $L$ is also diagonalizable with $e_i \otimes \phi_j$ forming an orthonormal basis of $\mathfrak H$.
\end{rk}

\begin{rk}
	\label{rk:alt_first_amplification}
	In \cref{app:alt_first_amplification} we present an alternative strategy, wherein $V$ and $V^*$ are lifted to essentially-skew adjoint operators, $L'$ and $L^{\prime +}$, respectively, that are represented by hollow tridiagonal matrices with off-diagonal entries given by weighted linear combinations of $V$ and $V^*$.
	These operators satisfy an analog of \cref{thm:amplification} (\cref{thm:alt_amplification}), and the results obtained in the remainder of the paper for $L$ and $L^+$ also hold for $L'$ and $L^{\prime+}$.
	However, a key difference between $L$ and $L'$ is that the latter does not factorize as in~\eqref{eq:l_v_skew_adj} when $V$ is antisymmetric.
	In consequence, we cannot deduce existence of eigenvectors of $L'$ on the basis of existence of eigenvectors of $V$.
	For this reason, the Fibonacci amplification approach presented in this section is our method of choice.
	We include the alternative approach in \cref{app:alt_first_amplification} as we believe it is of independent interest.
\end{rk}

\section{Amplification to weighted Fock space}
\label{sec:fock_space}

In our second amplification we will lift the skew-adjoint operator $L$ from \cref{sec:first_amplification} to an  operator acting as a derivation on a weighted Fock space with RKHA structure.

\subsection{Amplification from finite-dimensional subspace}
\label{sec:fock_space_finite_dim}

We choose a finite-dimensional subspace $\mathcal S \subset \mathcal H_\infty$ from which we recover the full dynamics $\Phi^t$ from a unitary group of composition operators acting on an RKHA.
For this to work, we use \cref{asmp:alg}.
This is not a significant restriction as \cref{prop:example} provides a large class of flows and analytic vectors $\mathcal S = \mathcal H_l \subset \mathcal H_\infty$ satisfying $\mathcal H = \overline{\alg(\mathcal H_l)}$ for  $l>0$ sufficiently large.

Let $R$ be the maximal radius of joint analyticity for all $f \in \mathcal S$ with respect to $V$ and $V^*$.
Then, by \cref{prop:algdomain}, $U_t f = e^{Vt}f = f \circ \Phi^t$ is a densely defined map,
$$U_t \colon \alg(\mathcal S) \to \mathcal H.$$
By the results in \cref{sec:first_amplification}, on the smaller interval $|t| < R/3$, $U_t|_{\mathcal S}$ can be recovered from the unitary evolution $e^{Lt}$,
$$\lim_{n \to \infty} (p_1 \otimes \Id) e^{Lt} E_n f = U_t f, \quad \forall f \in \mathcal S.$$
This limit has no reason to hold for functions $f \in \alg(\mathcal S)$ since $V^*$ and $L$ are not derivations which was crucial for \cref{prop:algdomain}.

To circumvent this obstruction, we dilate to the weighted Fock space
\begin{equation}
	\label{eq:weighted_fock_space}
	\mathfrak F = F_w(\mathfrak H) = \overline{\bigoplus_{n=0}^\infty \left(\mathfrak H^{\vee n}, w(n)^2\langle\cdot,\cdot\rangle_{\mathfrak H^{\vee n}} \right)},
\end{equation}
to take advantage of multiplicativity of the Koopman operator.
This space contains the free commutative algebra generated by $\mathcal S$, also called the tensor algebra:
$$T(\mathcal S) \cong T(e_1 \otimes \mathcal S) \subset T(\mathfrak H) =\bigcup_{N >0} \bigoplus_{n=0}^N \mathfrak H^{\vee n} \subset F_w(\mathfrak H).$$
In \eqref{eq:weighted_fock_space}, $\vee$ is the symmetric tensor product (see \eqref{eq:sym_tensor_product} in \cref{app:Fock_RKHA}), $\mathfrak H^{\vee n}$ is the closed subspace of $\mathfrak H^{\otimes n}$ consisting of symmetric tensors, and $w(n)^2 \langle \cdot, \cdot \rangle_{\mathfrak H^{\vee n}}$ represents the inner product of $\mathfrak H^{\vee n}$ multiplied by the factor $w(n)^2$.
Further details on these constructions can be found in \cref{app:rkha}.

As shown in \cite{GiannakisEtAl26}, if $w \colon \mathbb N_0 \to \mathbb R_+$ is inverse square-summable and subconvolutive (see \eqref{eq:subconv_fock_space_weights} in \cref{app:Fock_RKHA}), $F_w(\mathfrak H)$ supports a bounded coproduct $\Delta \colon \mathfrak F \to \mathfrak F \otimes \mathfrak F$ whose adjoint implements the symmetric tensor product,
\begin{equation}
	\label{eq:fock_space_coproduct_adj}
	\Delta^* (\eta \otimes \xi) = \eta \vee \xi \equiv \frac{\eta \otimes \xi + \xi \otimes \eta}{2}.
\end{equation}
As a result, $\mathfrak F$ has the structure of a commutative, unital Banach algebra with respect to the symmetric tensor product $\vee$, for a norm equivalent to the Hilbert space norm and with $1_{\mathbb C}$ (i.e., the vacuum vector of the Fock space) acting as the unit.
By Gelfand duality, $\mathfrak F$ is isomorphic to a space $\widehat{\mathfrak F} \subset C(\sigma(\mathfrak F))$ of continuous functions on the character spectrum, $\sigma(\mathfrak F) \subset \mathfrak F^*$, of $\mathfrak F$.
This latter space $\widehat{\mathfrak F}$ turns out to be a unital RKHA.
Thus, we may work interchangeably in the weighted Fock space setting of $\mathfrak F$, or the function space/RKHA setting of $\widehat{\mathfrak F}$.
Further discussion on the weighted Fock space is included in \cref{app:Fock_RKHA}.

In this paper, we nominally work with weights in the subexponential family (cf.~\eqref{eq:subexp_weights})
\begin{equation}
	\label{eq:subexp_fock_space_weights}
	w(n) = e^{\tau n^p}, \quad \tau > 0, \quad p \in (0, 1).
\end{equation}
While many of our results hold for general weights satisfying~\eqref{eq:subconv_fock_space_weights}, the particular choice in~\eqref{eq:subexp_fock_space_weights} will be important later on in \cref{sec:state_space_embedding_and_metric}.

Next, let $\mathbb H, \mathbb K$ be Hilbert spaces, and $A\colon \mathbb H \to \mathbb K$ a bounded linear map.
If $\lVert A \rVert \leq 1$, we lift $A$ to a bounded operator $F(A) \in F_w(\mathbb H) \to F_w(\mathbb K)$ defined as $F(A) = \bigoplus_{k=0}^\infty A^{\vee k}$.
If $\lVert A \rVert >1$, this defines a closable operator whose domain includes the tensor algebra $T(\mathbb H)$,
$$\dom(A) \supset T(\mathbb H) =\bigcup_{N >0} \bigoplus_{k=0}^N \mathbb H^{\vee k}.$$
We will mainly use
$$F(p_1 \otimes \Id)\colon F_w(\mathfrak H) \to F_w(\mathcal H), \quad F(e^{Lt}) \in B(F_w(\mathfrak H)), \quad F(E_n) \colon T(\mathcal H) \to T(\mathfrak H),$$
which are also maps between tensor algebras.

The weighted symmetric Fock space $\mathfrak F = F_w(\mathfrak H)$ also supports the unbounded antisymmetric operator $\mathcal D = 0 \oplus \bigoplus_{n=1}^\infty \Id^{\vee n-1} \vee L$, whose closure, $\overline{\mathcal D}$, is the generator of $F(e^{\overline{L}t})$.
By the same argument from \cite{ReedSimon80}*{Section~VIII.10, Example~2}, $\mathcal D$ is essentially skew-adjoint on the tensor algebra $T(\mathfrak H_0)$ and $F(e^{\overline{L}t}) = e^{\overline{\mathcal D} t}$.
We also need the densely defined operator
$$\displaystyle \mathfrak m = \sum_{k=0}^\infty \Delta^*_k \colon T(\mathcal H) \to \mathcal H, \quad \mathfrak m \left( z \oplus \bigoplus_{n=1}^N f_n  \right) = z 1_{\mathcal H} + \sum_{n=1}^N \Delta^*_n f_n,$$
to map back down to $\mathcal H$.
While $\mathfrak m$ as defined requires unitality of $\mathcal H$, if $\mathcal H$ is nonunital then $\mathcal H$ can be replaced with a unital RKHA, $\mathbb C \oplus \mathcal H \subset C^\infty(X)$ with the kernel $\tilde{k}(x,y) =1+k(x,y)$ as shown in \cite{GiannakisMontgomery25}.

Since the algebra generated by $\mathcal S$ densely spans $\mathcal H$, for any $f \in \mathcal H$ and $\epsilon > 0$, we can approximate $f$ by $\mathfrak m \tilde{f}$ where $ \tilde{f} = \sum_{k=0}^l \tilde{f}_k \in T(\mathcal S)$ for $\tilde{f}_k \in \mathcal S^{\vee k}$ and $\lVert f - \mathfrak m \tilde f \rVert_{\mathcal H} < \epsilon$.
Concretely, we can build $\tilde f$ by defining the nested family of subspaces
$$\mathcal H_l = \left\{\mathfrak m \tilde{f}:\, \tilde{f} = \sum_{k=0}^l \tilde{f}_k : \tilde f_k \in \mathcal S^{\vee k} \right\},$$
where $\mathcal H_l \subset \mathcal H_{l+1}$ and $\bigcup_{l \in \mathbb N} \mathcal H_l$ lies dense in $\mathcal H$.
We can then choose $\tilde f = \Delta_l^{*\dag} \proj_{\mathcal H_l} f $, where $\proj_{\mathcal H_l} \colon \mathcal H \to \mathcal H$ is the orthogonal projection onto $\mathcal H_l$, $\Delta_l^{*\dag} \colon \mathcal H \to \mathcal H^{\otimes l}$ is the pseudoinverse of $\Delta_l^*$ (i.e., $\Delta_l^* \Delta_l^{*\dag} = \Id_{\mathcal H}$ and $\Delta_l^{*\dag}$ is bounded since $\Delta_l^*$ has closed range when $\mathcal H$ is unital), and $l \in \mathbb N$ is sufficiently large.
In the context of the torus examples of \cref{sec:poly_examples}, the spaces $\mathcal H_l$ coincide with the spaces of bandlimited functions defined in~\eqref{eq:h_bandlimited} when $\mathcal S = \mathcal H_1$, and $\mathfrak m \tilde f$ becomes a bandlimited approximation of $f$ in the Fourier basis on the torus.
Besides this approach, many other schemes are available for approximating $f \in \mathcal H$ by ``lifted'' elements $\tilde f$ in finitely generated subspaces of the Fock space $\mathfrak F$.
For example, when $\tilde f$ is strictly positive, \cite{GiannakisEtAl24} employs nonlinear lifts of the form $ f \mapsto \tilde f = (\proj_{\mathcal S} f^{1/l})^{\otimes l}$ where the $l$-th root $f^{1/l} \in \mathcal H$ is given by the holomorphic functional calculus on the RKHA $\mathcal H$.

\begin{thm}
	\label{thm:tensor_koopman}
	Under \cref{asmp:main,asmp:alg}, and given the same notation as above and $|t| \leq R/3$,
	$$\lim_{n \to \infty} \left\lVert \mathfrak m\left( F(p_1 \otimes \Id) e^{\overline{\mathcal D} t} F(E_n) \tilde{f} \right) - U_t \mathfrak m\tilde{f} \right\rVert_\mathcal H =0.$$
\end{thm}

\begin{proof}
	The claim follows from three observations.
	First, $\tilde{f}$ belongs to the tensor algebra $T(\mathcal S)$ of the finite-dimensional Hilbert space $\mathcal S$, by construction.
	In particular, there exists $l \in \mathbb N$ such that the vectors $F(e^{Vt}) \tilde f$ and $F(p_1 \otimes \Id) e^{\overline{\mathcal D} t} F(E_n) \tilde{f}$ lie in $\mathfrak H^{\vee l}$ for all $\lvert t \rvert \leq R/3$ and $n \in \mathbb N$.
	Second, we have
	$$\lim_{n \to \infty} \left\lVert (p_1 \otimes \Id) e^{Lt} E_n|_{\mathcal S} - e^{Vt}|_{\mathcal S} \right\rVert_{\mathcal H} =0,$$
	by \cref{prop:phi_evo} and finite dimensionality of $\mathcal S$.
	Third, $e^{Vt}$ is multiplicative, by \cref{prop:algdomain}.
	With these observations, and the fact that norm convergence $a_n \to a$ of operators on a Hilbert space implies norm convergence of $a_n^{\vee k} \to a^{\vee k}$, we obtain,
	\begin{multline*}
		\lim_{n \to \infty} \left\lVert \mathfrak m\left( F(p_1 \otimes \Id) F(e^{Lt}) F(E_n) \tilde{f} \right) - e^{Vt} \mathfrak m\tilde{f} \right\rVert^2_{\mathcal H} \\
		\begin{aligned}
			 & =\lim_{n \to \infty} \left\lVert \mathfrak m\left( F(p_1 \otimes \Id) F(e^{Lt}) F(E_n) \tilde{f}\right) - \mathfrak m\left( F(e^{Vt}) \tilde{f}\right) \right\rVert^2_{\mathcal H}                                                                           \\
			 & \leq \lim_{n \to \infty} \left\lVert \mathfrak m|_{\mathcal H^{\vee l}} \right\rVert^2 \sum_{k=0}^l w(k)^2\left\lVert \left. \left((p_1 \otimes \Id) e^{Lt} E_n- e^{Vt}\right)\right|_{\mathcal S}^{\vee k} \tilde{f_k} \right\rVert_{\mathcal H^{\vee k}}^2 \\
			 & =0.
		\end{aligned}
	\end{multline*}
\end{proof}

This accomplishes a uniform time interval over which the Koopman evolutions of functions from a dense domain, $\ran(\mathfrak m) \subset \mathcal H$, can be recovered by a unitary operator, $e^{\overline{\mathcal D} t}$.
\Cref{thm:tensor_koopman} is summarized by the diagram below.
All maps below are algebra homomorphisms and the diagram commutes in the limit as $n$ goes to infinity.
$$ \begin{tikzcd}
		T(\mathfrak H) \arrow{r}{e^{\mathcal D t}}  & T(\mathfrak H) \arrow{d}{F(p_1 \otimes \Id)} \\
		T(\mathcal S) \arrow{u}{F(E_n)} \arrow[swap]{d}{\mathfrak m} & T(\mathcal H) \arrow{d}{\mathfrak m} \\
		\text{dom}(U_t) \arrow{r}{U_t}& \mathcal H
	\end{tikzcd}
$$

Iterating the approximations from \cref{thm:tensor_koopman}, we can uniformly approximate the Koopman evolution of general continuous observables vanishing at infinity, over any finite time interval.
In what follows, $U^t\colon C_0(X) \to C_0(X)$, $t \in \mathbb R$, denotes the Koopman operators on $C_0(X)$, $U^t f = f \circ \Phi^t$, which act as a strongly continuous group of isometries with respect to the supremum norm.

\begin{cor}
	\label{cor:uniform_approx}
	Under \cref{asmp:main,asmp:alg}, for any time $t>0$ any $\epsilon >0$, we may partition the interval $[0,t]$ into subintervals $\{ [t_k,t_{k+1}] \}_{k=0}^N$ such that $|t_{k+1}-t_k| < R/3$.
	Then, by \cref{thm:tensor_koopman}, for every $f \in \mathcal H$, there exists a sequence of positive integers $\{n_k\}_{k=0}^N$ and functions $\{\tilde{f}_k\}_{k=0}^N \subset T(\mathcal S)$ satisfying
	\begin{gather*}
		\left\lVert \mathfrak m\left( F(p_1 \otimes \Id) e^{\mathcal D (t_{k+1}-t_k)} F(E_{n_k}) \tilde{f}_k \right) - \mathfrak m \tilde{f}_{k+1} \right\rVert_\mathcal H \leq \frac{\epsilon}{2(N+2) \lVert \Delta \rVert},\\
		\left\lVert \mathfrak m\left( F(p_1 \otimes \Id) e^{\mathcal D (t_{k+1}-t_k)} F(E_{n_k}) \tilde{f}_k \right) - U_{t_{k+1}-t_k} \mathfrak m\tilde{f}_k \right\rVert_\mathcal H \leq \frac{\epsilon}{2(N+2) \lVert \Delta \rVert}, \\
		\left\lVert f - \mathfrak m\tilde{f}_0 \right\rVert_{\infty} \leq \frac{\epsilon}{N+2}.
	\end{gather*}
	Then,
	$$\left\lVert U^tf - \mathfrak m\left( F(p_1 \otimes \Id) e^{\mathcal D (t_{N+1}-t_N)} F(E_{n_N}) \tilde{f}_N \right) \right\rVert_\infty \leq \epsilon.$$
\end{cor}

\begin{proof}
	Recall the boundedness property~\eqref{eq:rkha_boundedness} for elements of the RKHA $\mathcal H$.
	The conditions above, together with this property, imply that
	$$\left\lVert \mathfrak m \tilde{f}_{k+1} - U_{t_{k+1}-t_k} \mathfrak m\tilde{f}_k \right \rVert_\infty \leq \frac{\epsilon}{N+2}.$$
	Then by stringing together these inequalities,
	\begin{align*}
		\left\lVert U^tf - \mathfrak m\left( F(p_1 \otimes \Id) e^{\mathcal D (t_{N+1}-t_N)} F(E_{n_N}) \tilde{f}_N \right) \right\rVert_\infty & \leq \left\lVert U^tf - U^{t-t_N} \mathfrak m \tilde{f}_N \right\rVert_\infty + \frac{\epsilon}{2(N+2)}               \\
		                                                                                                                                        & \leq \left\lVert U^{t-t_N}\left( U^{t_N}f - \mathfrak m\tilde{f}_N \right) \right\rVert_\infty + \frac{\epsilon}{N+2} \\
		                                                                                                                                        & =\left\lVert U^{t_N}f - \mathfrak m\tilde{f}_N \right\rVert_\infty + \frac{\epsilon}{N+2}                             \\
		                                                                                                                                        & \leq \left\lVert U^{t_{N-1}}f - \mathfrak m\tilde{f}_{N-1} \right\rVert_\infty + \frac{2\epsilon}{N+2}                \\
		                                                                                                                                        & \cdots                                                                                                                \\
		                                                                                                                                        & \leq \left\lVert U^{t_0}f - \mathfrak m\tilde{f}_0 \right\rVert_\infty + \frac{(N+1)\epsilon}{N+2}                    \\
		                                                                                                                                        & = \left\lVert f - \mathfrak m \tilde{f}_0 \right\rVert_\infty + \frac{(N+1)\epsilon}{N+2}                             \\
		                                                                                                                                        & \leq \epsilon.
	\end{align*}
\end{proof}

\begin{rk}
	\label{rk:uniform_approx}
	If $\mathcal H$ is a dense subspace of $C_0(X)$ then a version of \cref{cor:uniform_approx} holds for $f \in C_0(X)$.
	In other words, the amplification to the weighted Fock space achieves uniform approximation of the Koopman evolution of arbitrary continuous observables vanishing at infinity for arbitrarily long finite time intervals.
\end{rk}

\subsection{Spectral triples}
\label{sec:spectral_triples}

We employ the following definition of a weak spectral triple throughout this paper.

\begin{defn}
	\label{defn:spec_triple}
	A weak spectral triple $(\mathbb A, \mathbb H, \mathbb D)$ consists of a Hilbert space $\mathbb H$, a self-adjoint operator $\mathbb D \colon \dom(\mathbb D) \to \mathbb H$ (called Dirac operator), and a unital $*$-algebra $\mathbb A \subseteq B(\mathbb H)$, such that for every $a \in \mathbb A$, $[\mathbb D, a] = \mathbb D a - a \mathbb D$, is densely defined and extends to a bounded operator on $\mathbb H$.
\end{defn}

As noted in \cref{sec:contributions}, \cref{defn:spec_triple} is weaker than Connes' spectral triple \cite{Connes94} by dropping the compactness of the resolvent, but unitality of the algebra, self-adjointness of the Dirac operator, and boundedness of the commutators are retained.
This means that we lose notions of spectral volume and dimension, but, as in the standard framework, our weak spectral triples support strongly continuous unitary groups generated by the Dirac operator and have associated Lipschitz seminorms, $\mathsf L \colon \mathbb A \to [0, \infty)$, $\mathsf L(a) = \lVert [\mathbb D, a] \rVert$, obeying the Leibniz inequality,
\begin{equation*}
	\mathsf L(a b) \leq \lVert a \rVert \mathsf L(b) + \mathsf L(a) \lVert b \rVert.
\end{equation*}
Letting $S(\mathbb A)$ denote the state space of the $C^*$-algebra closure $C^*(\mathbb A) \subset B(\mathbb H)$, i.e., the set of continuous, positive, unital linear functionals on $C^*(\mathbb A)$, the Lipschitz seminorm $\mathsf L$ induces an extended pseudometric $d\colon S(\mathbb A) \to [0, \infty]$, defined as
\begin{equation}
	\label{eq:metric}
	d(\varphi, \psi) = \sup_{a \in \mathbb A}\left\{ \lvert \varphi(a) - \psi(a) \rvert: \mathsf L(a) \leq 1 \right\}.
\end{equation}
Note that, in contrast to the seminorms employed in the theory of compact quantum metric spaces \cite{Rieffel99,Rieffel04}, $\mathsf L$ may have a larger kernel than scalar multiples of the unit of $\mathbb A$, allowing $d$ to take infinite values.

We now build a sequence of weak spectral triples on the weighted symmetric Fock space.
First, recall (e.g., \cite{Hall13,Lehmann04}) that for the (unweighted) symmetric Fock space $F(\mathbb H) = \overline{\bigoplus_{n=0}^\infty \mathbb H^{\vee n}}$ generated by a Hilbert space $\mathbb H$, the creation operator $a_u^+$ associated with $u \in \mathbb H \setminus \{0\}$ is the densely defined unbounded operator on $F(\mathbb H)$ that acts by symmetric tensor multiplication by $u$,
\begin{equation}
	\label{eq:creation_op}
	a_u^+(v_1 \vee \cdots \vee v_n) = u \vee v_1 \vee \cdots \vee v_n, \quad v_1, \ldots, v_n \in \mathbb H.
\end{equation}
Moreover, the annihilation operator, $a_u$, is defined as
\begin{equation*}
	a_u (v_1 \vee \cdots \vee v_n) = \frac{1}{n} \sum_{i=1}^n \langle u, v_i\rangle_{\mathbb H} v_1 \vee \cdots \vee v_{i-1} \vee v_{i+1} \vee \cdots \vee v_n, \quad a_u 1_{\mathbb C} = 0.
\end{equation*}
These operators are formal adjoints of one another on the tensor algebra $T(\mathbb H) \subset F(\mathbb H)$,
\begin{equation*}
	\langle \eta, a_u \xi \rangle_{F(\mathbb H)} = \langle a_u^+ \eta, \xi \rangle_{F(\mathbb H)}, \quad \forall \eta,\xi \in T(\mathbb H),
\end{equation*}
which may be taken as a common dense domain.
In addition, they satisfy the canonical commutation relations (CCRs),
\begin{equation*}
	[a_u, a_v^+] = \langle u, v \rangle_{\mathbb H} \Id.
\end{equation*}
In Fock space theory, the eigenvectors of annihilation operators are known as coherent states.
They and their generalizations find applications in many areas, including many-body quantum physics and signal processing; e.g., \cite{TwarequeEtAl95}.

Returning to the weighted Fock space $\mathfrak F = F_w(\mathfrak H)$ (and reusing notation for creation/annihilation operators), boundedness of the coproduct $\Delta$ implies that symmetric tensor multiplication by $u \in \mathfrak H$ defines a bounded creation operator $a_{u}^+ \colon \mathfrak F \to \mathfrak F$, analogously to~\eqref{eq:creation_op}.
We define the corresponding annihilation operator $a_{u}$ as the adjoint of $a_{u}^+$ in $B(\mathfrak F)$.
Explicitly, we have
\begin{equation*}
	a_{u} (v_1 \vee \cdots \vee v_n) = \frac{w(n)^2}{nw(n-1)^2} \sum_{i=1}^n \langle u, v_i\rangle_{\mathfrak H} v_1 \vee \cdots \vee v_{i-1} \vee v_{i+1} \vee \cdots \vee v_n, \quad a_u 1_{\mathbb C} = 0,
\end{equation*}
where $v_1, \ldots, v_n \in \mathfrak H$.
Note that $a_{u}^* = a_{u}^+$ so that the creation and annihilation operators on the weighted Fock space $\mathfrak F$ form adjoint pairs, as opposed to formal adjoint pairs in the unweighted case.

We can now define a sequence of $*$-algebras, utilizing the annihilation operators on $\mathfrak F$:
\begin{equation*}
	\mathcal A_n = \alg^*\left( \{\Id_\mathfrak F\} \cup \left\{a_{E_nf}: f \in \mathcal S\right\}\right) \subset B(\mathfrak F), \quad n \in \mathbb N.
\end{equation*}
We will also make use of an enveloping $*$-algebra,
$$\mathcal A_\infty = \alg^*\left(\{\Id_\mathfrak F\} \cup \left\{ a_u : u \in \mathfrak H_0 \right\} \right) \subset B(\mathfrak F),$$
and the $C^*$-algebra it generates, $\mathcal A = C^*(\mathcal A_\infty) \subset B(\mathfrak F)$.
Observe that $\mathcal A$ also contains the operators $a_u$ and $a^*_u$ for $u \in \mathfrak H$ since the map $u \mapsto a_u^*$ is continuous from the $\mathfrak H$ norm to the operator norm as $\mathfrak F$ is an RKHA.

To simplify the notation going forward we extend the definitions of the creation and annihilation operators to maps $a_\xi,a^*_\xi \in B(\mathfrak F)$ parameterized by elements $\xi \in \mathfrak F$.
That is, $a_\xi^* \eta = \xi \vee \eta$, and this uniquely determines $a_\xi$.
Note the relations
\begin{gather*}
	a_{\eta \vee \xi}^* = a_{\eta}^* a_{\xi}^*, \quad a_{\eta \vee \xi} = a_{\eta} a_{\xi}, \\
	a^*_{\eta+\xi} = a_\eta^* + a_\xi^*, \quad a_{\eta+\xi} = a_\eta + a_\xi,
\end{gather*}
where $\eta, \xi \in \mathfrak F$.

\begin{thm}
	\label{thm:spectral_triple}
	Under \cref{asmp:main}, the triple  $(\mathcal A_n, \mathfrak F, -i \overline{\mathcal D})$ is a weak spectral triple.
\end{thm}

\begin{proof}
	By construction, $\mathcal D$ is a derivation on $T(\mathfrak H_0)$,  and so,
	$$[\mathcal D, a^*_{F(E_n)\tilde{f}}]=a^*_{\mathcal D(F(E_n)\tilde{f})}, \quad [\mathcal D, a_{F(E_n)\tilde{f}}]=a_{\mathcal D(F(E_n)\tilde{f})}, \quad \forall \tilde f \in T(\mathcal S).$$
	Boundedness of the commutators above then follows from boundedness of creation and annihilation operators on the weighted Fock space $\mathfrak F$.
	For a generic element $a \in \mathcal A_n$, we can apply the Leibniz rule for commutators to conclude boundedness of $[\mathcal D,a]$ as an operator from $T(\mathfrak H_0)$ to $\mathfrak F$.
\end{proof}

Not only is $[\mathcal D,a]$ bounded for $a \in \mathcal A_n$, but $[\mathcal D,\cdot] \colon \mathcal A_\infty \to \mathcal A_\infty$ is well-defined since $[\mathcal D,\cdot]$ is a derivation and\
$$[\mathcal D, a^*_u]=a^*_{Lu} \in \mathcal A_\infty, \quad [\mathcal D, a_u]=a_{Lu} \in \mathcal A_\infty,$$
for all $u \in \mathfrak H_0$ since $\mathfrak H_0$ is invariant under $L$.
Moreover, for every element $a \in \mathcal A$, there is a time evolution
\begin{equation*}
	a(t) = \Adj_{e^{t \overline{\mathcal D}}}a := e^{t\overline{\mathcal D}} a e^{-t\overline{\mathcal D}} \in \mathcal A,
\end{equation*}
This follows from the identities
$$\Adj_{e^{\overline{\mathcal D} t}}(a_u^*) = a_{e^{\overline{L}t}u}^*, \quad \Adj_{e^{\overline{\mathcal D} t}}(a_u) = a_{e^{\overline{L}t}u}, \quad u \in \mathfrak H,$$
and density of the $*$-algebra $\alg^*\left(\{\Id_\mathfrak F\} \cup \left\{ a_u : u \in \mathfrak H \right\} \right)$ inside $\mathcal A$, which is preserved by the isometric action of $\Adj_{e^{\overline{\mathcal D} t}}$.

To see how this unitary evolution is related to the original dynamics, in the next section we will embed the state space manifold $X$ into the state spaces of the $C^*$-algebras $C^*(\mathcal A_n)$.

\section{State space embedding and the quantum pseudometric}
\label{sec:state_space_embedding_and_metric}

We will continue to work with the finitely-generated RKHA $\mathcal H$, weighted Fock space $\mathfrak F = F_w(\mathfrak H) = F_w(\mathcal H \otimes \ell^2(\mathbb N_0))$, $C^*$-algebra $\mathcal A \subset B(\mathfrak F)$, and $*$-algebras $\mathcal A_n \subset \mathcal A_\infty \subset \mathcal A$ from \cref{sec:spectral_triples}.

Recall that since $\mathcal H$ is an RKHA, the norm of the kernel sections $k_x$ is uniformly bounded by the operator norm of the coproduct $\Delta$ (see \eqref{eq:rkha_kernel_bound}).
As a result, without loss of generality, we may rescale the inner product of $\mathcal H$ by $\lVert \Delta \rVert$ so that $\lVert k_x \rVert_\mathcal H \leq 1$.
Such a rescaling guarantees absolute convergence of
$$\mathcal K_x = \sum_{n=0}^\infty w(n)^{-2} (e_1 \otimes k_x)^{\vee n} \in \mathfrak F, \quad  x\in X.$$
One can further show \cite{GiannakisEtAl26} that
\begin{equation*}
	\langle \mathcal K_x,  \eta \vee \xi \rangle_{\mathfrak F} = \langle \mathcal K_x,  \eta \rangle_{\mathfrak F} \langle \mathcal K_x, \xi \rangle_{\mathfrak F}, \quad \forall \eta, \xi  \in \mathfrak F.
\end{equation*}
This means that $\chi_x = \langle \mathcal K_x, \cdot \rangle_{\mathfrak F} \in \mathfrak F^*$ defines a multiplicative linear functional that lies in the character spectrum of the abelian Banach algebra $\mathfrak F$.
Since $\mathfrak F$ is additionally a coalgebra with respect to $\Delta$ from~\eqref{eq:fock_space_coproduct_adj}, the vectors $\mathcal K_x$ are elements of the cospectrum of $\mathfrak F$; that is, the set of nonzero elements $\xi \in \mathfrak F$ satisfying $\Delta \xi = \xi \otimes \xi$ (see \cref{app:spec_cospec}).
We thus see that $\mathcal K \colon X \to \mathfrak F$, $\mathcal K(x) = K_x$, defines a cospectrum-valued feature map on the state space manifold.

Another property of the feature vectors $\mathcal K_x$, which follows by direct verification, is their relationship to $\mathfrak m$,
\begin{gather*}
	\langle \mathcal K_x , F(E_n)\tilde{f}\rangle_\mathfrak F = \langle k_x , \mathfrak m \tilde{f} \rangle_\mathcal H, \quad  \forall \tilde{f} \in T(\mathcal H), \\
	\langle \mathcal K_x , \xi\rangle_\mathfrak F = \langle k_x , \mathfrak m(F(p_1 \otimes \Id) \xi)\rangle_\mathcal H, \quad  \forall \xi \in T(\mathfrak H).
\end{gather*}
From these properties and multiplicativity of $\chi_x$ it follows that the $\mathcal K_x$ are coherent states for all annihilation operators associated with the tensor algebra of $\mathfrak H$; that is,
\begin{equation}
	\label{eq:coherent_states}
	\begin{gathered}
		a_{F(E_n)\tilde{f}} \mathcal K_x = \overline{(\mathfrak m \tilde{f})(x)} \mathcal K_x, \quad \forall\tilde f \in T(\mathcal H),\\
		a_{\xi} \mathcal K_x = \overline{\mathfrak m (F(p_1 \otimes \Id)\xi)(x)} \mathcal K_x, \quad \forall \xi \in T(\mathfrak H).
	\end{gathered}
\end{equation}
The $\mathcal K_x$ also provide a map $\tilde \Psi \colon \mathcal A \to \mathcal H$ back to the RKHA that can intertwine $[\mathcal D, \cdot]$ with $V$ for certain elements of $\mathcal A_\infty$.
Specifically, defining $(\tilde \Psi a)(x) = \langle \mathcal K_x, a 1_{\mathbb C} \rangle_{\mathfrak F}$, we get
\begin{equation*}
	\widehat{\mathcal K} [\mathcal D, a_{F(E_n) \tilde f}^*] = V\widehat{\mathcal K}(a^*_{F(E_n)\tilde{f}}) = V \mathfrak m \tilde f,
\end{equation*}
for all $\tilde f \in T(\mathfrak H_0)$ and $n \geq 2$.

Next, consider the vector states $\varphi_x \in S(B(\mathfrak F))$, defined for every $x \in X$ as
\begin{equation}
	\varphi_x(a) = \frac{\langle \mathcal K_x , a \mathcal K_x\rangle_{\mathfrak F}}{\lVert \mathcal K_x \rVert^2_\mathfrak{F}}.
\end{equation}
These states may be equivalently defined as $\varphi_x(a) = \tr(\rho_x a)$,
where $\rho_x \in B_1(\mathfrak F)$ is the rank-1 density operator
\begin{equation*}
	\rho_x = \frac{\langle \mathcal K_x, \cdot \rangle_{\mathfrak F} \mathcal K_x}{\lVert \mathcal K_x \rVert^2_\mathfrak{F}}.
\end{equation*}
Since $\mathcal A_n \subset B(\mathfrak F)$ for all $n \in \mathbb N$, these are also states of $\mathcal A_n$.
The following theorem is an analog of \cref{thm:tensor_koopman}, using expectations of elements of $\mathcal A_n$ with respect to the quantum states $\varphi_x$ to recover the underlying classical Koopman evolution.

\begin{thm}\label{thm:state_evolution_convergence}
	Under \cref{asmp:main,asmp:alg}, let $\tilde{f} \in T(\mathcal S)$ be a lift of $f \in \mathcal H$ (i.e., $\mathfrak m\tilde{f} = f$).
	Then
	\begin{equation*}
		\varphi_x\left(a^*_{F(E_n)\tilde{f}}\right) = f(x),
	\end{equation*}
	and for $|t| \leq R/3$,
	$$\lim_{n\to \infty} \varphi_x \left(a^*_{F(E_n)\tilde{f}}(t) \right)=\lim_{n\to \infty} \varphi_x \left(\Adj_{e^{t\overline{\mathcal D}}} a^*_{F(E_n)\tilde{f}} \right) = (U_tf)(x),$$
	where $R$ is the joint radius of analyticity for every function in $\mathcal S$ with respect to $V$ and $V^*$.
	Furthermore, the convergence is uniform over $x \in X$.
\end{thm}

\begin{proof}
	The claim can be proved using a nearly identical approach to the proof of \cref{thm:spectral_triple_uniform_convergence} below.
	We omit an explicit presentation in the interest of brevity.
\end{proof}

Note that the creation operators $a^*_{F(E_n)\tilde f}$ are not self-adjoint.
If we wish to work with self-adjoint quantum observables, then for real valued-functions $f = \mathfrak m \tilde{f}$, the observable $O_{\tilde{f},n} = (a^*_{F(E_n)\tilde{f}} +a_{F(E_n)\tilde{f}})/2$ also satisfies
$$\lim_{n \to \infty} \varphi_x\left( O_{\tilde{f},n}(t)\right) = \lim_{n \to \infty} \tr\left( (\Adj_{e^{-t\overline{\mathcal D}}} \rho_x) O_{\tilde{f},n} \right) = (U_tf)(x),$$
for $\lvert t \rvert < R/3$, uniformly over $x \in X$.

\subsection{Algebraically consistent embedding}
\label{sec:psi_embedding}

While \cref{thm:state_evolution_convergence} accomplishes several goals for this work, the vector states $\varphi_x$ do not define $*$-homomorphisms on $C^*(\mathcal A_n)$ and so the behavior above does not extend to the entire $C^*$-algebra.
For example, even though $\varphi_x(a^*_{F(E_n)\tilde{f}}) = f(x)$ and $\varphi_x(a_{F(E_n)\tilde{f}}) = \overline{f(x)}$, in general we cannot expect  $\varphi_x(a_{F(E_n)\tilde{f}}a^*_{F(E_n)\tilde{f}})$  to be equal to $\lvert f(x) \rvert^2$ since $\varphi_x$ is not a $*$-homomorphism.
For this reason we choose a different way to embed $X$ into the state spaces of $C^*(\mathcal A_n)$.
This alternate method requires a new assumption that is already satisfied by the examples in \cref{sec:poly_examples}.

\begin{asmp}\label{asmp:normalized_kernel}
	The RKHA $\mathcal H$ has unit-normalized kernel sections, $\lVert k_x \rVert_\mathcal H =1$ for all $x \in X$.
\end{asmp}

Consider the following truncations of $\mathcal K_x$,
$$\mathcal K_x^{(m)} = \sum_{n=m}^\infty w(n)^{-2} (e_1 \otimes k_x)^{\vee n} \in \mathfrak F, \quad m \in \mathbb N,$$
and their corresponding vector states $$\varphi_x^m(a) = \frac{\langle \mathcal K_x^{(m)} , a \mathcal K_x^{(m)}\rangle_{\mathfrak F}}{\lVert \mathcal K_x^{(m)} \rVert^2_{\mathfrak F}}.$$
A direct calculation shows,
\begin{equation}
	\label{eq:alt_coherent_states}
	a_{E_n f} \mathcal K_x^{(m)} = \overline{f(x)} \mathcal K_x^{(m-1)} \quad \quad a_{u} \mathcal K_x^{(m)} = \overline{\mathfrak m (F(p_1 \otimes \Id)u)(x)} \mathcal K_x^{(m-1)},
\end{equation}
for $f \in \mathcal H$ and $ u \in \mathfrak H$.
Thus, even though $\mathcal K_x^{(m)}$ are not coherent states (cf.~\eqref{eq:coherent_states}), they still exhibit non-trivial relationships with the annihilation operators.

The state space embedding in this subsection requires a careful analysis of $\mathcal A_\infty$.
As motivation we give the definition before its justification.
\begin{defn}\label{def:state_space_embedding}
	Under \cref{asmp:normalized_kernel} and using the subexponential weights in~\eqref{eq:subexp_fock_space_weights}, define $\Psi \colon X \to S(\mathcal A)$ by $\Psi(x) = \Psi_x$, where $\Psi_x(a) = \lim_{m \to \infty} \varphi_x^m(a)$ for $a \in \mathcal A$,
	making $\Psi_x$ the weak$^*$ limit of $\varphi_x^m$ inside $S(\mathcal A)$.
\end{defn}

To justify this definition we first build a standard form for elements of $\mathcal A_\infty$.
Define the weights
\begin{equation*}
	q(n) = \frac{(n+1)w^2(n)}{w^2(n+1)}
\end{equation*}
and the diagonal operators $M_h \in B(\mathfrak F)$, where
$$M_h= \bigoplus_{n=0}^\infty h(n) \Id_{\mathfrak H^{\vee n}}, \quad h \in l^\infty(\mathbb N_0).$$
We will also need the left and right shift operators on $l^\infty(\mathbb N_0)$, $S(h)(n) = h(n+1)$ and $S^{-1}(h)(n) = h(n-1)$ for $n\geq1$ and $S^{-1}(h)(0)=0$.
Then, the creation operators, annihilation operators, and diagonal operators satisfy the weighted commutation relations
\begin{subequations}
	\label{eq:wccr}
	\begin{gather}
		M_q a_\xi a_\eta^* - M_{S^{-1}(q)} a_\eta^* a_\xi = \langle \xi, \eta \rangle_{\mathfrak F} \Id,\\
		[a_\xi,a_\eta] = 0,\\
		a_\xi^* M_h = M_{S^{-1}(h)} a_\xi^*,\\
		a_\xi M_h = M_{S(h)} a_\xi,\\
		\label{eq:wccr5}a_\xi a_\eta^* = a_\eta^* M_{\frac{q}{S(q)}} a_\xi +\langle \xi, \eta \rangle_{\mathfrak F} M_{\frac{1}{q}}.
	\end{gather}
\end{subequations}
In particular, if
\begin{equation}\label{ratio_limit}
	\lim_{n \to \infty} \frac{w(n)}{w(n+1)} = r
\end{equation}
with $r$ strictly positive and finite, then
\begin{equation}
	\label{eq:s_q_limits}
	\lim_{n \to \infty} \frac{q(n)}{S(q)(n)} = 1, \quad \lim_{n \to \infty} \frac{1}{q(n)} = 0,
\end{equation}
making all terms above bounded operators.
This is also satisfied for the subexponential weights in~\eqref{eq:subexp_fock_space_weights}.
Note that $M_{1/q}$ need not belong to $\mathcal A_\infty$, or even its $C^*$-algebra closure.

\begin{defn}
	An element of $\mathcal A_\infty$ is in standard form if it is written as a linear combination of monomials of the form
	$$a^*_{u_1} \cdots a_{u_k}^* M_h a_{v_1} \cdots a_{v_l},$$
	for some $u_i, v_i \in \mathfrak H_0$ and $ h \in \ell^\infty(\mathbb N_0)$.
	We define the degree of such monomials as
	$$\degr(a^*_{u_1} \cdots a_{u_k}^* M_h a_{v_1} \cdots a_{v_l}) = k-l.$$
\end{defn}

With these definitions, we have:

\begin{prop}\label{standard_form}
	The $*$-algebra $\mathcal A_\infty$ is a $*$-subalgebra of the $\mathbb Z$-graded algebra
	$$\mathcal G_\infty = \bigcup_{d \in \mathbb Z} \mathcal G_d \subset B(\mathfrak F),$$
	where $\mathcal G_d$ are subspaces of $B(\mathfrak F)$ spanned by degree-$d$ monomials,
	$$\mathcal G_d = \spn \left\{a^*_{u_1} \cdots a_{u_k}^* M_h a_{v_1} \cdots a_{v_l} :\; \text{$u_i$, $v_i \in \mathfrak H_0$, $d = k-l$, $h \in l^\infty(\mathbb N_0)$} \right\}.$$
	Furthermore,
	$$\mathcal G_\infty \cong \bigoplus_{d\in \mathbb Z} \mathcal G_d$$
	and $\mathcal G_{d_1} \cdot \mathcal G_{d_2} \subset \mathcal G_{d_1+d_2}$.
\end{prop}
\begin{proof}
	A word in terms of $a_{u_i}^*$ and $a_{v_j}$ can be rearranged into a sum of monomials in standard form due to the commutator relations~\eqref{eq:wccr}.
	Well-definedness of the grading of $\mathcal G_\infty$ then comes from the action on the grading of the Fock space.
	Indeed, suppose that $h(m) \neq 0$.
	Then, the monomial
	$$a^*_{u_1} \cdots a_{u_k}^* M_h a_{v_1} \cdots a_{v_l}$$
	is a nonzero map from $\mathfrak H^{\vee (m+l)} \to \mathfrak H^{\vee (m+k)}$.
	This monomial cannot belong to any other graded component or sums of operators from other graded components since all other monomials outside of $\mathcal G_d$ are zero as maps from $\mathfrak H^{\vee (m+l)} \to \mathfrak H^{\vee (m+k)}$ for all $m \in \mathbb N$.
	Linear independence of these spaces implies that $\mathcal G_\infty \cong \bigoplus_{d \in \mathbb Z} \mathcal G_d$ as vector spaces.

	Finally, the commutation relation \eqref{eq:wccr5} implies that starting with a product of $k$ creation and $l$ annihilation operators, and applying \eqref{eq:wccr5} yields a sum of monomials whose degrees are all $k-l$.  Thus, $\mathcal G_{d_1} \cdot \mathcal G_{d_2} \subset \mathcal G_{d_1+d_2}$ since the degree is the difference between the number of creation operators and annihilation operators.
\end{proof}

\begin{prop}\label{prop:state_embedding}
	Suppose that the unital RKHA $\mathcal H$ satisfies \cref{asmp:main,asmp:normalized_kernel} and the weight $w$ of the Fock space is of the subexponential type in~\eqref{eq:subexp_fock_space_weights}.
	Let $h \in c(\mathbb N_0)$ be a real-valued convergent sequence with limit $ h_\infty = \lim_{n\to\infty} h(n)$.
	Then, for every $\tilde{f},\tilde{g} \in T(\mathfrak H)$ we have $$\lim_{n \to \infty} \varphi_x^n(a_{\tilde{f}}^* M_h a_{\tilde{g}}) = h_\infty \mathfrak m( F(p_1 \otimes \Id) \tilde{f})(x) \overline{\mathfrak m( F(p_1 \otimes \Id) \tilde{g})(x)}.$$
\end{prop}

\begin{proof}
	First, note that the symmetric tensor algebra $T(\mathfrak H)$ is generated by the vacuum vector and homogeneous tensors; that is, $T(\mathfrak H) = \spn\{ \{1_{\mathbb C} \} \cup \{ u^{\vee m} : \; \text{$u \in \mathfrak H$, $m \in \mathbb N$} \} \}$.
	See, e.g., the proof of \cite{GiannakisEtAl26}*{Theorem~9} for further details.
	As a result, it is sufficient to prove the claim for $\tilde{f} = f^{\vee m_1}$ and $\tilde{g} = g^{\vee m_2}$.
	Let $m=\text{min}\{m_1,m_2\}$.
	Then, by \eqref{eq:alt_coherent_states},
	\begin{align*}
		\varphi_x^n(a_{\tilde{f}}^* M_h a_{\tilde{g}}) & = \frac{\langle a_{\tilde{f}} \mathcal K_x^{(n)} , M_h a_{\tilde{g}} \mathcal K_x^{(n)}\rangle_{\mathfrak F}}{\lVert \mathcal K_x^{(n)} \rVert^2_{\mathfrak F}}                                                                                              \\
		                                               & =\mathfrak m( F(p_1 \otimes \Id) \tilde{f})(x) \overline{\mathfrak m\left( F(p_1 \otimes \Id) \tilde{g}\right)(x)} \frac{\langle \mathcal K_x^{(n-m_1)} , M_h \mathcal K_x^{(n-m_2)}\rangle_{\mathfrak F}}{\lVert \mathcal K_x^{(n)} \rVert^2_{\mathfrak F}} \\
		                                               & =\mathfrak m( F(p_1 \otimes \Id) \tilde{f})(x) \overline{\mathfrak m\left( F(p_1 \otimes \Id) \tilde{g}\right)(x)} \left( \frac{\sum_{k=n-m}^\infty h(k) w^{-2}(k)}{\sum_{k=n}^\infty w^{-2}(k)} \right).
	\end{align*}
	Note that the absence of $\mathcal K_x^{(m)}$-dependent normalization factors in the final result above is due to \cref{asmp:normalized_kernel}.

	It remains to show that
	$$\lim_{n \to \infty} \frac{\sum_{k=n-m}^\infty h(k) w^{-2}(k)}{\sum_{k=n}^\infty w^{-2}(k)} = h_\infty.$$
	The subexponential growth of $w(k)$ is critical for this to hold.
	For any $\varepsilon >0$ and $n$ sufficiently large $|h-h(n)|< \varepsilon$ and so we need only prove that
	$$\lim_{n \to \infty} \frac{\sum_{k=n-m}^{n-1} w^{-2}(k)}{\sum_{k=/n}^\infty w^{-2}(k)} = 0.$$

	To that end, since the sequence $w^{-2}(k) = e^{-2\tau k^p}$ is strictly decreasing, the finite sum of $m$ terms in the numerator is bounded from above by $m$ times its largest term,
	\[ \sum_{k = n - m}^{n-1} e^{-2\tau k^p} \le m e^{-2\tau (n-m)^p}. \]
	Moreover, we can bound the infinite sum from below using an integral,
	\[ \sum_{k=n}^\infty e^{-2\tau k^p} \ge \int_{n}^\infty e^{-2\tau x^p}\, dx. \]
	Because the function $f(x) = x^p$ is strictly concave for $p \in (0,1)$, its graph lies below its tangent line at $x=n$, giving
	\[ x^p \le n^p + p n^{p-1}(x-n). \]
	Multiplying by $-2\tau$ and exponentiating gives a lower bound for the integrand,
	\[ e^{-2\tau x^p} \ge e^{-2\tau n^p} e^{-2\tau p n^{p-1}(x-n)}. \]
	Evaluating the integral of this simpler lower bound yields
	\begin{equation*}
		\int_{n}^\infty e^{-2\tau x^p} \, dx \ge \int_{n}^\infty e^{-2\tau n^p} e^{-2\tau p n^{p-1}(x-n)} \, dx
		= \frac{n^{1-p}}{2\tau p} e^{-2\tau n^p}.
	\end{equation*}

	Substituting our bounds into the original ratio gives
	\[ 0 \le \frac{\sum_{k=n-m}^{n-1} w(k)}{\sum_{k=n}^\infty w(k)} \le \frac{m e^{-2\tau (n-m)^p}}{\frac{n^{1-p}}{2\tau p} e^{-2\tau n^p}} = \frac{2\tau p m}{n^{1-p}} e^{2\tau (n^p - (n-m)^p)}. \]
	It is now straightforward to show that
	$$\lim_{n \to \infty} \frac{2\tau p m}{n^{1-p}} e^{2\tau (n^p - (n-m)^p)} =0.$$
\end{proof}

We finally have the tools to justify \cref{def:state_space_embedding}.
\begin{proof}[Proof of well-posedness for \cref{def:state_space_embedding}]
	By the Banach--Alaoglu theorem, $(\varphi_x^{m})_{m=0}^\infty$ has a weak$^*$ cluster point $\Psi_x \in S(\mathcal A)$.
	Pick a convergent subsequence $\varphi_x^{m_i} \to \Psi_x$.
	Moreover, fix an element $a \in \mathcal A_\infty$ and put it into standard form, $a = \sum_j a_{\tilde f_j}^* M_{h_j} a_{\tilde g_j}$.
	By \eqref{eq:wccr5}, the only diagonal operators, $M_h$, that appear from commutation of creation and annihilation operators are the convergent sequences $h = \frac{q}{S(q)}$ or $h = \frac{1}{q}$.
	Since convergence of sequences is preserved under shifts, finite products, and finite sums by convergent sequences, all diagonal operators in the standard form of $a$ above, $M_{h_j}$, have convergent sequences $\lim_{n \to \infty} h_j(n) = h_{j,\infty}$.
	By \cref{prop:state_embedding},
	$$\Psi_x(a) = \sum_j h_{j,\infty} \mathfrak m( F(p_1 \otimes \Id) \tilde{f}_j)(x) \overline{\mathfrak m( F(p_1 \otimes \Id) \tilde{g}_j)(x)}.$$
	Since this is true for all weak$^*$ limit points of $(\varphi_x^m)_{m=0}^\infty$, all weak$^*$ cluster points agree on the dense $*$-subalgebra $\mathcal A_\infty \subset \mathcal A$.
	By continuity of states, there can only be one weak$^*$ cluster point, hence proving existence and uniqueness of $\lim_{m \to \infty} \varphi_x^m = \Psi_x$.
\end{proof}

\begin{thm}
	\label{thm:psi}
	Under \cref{asmp:main,asmp:alg,asmp:normalized_kernel} and provided that the weights $w(n)$ satisfy \eqref{ratio_limit}, the states $\Psi_x$, are $*$-homomorphisms from $\mathcal A_\infty$ to $\mathbb C$ making $\widehat{\Psi} \colon \mathcal A_\infty \to C_b^\infty(X)$, $(\widehat{\Psi} a)(x) = \Psi_x a$,  a $*$-algebra homomorphism.
	Furthermore, $\widehat{\Psi} \colon \mathcal A \to C_b(X)$ is a $C^*$-algebra homomorphism and the state space embedding $\Psi \colon X \to S(\mathcal A)$ is continuous in the weak$^*$ topology.
\end{thm}

\begin{proof}
	Let $a,b \in \mathcal A_\infty$, $\eta,\xi \in T(\mathfrak H_0)$.
	We will use $a a_{\xi} a_{\eta}^* b$ as a generic element in $\mathcal A_\infty$ that is not in standard form.

	Observe that
	$$\Psi_x(a a_{\xi} a_{\eta}^* b) = \lim_{m \to \infty}\left(\varphi_x^m(a a_{\eta}^* M_{q/S(q)} a_{\xi} b) +  \langle \xi,\eta \rangle_{\mathfrak F} \varphi_x^m(a  M_{1/q} b)\right).$$
	The presence of $M_{1/q}$ forces the second term to be zero since shifts of $1/q$ will appear in all monomials in the standard form of $a  M_{1/q} b$ and $\lim_{n \to \infty} 1/q(n) = 0$.
	Similarly, since $\lim_{n \to \infty} q(n)/S(q)(n) = 1$, we have $\Psi_x(a a_{\xi} a_{\eta}^* b) = \Psi_x(a a_{\eta}^* a_{\xi} b)$.
	This establishes commutativity of creation and annihilation operators inside $\Psi_x$.
	As a result, the formula
	$$\Psi_x(a_{\tilde f}^* a_{\tilde g}) = \mathfrak m( F(p_1 \otimes \Id) \tilde{f})(x) \overline{\mathfrak m( F(p_1 \otimes \Id) \tilde{g})(x)}$$
	clearly makes $\Psi_x$ a $*$-homomorphism.

	Since all vectors in $\mathfrak H_0$ are sequences of smooth functions, we also see that $\widehat{\Psi} \colon \mathcal A_\infty \to C^\infty(X)$.
	Since states are unit-norm linear functionals,
	$$\sup_{x \in X} |\Psi_x(a)| \leq \lVert a \rVert_\mathcal A \text{ for } a\in \mathcal A_\infty,$$
	it follows that $\widehat{\Psi} \colon \mathcal A \to C_b(X)$ and the embedding $\Psi\colon x \mapsto \Psi_x$ is continuous from $X$ to $S(\mathcal A)$ with the weak$^*$ topology.
\end{proof}

Another approach to this theorem passes through a two-sided ideal of $\mathcal A$ defined in \cref{app:ideal}. $\Psi_x$ also give states of a commutative quotient $C^*$-algebra of $\mathcal A$.

\subsection{The quantum pseudometric}

Consider the extended pseudometric \eqref{eq:metric} induced by the Dirac operator $-i\overline{\mathcal D}$ on the state space of the $*$-algebra $\mathcal A_1$ in the first spectral triple $(\mathcal A_n, \mathfrak F, -i \overline{\mathcal D})$ from \cref{thm:spectral_triple},
$$d_1(\varphi, \psi) = \sup\left\{ | \varphi(a) - \psi(a)| \colon \lVert [\mathcal D, a ] \rVert \leq 1, a \in \mathcal A_1 \right\}.$$
We may build a simple lower bound for $d_1$ on the states $\{ \varphi_x:  x\in X\}$ and $\{ \Psi_x:  x\in X\}$ by only testing against elements $a^*_{F(E_1)f}$ for $f \in \mathcal S$.
Using the operator norm bound  $\lVert a^*_g \rVert \leq C \lVert g \rVert_{\mathcal H}$ for a constant $C>0$ and $g \in \mathcal H$, we get
\begin{align*}
	d_1(\varphi_x,\varphi_y) & \geq \sup\{ |f(x)-f(y)| : \text{$\lVert a^*_{e_1 \otimes A f - e_2 \otimes Sf} \rVert \leq 1$, $f\in \mathcal S$}\}                                            \\
	                         & \geq \frac{1}{C} \sup \left\{ |f(x)-f(y)| : \text{$\sqrt{\lVert A f \rVert^2_\mathcal H +\lVert S f \rVert^2_\mathcal H} \leq 1$, $f \in \mathcal S$}\right\},
\end{align*}
and
\begin{align*}
	d_1(\Psi_x,\Psi_y) & \geq \sup\{ |f(x)-f(y)| : \text{$\lVert a^*_{e_1 \otimes A f - e_2 \otimes Sf} \rVert \leq 1$, $f\in \mathcal S$}\}                                            \\
	                   & \geq \frac{1}{C} \sup \left\{ |f(x)-f(y)| : \text{$\sqrt{\lVert A f \rVert^2_\mathcal H +\lVert S f \rVert^2_\mathcal H} \leq 1$, $f \in \mathcal S$}\right\}.
\end{align*}
Thus, if $\mathcal S$ separates points in $X$, the extended pseudometric induced by $d_1$ under either of the embeddings $\mathcal K$ or $\Psi$ is non-degenerate.
Under \cref{asmp:alg}, this will be the case when $\mathcal H$ is a dense subspace of $C_0(X)$, including the examples in \cref{sec:poly_examples}.
While similar considerations apply for the pseudometrics $d_n(\varphi_x,\varphi_y)$ and $d_n(\Psi_x,\Psi_y)$ associated with $(\mathcal A_n, \mathfrak F, - i \overline{\mathcal D})$, $ n > 1$, we cannot easily bound these metrics away from zero uniformly in the limit since the norm of $F(E_n)\tilde{f}$ grows like $\phi^{Nn} \lVert \tilde{f} \rVert_{F_w(\mathcal S)}$, where $\phi = (1+\sqrt{5})/2$ is the golden ratio and $N$ is the maximal grading of $\tilde f$ in $T(\mathcal S)$.
Instead we will analyze the metrics induced on $X$ from $(\mathcal A_n, \mathfrak F, - i \overline{\mathcal D})$ and the embedding $\Psi$ from \cref{def:state_space_embedding}.

We first observe that $\Psi$ intertwines the dynamical vector field $V$ and the action $[\mathcal D, \cdot]$ of the Dirac operator on the algebras $\mathcal A_n$.

\begin{thm}
	Under \cref{asmp:main,asmp:alg,asmp:normalized_kernel} and provided that the weights $w(n)$ satisfy \eqref{ratio_limit}, for every $n \geq 2$ and $a \in \mathcal A_n$, $\Psi([\mathcal D, a]) = V(\Psi a).$
\end{thm}

\begin{proof}
	\cref{prop:state_embedding} implies that $\Psi_x$ is well-defined on monomials whose diagonal operators $M_h$ are built from convergent sequences $h \in c(\mathbb N_0)$.
	For such monomials we have
	\begin{align*}
		\Psi_x\left([\mathcal D, a^*_{E_nf_1} \cdots a_{E_nf_k}^* M_h a_{E_ng_1} \cdots a_{E_ng_l}]\right) & = \Psi_x \left(a^*_{LE_nf_1} \cdots a_{E_nf_k}^* M_h a_{E_ng_1} \cdots a_{E_ng_l} \right)                  \\
		                                                                                                   & \quad +\Psi_x\left(a^*_{E_nf_1} a_{LE_nf_2}^* \cdots a_{E_nf_k}^* M_h a_{E_ng_1} \cdots a_{E_ng_l} \right) \\
		                                                                                                   & \quad + \ldots                                                                                             \\
		                                                                                                   & \quad + \Psi_x \left(a^*_{E_nf_1} \cdots a_{LE_nf_k}^* M_h a_{E_ng_1} \cdots a_{E_ng_l}\right)             \\
		                                                                                                   & \quad+\Psi_x \left(a^*_{E_nf_1} \cdots a_{E_nf_k}^* M_h a_{LE_ng_1} \cdots a_{E_ng_l}\right)               \\
		                                                                                                   & \quad + \Psi_x\left(a^*_{E_nf_1} \cdots a_{E_nf_k}^* M_h a_{E_ng_1}a_{LE_n g_2} \cdots a_{E_ng_l}\right)   \\
		                                                                                                   & \quad + \ldots                                                                                             \\
		                                                                                                   & \quad + \Psi_x\left(a^*_{E_nf_1} \cdots a_{E_nf_k}^* M_h a_{E_ng_1} \cdots a_{LE_ng_l}\right).
	\end{align*}
	Since, by \cref{thm:amplification}, $(p_1 \otimes \Id)LE_n f = Vf$ for $n \geq 2$, this simplifies to
	\begin{multline*}
		\Psi_x\left([\mathcal D, a^*_{E_nf_1} \cdots a_{E_nf_k}^* M_h a_{E_ng_1} \cdots a_{E_ng_l}]\right) \\
		\begin{aligned}
			 & =h_\infty \left( (Vf_1)(x) f_2(x) \cdots f_k(x) \overline{g_1(x) \cdots g_l(x)} + \cdots + f_1(x) f_2(x) \cdots f_k(x) \overline{g_1(x) \cdots (Vg_l)(x)} \right) \\
			 & = V\left[x \mapsto \Psi_x\left(a^*_{E_nf_1} \cdots a_{E_nf_k}^* M_h a_{E_ng_1} \cdots a_{E_ng_l}\right) \right],
		\end{aligned}
	\end{multline*}
	where $\lim_{n \to \infty} h(n) = h_\infty$.
	Restricting to sums of monomials that belong to $\mathcal A_n$ gives the result.
\end{proof}

We immediately have an upper bound on the quantum metrics $d_n$ since
$$\lVert V (\Psi a) \rVert_\infty = \lVert \Psi[\mathcal D, a] \rVert_\infty \leq \lVert \Psi \rVert \lVert [\mathcal D, a ] \rVert = \lVert [\mathcal D, a ] \rVert,$$
giving
\begin{align*}
	d_n(\Psi_x,\Psi_y) & = \sup\left\{ | \Psi_xa - \Psi_ya|:\; \text{$\lVert [\mathcal D, a ] \rVert <1$, $a \in \mathcal A_n$} \right\}                   \\
	                   & \leq \sup\left\{ | f(x) - f(y)|: \; \text{$\lVert Vf \rVert_\infty <1$, $f \in \alg^*(\mathcal S) \subset C^\infty(X)$} \right\}.
\end{align*}
In general, this upper bound may take on infinite values, but if $x$ and $y$ lie in the same dynamical orbit, it is bounded above by the optimal Lipschitz constant with respect to $V$ and for functions in $\alg^*(\mathcal S)$.

\begin{cor}
	\label{cor:quantum_metrics}
	If $x$ and $y$ lie in the same orbit under $\Phi^t$, then $d_n(\Psi_x, \Psi_y)$ is less than the smallest time $|t|$ such that $x=\Phi^t(y)$.
\end{cor}

\begin{proof}
	Suppose that the orbit containing $x$ and $y$ is isomorphic to $\mathbb R$.
	Then, $\Phi_x \colon \mathbb R \to \{\Phi^t(x) : t \in \mathbb R\}$, $\Phi_x(t) = \Phi^t(x)$, is a diffeomorphism.
	For $f \in C^\infty(X)$, we have $\frac{d}{dt} f\circ \Phi_x (t) = V_{\Phi^t(x)} f$ and thus
	\begin{align*}
		d_n(\Psi_x,\Psi_y) & \leq \sup\left\{ | f(x) - f(y)|: \; \text{$\lVert Vf \rVert_\infty <1$, $f \in \alg^*(\mathcal S) \subset C^\infty(X)$} \right\}                      \\
		                   & \leq \sup\left\{ | f(x) - f(y)|: \; \text{$\left\lVert \frac{d}{dt} f\circ \Phi_x(t) \right\rVert_{\infty} <1$, $f \in C^\infty(X)$} \right\}         \\
		                   & \leq \sup\left\{ |g(0) - g(\Phi_x^{-1}(y))|: \; \text{$\left\lVert \frac{d}{dt} g(t) \right\rVert_{\infty} <1$, $g \in C^\infty(\mathbb R)$} \right\} \\
		                   & =|\Phi_x^{-1}(y)| = \lvert t \rvert.
	\end{align*}
	The claim can be proved similarly when the orbit is diffeomorphic to the unit circle and it holds trivially when the orbit it a point.
\end{proof}

The following is an analog of \cref{thm:state_evolution_convergence}, utilizing the embedding $\Psi$ instead of $\mathcal K$.

\begin{thm}\label{thm:spectral_triple_uniform_convergence}
	Under \cref{asmp:main,asmp:alg,asmp:normalized_kernel} and provided that the weights $w(n)$ satisfy \eqref{ratio_limit}, let $\tilde{f} \in T(\mathcal S)$ be a lift of $f \in \mathcal H$ (i.e., $\mathfrak m\tilde{f} = f$).
	Then,
	\begin{equation*}
		\Psi_x\left(a^*_{F(E_n)\tilde{f}}\right) = f(x)
	\end{equation*}
	and for $|t| \leq R/3$
	$$\lim_{n\to \infty} \Psi_x \left(a^*_{F(E_n)\tilde{f}}(t) \right)=\lim_{n\to \infty} \Psi_x \left(\Adj_{e^{t\mathcal D}} a^*_{F(E_n)\tilde{f}} \right) = (U_tf)(x),$$
	where $R$ is the joint radius of analyticity for every function in $\mathcal S$ with respect to $V$ and $V^*$.
	Furthermore, the convergence is uniform over $x \in X$.
\end{thm}

\begin{proof}
	The result follows from unpacking the various definitions above.
	First, we have
	$$\varphi_x^m\left(a^*_{F(E_n)\tilde{f}}\right) = \frac{\left\langle \mathcal K_x^{(m)} , a^*_{F(E_n)\tilde{f}} \mathcal K_x^{(m)} \right\rangle_{\mathfrak F}}{\lVert \mathcal K_x^{(m)} \rVert_{\mathfrak F}^2}= \frac{\left\langle a_{F(E_n)\tilde{f}} \mathcal K_x^{(m)} ,  \mathcal K_x^{(m)} \right\rangle_{\mathfrak F}}{\lVert \mathcal K_x^{(m)} \rVert_{\mathfrak F}^2} = f(x) \frac{\left\langle \mathcal K_x^{(m-1)} ,  \mathcal K_x^{(m)} \right\rangle_{\mathfrak F}}{\lVert \mathcal K_x^{(m)} \rVert_{\mathfrak F}^2} = f(x),$$
	and the claim follows since the last expression above is independent of $m$.
	For the second claim, using~\eqref{eq:alt_coherent_states}, we get
	\begin{align*}
		\varphi_x^k \left(a^*_{F(E_n)\tilde{f}}(t) \right) & = \varphi_x^k \left(a^*_{F(e^{Lt})F(E_n)\tilde{f}} \right)                                                                                                                                                              \\
		                                                   & = \mathfrak m \left(F(p_1 \otimes \Id)F(e^{Lt})F(E_n)\tilde{f} \right)(x)  \frac{\left\langle \mathcal K_x^{(k-1)} ,  \mathcal K_x^{(k)} \right\rangle_{\mathfrak F}}{\lVert \mathcal K_x^{(k)} \rVert_{\mathfrak F}^2} \\
		                                                   & = \mathfrak m \left(F(p_1 \otimes \Id)F(e^{Lt})F(E_n)\tilde{f} \right)(x).
	\end{align*}
	Again, the last expression is independent of $k$, so the first limit is superfluous.
	Then, by \cref{thm:tensor_koopman},
	\begin{equation*}
		\left| \lim_{k\to \infty} \varphi_x^k \left(a^*_{F(E_n)\tilde{f}}(t)\right) - (U_tf)(x)\right| \leq \lVert \Delta \rVert \left\lVert \mathfrak m\left( F(p_1 \otimes \Id) e^{\mathcal D t} F(E_n) \tilde{f} \right) - U_t \mathfrak m(\tilde{f}) \right\rVert_\mathcal H
	\end{equation*}
	converges to zero as $n\to \infty$, uniformly in $x \in X$ since the right-hand side is independent of $x$.
\end{proof}

In contrast to \cref{thm:state_evolution_convergence}, \cref{thm:spectral_triple_uniform_convergence} applies more broadly to elements of $\mathcal A_n$ since $\widehat{\Psi}$ is a $*$-homomorphism.

\begin{cor}
	Let $f \in \alg^*(\mathcal S) \subset C_b(X)$ and fix a quantum observable $O \in \mathcal A_1$ such that $\Psi(O) = f$, of which there are many choices.
	Put $O$ into standard form with $\tilde f_1, \ldots, \tilde f_N, \tilde g_1, \ldots, \tilde g_N \in T(\mathcal S)$ and $h_1, \ldots, h_N \in c(\mathbb N_0)$, $O= \sum_{j=1}^N a_{F(E_1)\tilde{f}_j}^* M_{h_j} a_{F(E_1)\tilde{g}_j}$.
	Then, defining
	$$O_{\tilde{f}_j, h_j, \tilde{g}_j,n} = \sum_{j=1}^N a_{F(E_n)\tilde{f}_j}^* M_{h_j} a_{F(E_n)\tilde{g}_j} \in \mathcal A_n,$$
	we have $\Psi(O_{\tilde{f}_j, \tilde{g}_j,n})= f$ for all $n \geq 1$ and
	$$\lim_{n \to \infty} \Psi_x \left( \Adj_{e^{ t \mathcal D}} O_{\tilde{f}_j, \tilde{g}_j,n} \right) = (U_tf)(x),$$
	where $\lvert t \rvert < R/3$ and the convergence is uniform with respect to $x \in X$.
\end{cor}

\subsection{Compact resolvent}

The weak spectral triples $(\mathcal A_n, \mathfrak F, \mathcal D)$ fail to be examples of Connes' spectral triples \cite{Connes94} even in the simplest examples.
Let $\Lambda \subset i \mathbb C$ be the point spectrum of $L$.
Then, $\mathcal D$ will have eigenvalues of the form
$$\{0\} \cup \{n_1 \lambda_1 +\cdots+ n_k \lambda_k: n_j \in \mathbb N, \lambda_j \in \Lambda, k \in \mathbb N_0\}.$$
If $\Lambda$ contains an eigenvalue with positive real part and an eigenvalue with negative real part, then $\mathcal D$ will fail to have compact resolvent; see \cref{rk:diagonalizable_l} for an example.
In this subsection, we outline a strategy to obtain true spectral triples from $(\mathcal A_n,\mathfrak F, \mathcal D)$.

Instead of representing $\mathcal A_n$ on the entire weighted Fock space $\mathfrak F$, we may reduce to a subrepresentation space of $\mathcal A_n$,
$$\mathfrak F_n = F_w(\spn\{e_1,\cdots e_n\} \otimes \mathcal S) \subset \mathfrak F, \quad \mathcal A_n(\mathfrak F_n) \subset \mathfrak F_n.$$
The orthogonal projection onto this subrepresentation space $\mathfrak F_n$ is given by $F(Q_n)$, where
$$Q_n = \sum_{k=1}^n e_{k,k} \otimes Q_\mathcal S \colon l^2(\mathbb N) \otimes \mathcal H \to \spn\{e_1,\cdots,e_n\} \otimes \mathcal S \subset l^2(\mathbb N) \otimes \mathcal H,$$
and $Q_\mathcal S \colon \mathcal H \to \mathcal S \subset \mathcal H$ is the orthogonal projection onto $\mathcal S$.
Since the annihilation and creation operators of $\mathcal A_n$ only use vectors from $\spn\{e_1,\cdots,e_n\} \otimes \mathcal S$, $F(Q_n) a F(Q_n) = a F(Q_n)$.
Denote the image of this sub-representation inside $B(\mathfrak F_n)$ by $\mathcal B_n =  F(Q_n) \mathcal A_n|_{\mathfrak F_n}$.

We may conjugate the Dirac operator by $F(Q_n)$ to define a new operator, $\mathcal D_n = F(Q_n) \mathcal D F(Q_n)$, that satisfies
\begin{gather*}
	[\mathcal D_n, a] = [F(Q_n) \mathcal D  F(Q_n),a] = F(Q_n) [ \mathcal D,a]F(Q_n), \quad a \in \mathcal B_n, \\
	[\mathcal D_n, a_{E_n f}] = a_{Q_nLE_nf}F(Q_n) = a_{Q_nLE_nf} \in B(\mathfrak F_n), \quad f \in \mathcal S.
\end{gather*}
To build a state space embedding analogous to $\Psi$ we need the projection of kernel sections to $\mathcal S$ to have a constant norm.

\begin{asmp}\label{asmp:projected_kernel_sections}
	The RKHS inner product on $\mathcal H$ is normalized so that $\lVert Q_\mathcal S k_x \rVert_\mathcal H = 1$ for all $x \in X$.
\end{asmp}
This is not an unreasonable restriction as the examples provided in \cref{sec:poly_examples} satisfy this extra condition whenever $\mathcal S$ is the linear span of a finite number of Fourier modes on $\mathbb T^d$ after rescaling the inner product of $\mathcal H$.
Of course, it should be kept in mind that, in general, \cref{asmp:normalized_kernel,asmp:projected_kernel_sections} cannot be satisfied simultaneously so the constructions described in this subsection require a different kernel normalization than those in \cref{sec:psi_embedding}.

Provided \cref{asmp:projected_kernel_sections} holds, the exact same construction of $\Psi$ from \cref{thm:psi} now applies to $\mathcal A_n$ acting on the weighted Fock space $\mathfrak F_n$.
Define the vectors
$$\mathcal K_x^{(m),\mathcal S} = \sum_{n=m}^\infty w(n)^{-2} (e_1 \otimes Q_\mathcal S(k_x))^{\vee n} \in \mathfrak F, \quad m \in \mathbb N,$$
and their corresponding vector states $$\varphi_x^{m,\mathcal S}(a) = \frac{\langle \mathcal K_x^{(m),\mathcal S} , a \mathcal K_x^{(m),\mathcal S}\rangle_{\mathfrak F}}{\lVert \mathcal K_x^{(m),\mathcal S} \rVert^2_{\mathfrak F}}.$$
Then just as in \cref{def:state_space_embedding}, let $\displaystyle\Psi_x^\mathcal S(a) = \lim_{m \to \infty} \varphi_x^{m,\mathcal S}(a).$
The states $\Psi_x^\mathcal S$, are $*$-homomorphisms of $\mathcal B_n$ by the same proof that $\Psi_x$ is a $*$-homomorphism of $\mathcal A_n$.

Since $[\mathcal D_n,a^*_{E_n f}] = a^*_{Q_n L E_n f}$, we also have
$$\Psi_x^\mathcal S ([\mathcal D_n,a^*_{E_n f}]) = (p_1 \otimes \Id) Q_n L E_n f = Q_\mathcal S (Vf)(x)$$
and
\begin{align*}
	\Psi_x^\mathcal S ([\mathcal D_n,a^*_{F(E_n) \tilde f}a_{F(E_n)\tilde g}]) & = \mathfrak m \left(F(p_1 \otimes \Id) F(Q_n) F(L) F(E_n) \tilde f\right) \overline{\mathfrak m \left(F(p_1 \otimes \Id) F(Q_n) F(L) F(E_n) \tilde g\right)} \\
	                                                                           & = \mathfrak m \left( F(Q_\mathcal S) \Gamma(V) \tilde f \right)\overline{\mathfrak m \left( F(Q_\mathcal S) \Gamma(V) \tilde g \right)},
\end{align*}
where $\Gamma(V) = \sum_{m=1}^\infty \Id^{\vee m-1} \vee V$.

Now the core question at hand is resolvent compactness of $\mathcal D_n$.
As an unbounded operator on $\mathfrak F_n$,
$$\mathcal D_n = \sum_{m=1}^\infty \Id^{\vee m-1} \vee Q_nL Q_n.$$
Since $Q_nL Q_n$ is a bounded skew-adjoint operator on a finite-dimensional space, it is diagonalizable.
Let $\Lambda$ be its set of eigenvalues with multiplicity.
Then $\mathcal D_n$ has as the complete set of eigenvalues
$$\sigma_p(\mathcal D_n) =\{0\} \cup \{n_1 \lambda_1 +\cdots+ n_k \lambda_k: n_j \in \mathbb N, \lambda_j \in \Lambda, k \in \mathbb N_0\}.$$
Then $\mathcal D_n$ will have compact resolvent precisely when all eigenvalues in $\Lambda$ lie strictly above the real line or strictly below.
Alternatively, resolvent compactness of $\mathcal D_n$ occurs when the following operator is either strictly positive or strictly negative:
\begin{equation*}
	i Q_n L Q_n = i\left[
		\begin{array}{cccccc}
			Q_\mathcal S A Q_\mathcal S & Q_\mathcal S S Q_\mathcal S  & 0                          & 0                         & \cdots                       & 0                           \\
			-Q_\mathcal S SQ_\mathcal S & Q_\mathcal SA Q_\mathcal S   & Q_\mathcal S SQ_\mathcal S & 0                         & \cdots                       & 0                           \\
			0                           & - Q_\mathcal S SQ_\mathcal S & Q_\mathcal SAQ_\mathcal S  & Q_\mathcal SSQ_\mathcal S & \cdots                       & 0                           \\
			0                           & 0                            & -Q_\mathcal SSQ_\mathcal S & Q_\mathcal SAQ_\mathcal S & \cdots                       & \vdots                      \\
			\vdots                      & \vdots                       & \vdots                     & \vdots                    & \ddots                       & Q_\mathcal S S Q_\mathcal S \\
			0                           & 0                            & 0                          & \cdots                    & -Q_\mathcal S S Q_\mathcal S & Q_\mathcal S A Q_\mathcal S
		\end{array}\right]
\end{equation*}

\begin{prop}
	$iQ_nLQ_n$ is strictly positive iff $iQ_\mathcal SAQ_\mathcal S - 2\cos\left(\frac{\pi k}{n+1}\right)Q_\mathcal SSQ_\mathcal S > 0$ for all $k=1,\cdots,n$.
\end{prop}
\begin{proof}
	The tri-diagonal Toeplitz matrix $M=\sum_{k=1}^{n-1} e_{k,k+1} - e_{k+1,k}$ is well-known to have eigenvalues $2i\cos\left(\frac{\pi k}{n+1}\right)$ for $k=1,\cdots,n$.
	Let $M=R\Lambda R^{-1}$ be its eigendecomposition.
	Since $iQ_nLQ_n = i\Id\otimes Q_\mathcal S AQ_\mathcal S + iM\otimes Q_\mathcal S SQ_\mathcal S$ can be block-diagonalized by conjugation with $R\otimes \Id$, positivity occurs exactly when
	$$iQ_\mathcal S AQ_\mathcal S - 2\cos\left(\frac{\pi k}{n+1}\right) Q_\mathcal S SQ_\mathcal S >0$$
	for all $k=1,\cdots,n$.
\end{proof}
By taking convex combinations a sufficient condition for all $n \in \mathbb N$ is
$$iQ_\mathcal SAQ_\mathcal S - 2Q_\mathcal SSQ_\mathcal S > 0 \quad \text{ and } \quad iQ_\mathcal SAQ_\mathcal S + 2Q_\mathcal SSQ_\mathcal S > 0.$$

\begin{thm}
	With notation as above and under \cref{asmp:main}, fix a finite-dimensional subspace $\mathcal S \subset \mathcal H_\infty$ and an integer $n \in \mathbb N$.
	If $iQ_\mathcal SAQ_\mathcal S - 2\cos\left(\frac{\pi k}{n+1}\right)Q_\mathcal SSQ_\mathcal S$ is strictly positive (resp.\ strictly negative) for all $k=1,\ldots,n$ then the triple $(\mathcal B_n,\mathfrak F_n,\mathcal D_n)$ is a spectral triple.
\end{thm}

\begin{exmp*}
	Fix $X = \mathbb T^2$ and consider the vector field
	$$\vec V(\theta) = (1+\alpha\cos(\theta_1+\theta_2)) \frac{\partial\ }{\partial \theta_1} + (1+\beta \cos(\theta_1))\frac{\partial\ }{\partial \theta_2}.$$
	with $\alpha,\beta \in \mathbb R$.
	Then with the cutoff $\mathcal S_N = \text{span}\{e^{i(k\theta_1 + l \theta_2)} : 1 \leq k,l \leq N \}$, we may ask for what values of $\alpha$ and $\beta$ is
	$$iQ_{\mathcal S_N}AQ_{\mathcal S_N} - 2Q_{\mathcal S_N}SQ_{\mathcal S_N} < 0 \quad \text{and} \quad iQ_{\mathcal S_N}AQ_{\mathcal S_N} + 2Q_{\mathcal S_N}SQ_{\mathcal S_N} < 0,$$
	where the adjoint of $V$ is taken on $L^2(\mathbb T^2)$ with respect to the Haar measure.
	This will differ from the adjoint in an RKHS $\mathcal H$ by an arbitrarily small perturbation since the kernel
	\begin{equation*}
		k(\theta, \theta') = \sum_{m \in \mathbb Z^2} \omega^{-2}(m) e^{i m \cdot (\theta-\theta')}
	\end{equation*}
	could be chosen with $\omega^{-2}(m)$ subconvolutive but $\omega^{-2}(k,l) \approx 1$ for $1 \leq k,l\leq N$.
	Suppose $\alpha \approx 1 \approx \beta$.
	Then for $|\alpha| <1$ and $|\beta|<1$ and all $N$ strict negativity appears to hold from numerical computations.
	For $|\alpha|>1$ and $|\beta|>1$ numerical computations indicate that strict negativity can fail for sufficiently large $N$.
\end{exmp*}

\section{Kernel mean embedding and the statistical evolution of observables}
\label{sec:discussion}

The kernel mean embedding of probability measures \cite{BerlinetThomasAgnan04,GrettonEtAl06,SmolaEtAl07,SriperumbudurEtAl10} is a procedure for representing Borel probability measures on topological spaces by RKHS vectors given by vector-valued integration of the kernel sections,
\begin{equation}
	\label{eq:kme}
	h = \int_X k(\cdot, x) \, d\mu(x).
\end{equation}
Here, $k\colon X \times X \to \mathbb R$ is a measurable, bounded, positive-definite kernel on $X$, $\mu$ is a Borel probability measure, and the RKHS vector $h \in \mathcal H$ is defined uniquely by $h(y) = \int_X k(y,x) \, d\mu(x)$ for all $y \in X$.
Using $h$, we may compute expectation values of observables $f \in \mathcal H$ using RKHS inner products,
\begin{equation*}
	\int_X f \, d\mu = \langle h, f\rangle_{\mathcal H}.
\end{equation*}

Let $\mathcal P(X)$ denote the set of Borel probability measures on $X$.
The kernel $k$ is said to be characteristic if the map $\mu \mapsto h$ from~\eqref{eq:kme} is injective on $\mathcal P(X)$.
It is known \cite{SriperumbudurEtAl10} that the kernel mean embedding induced by a characteristic kernel metrizes the weak$^*$ topology on $\mathcal P(X)$.
Moreover, if $\mathcal H$ lies dense in $C_0(X)$ then its reproducing kernel is characteristic.
This includes the RKHA examples on tori from \cref{sec:poly_examples}.

Thinking of a point $x \in X$ as the corresponding Dirac probability measure $\delta_x \in \mathcal P(X)$, \cref{thm:state_evolution_convergence,thm:spectral_triple_uniform_convergence} may be thought of as providing approximations for the evolution of expectation values of observables of dynamical systems with respect to Dirac probability measures.
We may extend these results to expectation values with respect to arbitrary probability measures in $\mathcal P(X)$ by defining kernel mean embeddings associated with the state space embeddings $x \mapsto \varphi_x $ and $x \mapsto \psi_x$ from \cref{sec:state_space_embedding_and_metric}.
Reusing notation from that section, and making the same assumptions as in that section, we define $\mathcal K \colon \mathcal P(X) \to S(\mathcal A)$ and $\Psi \colon \mathcal P(X) \to S(\mathcal A)$, where $\mathcal K(\mu) \equiv \mathcal K_\mu$ and $\Psi(\mu) \equiv \Psi_\mu$ are defined weakly via the integrals
\begin{equation*}
	\mathcal K_\mu a = \int_X \varphi_x(a) \, d\mu(x), \quad \Psi_\mu(a) = \int_X \psi_x(a)\, d\mu(x).
\end{equation*}
Defining the time-evolved states
$$\mathcal K_\mu^{(t)} = \mathcal K_\mu \circ \Adj_{e^{t\mathcal D}}, \quad \quad \Psi_\mu^{(t)} = \Psi_\mu \circ \Adj_{e^{t\mathcal D}},$$
we may then reconstruct the evolution of expectation values of observables using either scheme:
$$\int_X (U_tf)(x) \, d\mu(x) = \lim_{n \to \infty} \left(\mathcal K_\mu^{(t)} O_{\tilde{f},n}\right) = \lim_{n \to \infty} \Psi_{\mu}^{(t)}(O_{\tilde{f}_j,h_j,\tilde{g}_j,n}),$$
whenever $f(x) = \mathcal K_{\delta_x}(O_{\tilde{f},n}) = \Psi_{\delta_x}(O_{\tilde{f}_j,h_j,\tilde{g}_j,n})$.

\subsection{Discussion}

To draw an analogy with the Koopman--von Neumann method for classical statistical dynamics, we have replaced the $L^2(X)$ Hilbert space associated with a volume measure on the manifold with a weighted Fock space, $\mathfrak F$, generated by the amplification $\ell^2(\mathbb N) \otimes \mathcal H$ of a finitely generated RKHA on $X$ that satisfies joint analyticity conditions with respect to the dynamical generator, $V$, and its adjoint.
We have also replaced the Koopman--von Neumann wavefunctions $\psi_t \in L^2(X)$ (see \cref{sec:introduction}) with quantum states, $\mathcal K_\mu$ or $\Psi_\mu$, of a $C^*$-algebra $\mathcal A \subset B(\mathfrak F)$.
Doing so, has allowed us to reconstruct the statistical evolution of observables in $\mathcal H$ through a unitary evolution on $\mathfrak F$ generated by a Dirac operator, $-i \overline{\mathcal D}$, for Borel probability measures $\mu \in \mathcal P(X)$ (not necessarily absolutely continuous with respect to the volume measure).
This approach has also allowed us to build a family of weak spectral triples and associated quantum metrics on $X$, the latter behaving as Lipschitz metrics under the dynamical vector field.

Compiled in \cref{tab:comparison} is a comparison between Koopman von Neumann mechanics and the schemes developed here.
There, the Hilbert space state vector $\psi_0 = \sqrt{p_0} \in L^2(X)$ associated with a probability density $p_0 \in L^1(X)$ is seen to play an analogous role as the vectors $\frac{\mathcal K_x}{\lVert\mathcal K_x \rVert}_{\mathfrak F}$ underpinning the states $\mathcal K_{\delta_x}$ under the first scheme.
Observe that no state vector is given for the second scheme since the states $\Psi_{\delta_x}$ are not vector states for the representation of $\mathcal A_n$ on $\mathfrak F$.
Moreover, neither scheme produces vector states when $\mu$ is not a Dirac measure.
Using the GNS construction on $\mathcal A_n$ may provide state vectors for $\Psi$, but we do not explore that here.

\begin{table}
	\caption{\label{tab:comparison}Basic objects used in the Koopman--von Neumann formulation of classical statistical dynamics and the measure-free schemes studied in this paper.}
	\begin{tabular}{cccc}
		                     & Koopman--von Neumann  & \multicolumn{2}{c}{Measure-Free Koopman--von Neumann}                                                                             \\
		\cline{2-4}
		                     &                       & Scheme 1                                                                                   & Scheme 2                             \\
		\cline{3-4}
		Unitary Evolution    & $e^{t\tilde{A}}$      & $e^{t\mathcal D}$                                                                          & $e^{t\mathcal D}$                    \\
		Observable           & $f$                   & $O_{\tilde{f},n}$                                                                          & $O_{\tilde{f}_j,h_j, \tilde{g}_j,n}$ \\
		Initial state        & $p_0$                 & $\mathcal K_{\mu}$                                                                         & $\Psi_{\mu}$                         \\
		State vector         & $\psi_0 = \sqrt{p_0}$ & $\frac{\mathcal K_x}{\lVert \mathcal K_x \rVert_{\mathfrak F}}$                            & N/A                                  \\
		Evolved state        & $p_t$                 & $\mathcal K_{\mu}^{(t)}$                                                                   & $\Psi_{\mu}^{(t)}$                   \\
		Evolved state vector & $\psi_t = \sqrt{p_t}$ & $e^{-t\overline{\mathcal D}}\frac{\mathcal K_x}{\lVert \mathcal K_x \rVert_{\mathfrak F}}$ & N/A                                  \\
		\hline
	\end{tabular}
\end{table}

\section*{Acknowledgments}

Dimitrios Giannakis acknowledges support from the U.S.\ Department of Defense, Basic Research Office under Vannevar Bush Faculty Fellowship grant N00014-21-1-2946 and the U.S.\ Department of Energy under grant DE-SC0025101.
Michael Montgomery was supported as a postdoctoral fellow from the first grant.

\begin{appendices}

	\section{Reproducing kernel Hilbert algebras}
	\label{app:rkha}

	In this appendix, we summarize a few properties of RKHAs as defined in \cref{def:rkha}.
	Further details on these topics can be found in \cite{DasGiannakis23,DasEtAl23,GiannakisMontgomery25}.

	\subsection{Banach algebra norm.}
	An RKHA $\mathcal H$ on a set $X$ can be equipped with a norm $\vertiii{\cdot}$ induced by the operator norm on multiplication operators in $B(\mathcal H)$,
	\begin{equation*}
		\vertiii{f} = \sup_{g \in \mathcal H \setminus \{0\}} \frac{\lVert f g \rVert_{\mathcal H}}{\lVert g \rVert_{\mathcal H}}.
	\end{equation*}
	By properties of operator norms, $\vertiii{\cdot}$ is a submultiplicative (Banach algebra) norm,
	\begin{equation*}
		\vertiii{fg} \leq \vertiii{f} \vertiii{g}, \quad \forall f, g \in \mathcal H.
	\end{equation*}
	If $\mathcal H$ is unital, with unit $1_X(x) = 1$, this norm is equivalent to the Hilbert space norm,
	\begin{equation*}
		\frac{1}{\lVert 1_X \rVert_{\mathcal H}} \lVert f \rVert_{\mathcal H} = \frac{\lVert f 1_X \rVert_{\mathcal H}}{\lVert 1_X \rVert_{\mathcal H}} \leq \vertiii{f} = \sup_{g \in \mathcal H \setminus \{0\}} \frac{\lVert \Delta^*(f \otimes g)\rVert_{\mathcal H \otimes \mathcal H}}{\lVert g \rVert_{\mathcal H}} \leq \lVert \Delta \rVert \lVert f \rVert_{\mathcal H}.
	\end{equation*}
	If $\mathcal H$ is non-unital, only the upper bound holds and $\vertiii{\cdot}$ induces a coarser topology than the Hilbert space norm, $\lVert \cdot \rVert_{\mathcal H}$.

	\subsection{Spectrum and cospectrum}
	\label{app:spec_cospec}

	Using standard definitions from the theory of abelian Banach algebras, we define the spectrum, $\sigma(\mathcal H)$, of an RKHA $\mathcal H$ on a set $X$ as the set of nonzero, multiplicative linear functionals on $\mathcal H$ called characters, equipped with the weak$^*$ topology from $\mathcal H$.
	All characters are automatically continuous so $\sigma(\mathcal H)$ is a subset of $\mathcal H^*$.
	If $\mathcal H$ is unital, it is a weak$^*$ compact set.

	It is immediate that $\sigma(\mathcal H)$ contains all nonzero evaluation functionals $\langle k_x, \cdot \rangle_{\mathcal H}$, $x \in X$.
	This implies that the spectrum separates points in $\mathcal H$ and we can canonically identify $\mathcal H$ with a subspace of $C(\sigma(\mathcal H))$ under the Gelfand transform, $f \in \mathcal H \mapsto \hat f \in C(\sigma(\mathcal H))$ with $\hat f(\phi) = \phi(f)$.

	Under appropriate conditions on the reproducing kernel $k$, $\sigma(\mathcal H)$ is homeomorphic to $X$ as a topological space, with the homeomorphism given by the feature canonical map from~\eqref{eq:canonical_feature_map}, $x \mapsto \langle \kappa(x), \cdot \rangle_{\mathcal H}$.
	Examples include translation-invariant kernels on LCAs satisfying the Gelfand--Raikov--Shilov condition (see equation (GRS) in \cite{GiannakisMontgomery25}).
	The kernels on $X = \mathbb T^d$ induced from the subexponential weights in~\eqref{eq:subexp_weights} belong in this class.
	In other cases, $\sigma(\mathcal H)$ may contain additional elements besides the evaluation functionals; see \cite{DasEtAl23} for examples in that direction based on kernels on $\mathbb T^1$ that fail to meet the GRS condition.
	Still, for any compact subset $X \subset \mathbb R^d$ there exists an RKHA $\mathcal H \subseteq C(X)$ such that $X \cong \sigma(\mathcal H)$; see \cite{GiannakisMontgomery25}*{Theorem~3.8}.

	The cospectrum of an RKHA $\mathcal H$, denoted as $\sigma_\text{co}(\mathcal H)$ is the set of all nonzero elements $f \in \mathcal H$ satisfying $\Delta f = f \otimes f$.
	It is equipped with the weak topology from $\mathcal H$.
	By the Riesz representation theorem, the spectrum and cospectrum are homeomorphic as topological spaces, as $ f \in \sigma(\mathcal H)$ if and only if $\langle f, \cdot \rangle_{\mathcal H} \in \sigma(\mathcal H)$.

	\subsection{RKHA examples}
	\label{app:rkha_examples}

	On LCAs such as $\mathbb R^d$ or $\mathbb T^d$, RKHAs can be built as Fourier images of weighted $L^2$ spaces on the dual group, associated with subconvolutive weights.
	The examples on $\mathbb T^d$ discussed in \cref{sec:poly_examples} fall in that category.
	In fact, it can be shown \cite{GiannakisMontgomery25}*{Theorem~2.3} that the inverse square summability and subconvolutivity conditions in \eqref{eq:subconv} (generalized to an arbitrary LCA) are necessary and sufficient for an RKHS with translation-invariant continuous kernel on an LCA to be an RKHA.
	For context, we note that subconvolutive weight functions have a rich history of study in harmonic analysis; see e.g., work of Feichtinger~\cite{Feichtinger79} which initiated many developments in the field.
	A survey of weight functions in harmonic analysis can be found in \cite{Grochenig07}.

	Another approach for building RKHAs on arbitrary sets is through weighted power series of bounded positive-definite kernels.
	Let $\mathbb H$ be a Hilbert space and $w: \mathbb N_0 \to \mathbb R_{>0}$ a weight satisfying
	\begin{equation}
		\label{eq:subconv_fock_space_weights}
		w(0) = 1, \quad w^{-2} \in \ell^1(\mathbb N_0), \quad  (w^{-2} * w^{-2})(n) \leq C w^{-2}(n).
	\end{equation}
	Let also $R_w$ be the radius convergence of the series $\sum_{n=0}^\infty w^{-2}(n) z^n$, $z \in \mathbb C$, where $R_w \geq 1$ since $w^{-2} \in \ell^1(\mathbb N_0)$.
	Then, given a bounded feature map $\kappa \colon X \to \mathbb H$, normalized such that $\sup_{x\in X}\lVert \kappa(x) \rVert_{\mathbb H} \leq R_w$, the function $ \tilde k \colon X \times X \to \mathbb C$ defined as
	\begin{equation}
		\label{eq:power_series_kernel}
		\tilde k(x, y) = \sum_{n=0}^\infty w^{-2}(n) \langle \kappa(x), \kappa(y) \rangle_{\mathbb H}^n
	\end{equation}
	is the reproducing kernel for a unital RKHA, $\mathcal H_{\tilde k}$, on $X$, where the series in~\eqref{eq:power_series_kernel} converges absolutely and uniformly over $(x, y) \in X \times X$.
	Moreover, the cospectrum of $\mathcal H_{\tilde k}$ consists of functions of the form
	\begin{equation*}
		f(x) = \sum_{n=0}^\infty w^{-2}(n) \langle \kappa(x), u \rangle_{\mathbb H}^n,
	\end{equation*}
	where $u$ is any nonzero vector of norm $\lVert u \rVert_{\mathbb H} \leq R_w$.
	In other words, the spectrum and cospectrum of $\mathcal H_{\tilde k}$ are homeomorphic to a closed ball of radius $R_w \geq 1$ in the underlying Hilbert space $\mathbb H$, the latter equipped with the weak topology.
	If $\mathbb H = \mathcal H$ is an RKHS on $X$ and $\kappa \colon X \to \mathcal H$ is chosen as the canonical feature map (see \eqref{eq:canonical_feature_map}), then $\tilde k(x, y) = \sum_{n=0}^\infty w^{-2}(n) k(x,y)^n$ and $\mathcal H_{\tilde k}$ contains $\mathcal H$ as a dense subspace.
	See \cite{GiannakisEtAl26} for further details.

	\section{Weighted symmetric Fock space}
	\label{app:Fock_RKHA}

	Let $F(\mathbb H) = \overline{T^\otimes(\mathbb H)}$ denote the full Fock space, which is closed under the inner product
	\begin{equation}
		\label{eq:fock_innerp}
		\begin{gathered}
			\langle \alpha, \beta \rangle_{F(\mathbb H)} = \bar \alpha \beta, \quad \alpha, \beta \in \mathbb C, \\
			\langle \alpha, u \rangle_{F(\mathbb H)} = 0, \quad \alpha \in \mathbb C, \quad u \in \mathbb H, \\
			\langle u_1 \otimes \cdots \otimes u_n, v_1 \otimes \cdots \otimes v_n \rangle_{F(\mathbb H)} = \prod_{i=1}^n \langle u_i, v_i \rangle_{\mathbb H}, \quad u_i, v_i \in \mathbb H.
		\end{gathered}
	\end{equation}
	We also define the weighted symmetric Fock space ${F_w(\mathbb H)}$ as the closure of the symmetric tensor algebra $T(\mathbb H) := \mathbb C \oplus \mathbb H \oplus \mathbb H^{\vee 2} \oplus \ldots$ with respect to the inner product satisfying (cf.\ \eqref{eq:fock_innerp})
	\begin{equation*}
		\begin{gathered}
			\langle \alpha, \beta \rangle_{F_w(\mathbb H)} = \bar \alpha \beta, \quad \alpha, \beta \in \mathbb C, \\
			\langle \alpha, u \rangle_{F_w(\mathbb H)} = 0, \quad \alpha \in \mathbb C, \quad u \in \mathbb H, \\
			\langle u_1 \vee \cdots \vee u_n, v_1 \vee \cdots \vee v_n \rangle_{{F_w(\mathbb H)}} = \frac{w^2(n)}{n!^2} \sum_{\sigma,\sigma'\in S_n} \prod_{i=1}^n \langle u_{\sigma(i)}, v_{\sigma'(i)} \rangle_{\mathbb H}, \quad u_i, v_i \in \mathbb H,
		\end{gathered}
	\end{equation*}
	for a strictly positive weight function $w \colon \mathbb N_0 \to \mathbb R_{>0}$.
	Here, $\vee$ denotes the symmetric tensor product, defined as the average
	\begin{equation}
		\label{eq:sym_tensor_product}
		u_1 \vee \dots \vee u_n = \frac{1}{n!} \sum_{\sigma \in S_n} u_{\sigma(1)} \otimes \dots \otimes u_{\sigma(n)}, \quad u_i \in \mathbb H,
	\end{equation}
	over the $n$-element permutation group $S_n$, and $\mathbb H^{\vee n}$ is the closed subspace of $\mathbb H^{\otimes n}$ consisting of symmetric tensors.
	The map $u_1 \otimes \cdots \otimes u_n \mapsto u_1 \vee \cdots \vee u_n$ defines, by linear extension, the orthogonal projection from $\mathbb H^{\otimes n}$ to $\mathbb H^{\vee n}$.

	In \cite{GiannakisEtAl26}, it is shown that if $w^{-2}$ satisfies \eqref{eq:subconv_fock_space_weights}, ${F_w(\mathbb H)}$ becomes a unital Banach algebra with respect to the symmetric tensor product for a norm $\vertiii{\cdot}$ induced by the operator norm on multiplication operators,
	\begin{displaymath}
		\vertiii{\eta} = \sup_{\xi \in F_w(\mathbb H) \setminus \{0\}} \frac{\lVert \eta \vee \xi \rVert_{F_w(\mathbb H)}}{\lVert \xi \rVert_{F_w(\mathbb H)}}, \quad \vertiii{\eta \vee \xi} \leq \vertiii{\eta} \vertiii{\xi},
	\end{displaymath}
	and with $1_{\mathbb C} \in \mathbb C \subset F_w(\mathbb H)$ as the unit.
	Moreover, associated with ${F_w(\mathbb H)}$ is a coproduct, i.e., a bounded operator $\Delta \colon {F_w(\mathbb H)} \to {F_w(\mathbb H)} \otimes {F_w(\mathbb H)}$ such that
	\begin{displaymath}
		\Delta^*(\eta \otimes \xi) = \eta \vee \xi.
	\end{displaymath}
	Boundedness of the coproduct in conjunction with unitality imply that the Banach algebra norm $\vertiii{\cdot}$ and the Hilbert space norm $\lVert \cdot \rVert_{F_w(\mathbb H)}$ are equivalent, analogously to the case of a unital RKHA described in \cref{app:rkha}.

	The structural similarities between the symmetric Fock space setting in this appendix and the function space setting in \cref{app:rkha} continue at the level of the (co)spectrum.
	Specifically, we may characterize the cospectrum of ${F_w(\mathbb H)}$ (i.e., the set of nonzero vectors $\xi \in F_w(\mathbb H)$ satisfying $\Delta \xi = \xi \vee \xi$) as the set
	\begin{displaymath}
		\sigma_\text{co}({F_w(\mathbb H)}) = \left\{ \xi = \sum_{n=0}^\infty w^{-2}(n) u^{\vee n}: u \in (\mathbb H)_{R_w} \setminus \{0\} \right\} \subset F_w(\mathbb H),
	\end{displaymath}
	where $R_w \geq 1$ is the radius of convergence of the series $\sum_{n=1}^\infty w^{-2}(n) z^n$, $z \in \mathbb C$, and $(\mathbb H)_{R_w} \subset \mathbb H $ is the closed ball of radius $R_w$ centered at the origin.
	The character spectrum $\sigma(F_w(\mathbb H)) \subset F_w(\mathbb H)^*$ (i.e., the set of nonzero multiplicative linear functionals on $F_w(\mathbb H)$) with respect to the symmetric tensor product may then be identified with the  Riesz dual of $\sigma_\text{co}(F_w(\mathbb H))$,
	\begin{equation*}
		\xi \in \sigma_\text{co}(F_w(\mathbb H)) \iff \langle \xi, \cdot \rangle_{F_w(\mathbb H)} \in \sigma(F_w(\mathbb H)).
	\end{equation*}
	We equip $\sigma({F_w(\mathbb H)})$ and $\sigma_\text{co}({F_w(\mathbb H)})$ with the weak$^*$ topology on ${F_w(\mathbb H)}^*$ and the weak topology on ${F_w(\mathbb H)}$, respectively.
	With these topologies, they become homeomorphic compact Hausdorff spaces.

	As mentioned in \cref{sec:fock_space_finite_dim}, the weighted Fock space ${F_w(\mathbb H)}$ may be realized as an RKHA of continuous functions on $\sigma({F_w(\mathbb H)})$, denoted as $\widehat{F}_w(\mathbb H)$, via the Gelfand map $\Gamma \colon F_w(\mathbb H) \to C(\sigma(F_w(\mathbb H)))$, where $(\Gamma\xi)(\chi) = \chi(\xi)$.
	Specifically, the reproducing kernel $\widehat k\colon \sigma(F_w(\mathbb H)) \times \sigma(F_w(\mathbb H)) \to \mathbb C$ of $\widehat F_w(\mathbb H)$ is given by
	\begin{equation*}
		\widehat{k}(\chi_1, \chi_2) = \sum_{n=0}^\infty w^{-2}(n) \langle u_1, u_2 \rangle_{\mathbb H}^n,
	\end{equation*}
	for characters $ \chi_i = \langle \xi_i, \cdot \rangle_{F_w(\mathbb H)}$ with corresponding elements $\xi_i = \sum_{n=0}^\infty w^{-2}(n) u_i^{\vee n}$ of the cospectrum parameterized by $u_i \in (\mathbb H)_{R_w}$.

	In addition, given a feature map $\kappa\colon X \to \mathbb H$ taking values in $(\mathbb H)_{R_w}$, we may apply the procedure described in \cref{app:rkha_examples} to build a unital RKHA $\mathcal H_{\tilde \kappa}$ on $X$, whose reproducing kernel $\tilde k$ is given by~\eqref{eq:power_series_kernel}.
	We also have an induced feature map $\mathcal K \colon X \to F_w(\mathbb H)$ into the weighted Fock space, where
	\begin{equation*}
		\mathcal K(x) \equiv \mathcal K_x = \sum_{n=0}^\infty w^{-2}(n) (\kappa(x))^{\vee n} \in \sigma_\text{co}(F_w(\mathbb H)).
	\end{equation*}
	The RKHA $\mathcal H_{\tilde k}$ on $X$ is then isomorphic to the closed subspace $F_w(\mathbb H)(X) \subseteq F_w(\mathbb H)$ generated by the feature vectors $\mathcal K_x$,
	$$ \mathcal H_{\tilde k} \cong F_w(\mathbb H)(X) := \overline{\text{span}\left\{ \mathcal K_x : x \in X \right\}}.$$
	This isomorphism is implemented by the map $\pi\colon F_w(\mathbb H) \to \mathcal H_{\tilde k}$ defined by linear extension of $\pi(\eta \vee \xi)(x) = \langle \mathcal K_x, \eta \rangle_{F_w(\mathbb H)}\langle \mathcal K_x, \xi\rangle_{F_w(\mathbb H)}$.

	\section{Alternative first amplification}
	\label{app:alt_first_amplification}

	In this appendix we describe the alternative amplification scheme for $V$ and $V^*$ to essentially skew-adjoint operators $L'$ and $L^{+\prime}$, respectively, mentioned in \cref{rk:alt_first_amplification}.
	These operators are defined on the same dense domain $\mathfrak H_0 = C_c(\mathbb N) \otimes_\text{alg} \mathcal H_\infty \subset \mathfrak H$ as $L$ and $L^+$, but instead of being represented by tridiagonal operator matrices featuring $A$ and $S$ they are represented by hollow tridiagonal matrices featuring $V$ and $V^*$.
	Specifically, we define
	\begin{equation}
		\label{eq:alt_l_ops}
		\begin{aligned}
			L'          & = \sum_{i=1}^\infty e_{i,i+1} \otimes (\alpha_i V^* +\beta_i V) - e_{i+1,i} \otimes (\alpha_i V +\beta_i V^*), \\
			L^{\prime+} & = \sum_{i=1}^\infty e_{i,i+1} \otimes (\alpha_i V +\beta_i V^*) - e_{i+1,i} \otimes (\alpha_i V^* +\beta_i V),
		\end{aligned}
	\end{equation}
	where $q \in (0,1)$ and
	\begin{equation}
		\label{eq:alpha_beta}
		\alpha_{2k}=\alpha_{2k+1} = q^3\frac{1-q^{4k}}{1-q^4}, \quad \beta_{2k}=\beta_{2k-1} = q \frac{1-q^{4k}}{1-q^4}, \quad l_i = \alpha_i V + \beta_i V^*.
	\end{equation}
	This leads to the skew-symmetric operator matrices (cf.~\eqref{eq:l_mats})
	\begin{equation*}
		L' =\left[
			\begin{array}{ccccc}
				0      & l_1^*  & 0      & 0      & \cdots \\
				-l_1   & 0      & l_2^*  & 0      & \cdots \\
				0      & -l_2   & 0      & l_3^*  & \cdots \\
				0      & 0      & -l_3   & 0      & \cdots \\
				\vdots & \vdots & \vdots & \vdots & \ddots \\
			\end{array}\right],
		\quad
		L^{\prime+} =\left[
			\begin{array}{ccccc}
				0      & l_1    & 0      & 0      & \cdots \\
				-l_1^* & 0      & l_2    & 0      & \cdots \\
				0      & -l_2^* & 0      & l_3    & \cdots \\
				0      & 0      & -l_3^* & 0      & \cdots \\
				\vdots & \vdots & \vdots & \vdots & \ddots \\
			\end{array}\right].
	\end{equation*}

	Next, define $E'_n \colon \mathcal H \to \mathfrak H$, $n \in \mathbb N$, as
	\begin{equation*}
		E'_nf = \sum_{i=1}^n \frac{1}{q^{i-1}} e_{i} \otimes f.
	\end{equation*}
	These maps play an analogous role to the Fibonacci embeddings $E_n$ from~\eqref{eq:amplification_embedding}, satisfying
	\begin{equation*}
		(p_1 \otimes \Id) L' E'_n f = Vf, \quad (p_1 \otimes \Id) L^{\prime+} E'_n f = V^*f,
	\end{equation*}
	for all $n >1$ and $f \in \mathcal H_\infty$.

	\begin{thm}[alternate of \cref{thm:amplification}]
		\label{thm:alt_amplification}
		For every $q \in (0, 1)$ the unbounded operators $L^{\prime+}$ and $L'$ from~\eqref{eq:alt_l_ops} have the following properties:
		\begin{enumerate}[(a)]
			\item $L'$ and $L^{\prime+}$ are essentially skew-adjoint on $\mathfrak H_0$.
			\item For every $n\in \mathbb N$, $k < n$, and $f \in \mathcal H_\infty$,
			      \begin{equation*}
				      (p_1 \otimes \Id) L^k E'_nf = V^kf, \quad (p_1 \otimes \Id) (L^{\prime+})^k E'_nf = (V^*)^kf.
			      \end{equation*}
		\end{enumerate}
	\end{thm}

	\begin{proof}
		We follow closely the proof of \cref{thm:amplification} and consider only the claims for $L^{\prime+}$.

		For part (a), observe that $\alpha_k$ and $\beta_k$ are strictly increasing and bounded above by $\alpha = q/^3(1-q^4)$ and $\beta=q/(1-q^4)$.
		We also need to analyze the combinatorial growth of terms due to the connectivity of $L^{\prime+}$ as an infinite matrix.
		Precisely, for $\delta,\epsilon,\eta \colon [n] \to [2]$, the number of terms of the form $\prod_{i=1}^n (\alpha_{\eta(i)} V_{\epsilon(i)} +\beta_{\eta(i)} V_{\delta(i)})f$, in the $k^{th}$ entry of $(L^{\prime+})^n(e_l \otimes f)$ is counted by the adjacency matrix for the line graph,
		$$\Gamma' =
			\begin{tikzpicture}
				\clip (-.1,-.1) rectangle (6.5,.5);
				\draw[very thick] (0,0)--(5,0);
				\draw[fill=black] (0,0) circle (.5mm);
				\draw[fill=black] (1,0) circle (.5mm);
				\draw[fill=black] (2,0) circle (.5mm);
				\draw[fill=black] (3,0) circle (.5mm);
				\draw[fill=black] (4,0) circle (.5mm);
				\draw[fill=black] (5.5,0) circle (.2mm);
				\draw[fill=black] (5.75,0) circle (.2mm);
				\draw[fill=black] (6,0) circle (.2mm);
			\end{tikzpicture}$$
		Since the operator norm of this adjacency matrix is 2, we have $\lvert \Gamma^{\prime n}(k,l) \rvert \leq \lVert \Gamma' \rVert^n = 2^n$.
		Then the vectors $e_l \otimes f \in l^2(\mathbb N) \otimes \mathcal H_\infty$ are analytic for $L^{\prime+}$ with a radius of convergence $R_f/(8\beta)$ when the radius of convergence for $f$ is $R_f$ as an analytic vector of $V$ and $V^*$:
		\begin{align*}
			\sum_{n=0}^\infty \frac{z^n}{n!} \left\lVert (L^{\prime+})^n (e_l \otimes f) \right\rVert_{\mathcal H}
			 & \leq \sum_{n=0}^\infty \frac{z^n}{n!} \sum_{\epsilon,\delta,\eta\colon [n] \to [2]} \left\lVert \prod_{i=1}^n (\alpha_{\eta(i)} V_{\epsilon(i)} +\beta_{\eta(i)} V_{\delta(i)}) f \right\rVert_{\mathcal H}                \\
			 & \leq \sum_{n=0}^\infty \frac{z^n}{n!} \left(\sum_{k=1}^{l+n} |\Gamma^{\prime n}(k,l)| \right) \beta^n 2^n \max_{\epsilon\colon [n] \to [2]} \left\lVert V_{\epsilon(1)} \cdots V_{\epsilon(n)} f \right\rVert_{\mathcal H} \\
			 & \leq \sum_{n=0}^\infty \frac{z^n}{n!} (l+n) \beta^n 4^n \sum_{\epsilon \colon [n] \to [2]} \left\lVert V_{\epsilon(1)} \cdots V_{\epsilon(n)} f \right\rVert_{\mathcal H}.
		\end{align*}
		Therefore, $\mathfrak H_0 = C_c(\mathbb N) \otimes_\text{alg}  \mathcal H_\infty$ is a dense family of analytic vectors for $L^{\prime+}$.

		For part (b) consider the formal vector $E' f = \sum_{i=1}^\infty \frac{1}{q^{i-1}} e_{i} \otimes f \in \mathcal H^{\mathbb N}$ and its behavior under $L^{\prime+}$.
		We claim that by the choice of $\alpha_i$ and $\beta_i$ from~\eqref{eq:alpha_beta},
		\begin{equation}
			\label{eq:alt_intertwine_formal}
			E' V^* f = L^{\prime+} E 'f
		\end{equation}
		as formal vectors and matrices.
		Indeed, we have
		\begin{align*}
			L^{\prime+} E'
			 & = \sum_{i=1}^\infty \frac{1}{q^{i}} e_i \otimes (\alpha_i(q) V +\beta_i(q) V^*) - \frac{1}{q^{i-1}} e_{i+1} \otimes (\alpha_i(q) V^* +\beta_i(q) V)                                                                        \\
			 & =\frac{1}{q} e_1 \otimes (\alpha_1(q) V +\beta_1(q) V^*)                                                                                                                                                                   \\
			 & \quad + \sum_{i=2}^\infty   e_i \otimes \left(\frac{1}{q^{i}}\alpha_i(q) -\frac{1}{q^{i-2}}\beta_{i-1}(q) \right) V + e_i \otimes \left(\frac{1}{q^{i}}\beta_i(q) - \frac{1}{q^{i-2}} \alpha_{i-1}(q)\right) V^*           \\
			 & = e_1 \otimes \frac{1}{q}(\alpha_1(q) V +\beta_1(q) V^*)                                                                                                                                                                   \\
			 & \quad + \sum_{i=2}^\infty \frac{1}{q^{i-1}} e_i \otimes\left( \left(\frac{1}{q}\alpha_i(q) -\frac{1}{q^{-1}}\beta_{i-1}(q) \right) V  +  \left(\frac{1}{q}\beta_i(q) - \frac{1}{q^{-1}} \alpha_{i-1}(q)\right) V^*\right),
		\end{align*}
		Equation~\eqref{eq:alt_intertwine_formal} then follows from~\eqref{eq:alpha_beta} as the latter implies
		\begin{gather*}
			\frac{1}{q}\alpha_1=0, \quad \frac{1}{q}\beta_1 =1, \\
			\left(\frac{1}{q}\alpha_i -\frac{1}{q^{-1}}\beta_{i-1} \right) = 0, \quad
			\left(\frac{1}{q}\beta_i - \frac{1}{q^{-1}} \alpha_{i-1}\right)=1, \quad i \geq 2.
		\end{gather*}
		Since the infinite matrix $L^{\prime+}$ is banded, when the map $E'$ is truncated, we have $(p_1 \otimes \Id) (L^{\prime +})^k E'_n f = (p_1 \otimes \Id) (L^{\prime +})^k E f= (V^*)^k f$ as formal vectors provided $k < n$.
	\end{proof}

	\section{The ideal generated by compact diagonal operators in $\mathcal A$}\label{app:ideal}

	While the $C^*$-algebras $\mathcal A$ and $\mathcal G = \overline{\mathcal G_\infty}$ (see \cref{standard_form}) are noncommutative, the creation and annihilation operators commute asymptotically when restricted to higher gradings of the Fock space.
	To make this precise, we may build a two-sided ideal
	$$\mathcal I_\infty = \text{span} \left\{ a^*_{\xi_1} \cdots a_{\xi_k}^* M_h a_{\eta_1} \cdots a_{\eta_l} : \xi_i,\eta_j \in \mathfrak H_0, \, h \in c_0(\mathbb N_0) \right\} \subset \mathcal G_\infty$$
	and its closed counterpart
	$$\mathcal I = \overline{\mathcal I_\infty} = \overline{\text{span} \left\{ a^*_{\xi_1} \cdots a_{\xi_k}^* M_h a_{\eta_1} \cdots a_{\eta_l} : \xi_i,\eta_i \in \mathfrak H, \, h \in c_0(\mathbb N_0) \right\}} \subset \mathcal G.$$
	Note that $M_h$ is a compact operator on the weighted Fock space $\mathfrak F$ whenever $h$ vanishes at infinity.
	Thus, since all creation and annihilation operators in $\mathcal G_\infty$ are bounded, every element of $\mathcal I_\infty$ is compact.

	\begin{prop}
		Under \cref{asmp:main} and using the subexponential weights in~\eqref{eq:subexp_fock_space_weights}, the closed subspace $\mathcal I \subset \mathcal G$ does not contain the identity and is a two-sided closed ideal.
		Furthermore, the $C^*$-algebra $\mathcal A/(\mathcal A \cap \mathcal I)$ is commutative.
		Finally, the map $\Adj_{e^{\overline{\mathcal D} t}}$ is an automorphism of $\mathcal G$, an automorphism of $\mathcal A$, and $\Adj_{e^{\overline{\mathcal D} t}}(\mathcal I) = \mathcal I$.
	\end{prop}

	\begin{proof}
		First observe that the space of sequences $c_0(\mathbb N_0)$ vanishing at infinity is invariant under the shift operators $S$ and $S^{-1}$.
		Then due to the commutation relations \eqref{eq:wccr}, $\mathcal I_\infty \subset \mathcal G_\infty$ is a two-sided ideal.
		The closure of $\mathcal I_\infty$ cannot contain the identity since functions $h \in c_0(\mathbb N_0)$ converge to zero at infinity.
		To verify this, let $\pi_k \colon \mathfrak F \to \mathfrak F$ be the orthogonal projection onto the $k^{th}$ grading of $\mathfrak F$, and fix $ a =\sum_{i=1}^N a_{\tilde{f_i}}^* M_{h_i} a_{\tilde{g_i}} \in \mathcal I_\infty$ with $m, N \in \mathbb N$, $\tilde f_i,\tilde g_i \in \bigoplus_{k=0}^m \mathfrak H^{\vee k}$, and $h_i \in c_0(\mathbb N_0)$.
		Since
		$$\lim_{n_0 \to \infty} \left(\sup_{i=1,...,N; k\geq n_0 -m} |h_i(k)| \right) =0,$$
		there exists $n_0 \in \mathbb N$ such that
		$$\left(\sup_{i=1,...,N; k\geq n_0 -m} |h_i(k)| \right) \sum_{i=1}^N \lVert a^*_{\tilde{f}_i} \rVert \lVert a_{\tilde{g}_i} \rVert \rvert < \frac{1}{2}.$$
		For such an $n_0$, we have
		$$\lVert \Id_\mathfrak F - a \rVert \geq \lVert \pi_{n_0} - \pi_{n_0}  a \pi_{n_0} \rVert \geq \left\lvert \lVert \pi_{n_0}\rVert-\left(\sup_{i=1,...,N; k\geq n_0 -m} |h_i(k)| \right) \sum_{i=1}^N \lVert a^*_{\tilde{f}_i} \rVert \lVert a_{\tilde{g}_i} \rVert \right\rvert \geq \frac{1}{2},$$
		proving that $\mathcal I_\infty$ is uniformly bounded away from the identity in $B(\mathfrak F)$.
		The fact that $\mathcal I$ is a nontrivial two-sided closed ideal follows immediately since $\mathcal I = \overline{\mathcal I_\infty}$.

		For commutativity of $\mathcal A/(\mathcal A \cap \mathcal I)$, we first prove that $\mathcal A_\infty/(\mathcal A_\infty \cap \mathcal I_\infty)$ is commutative.
		Indeed, by~\eqref{eq:wccr5} and~\eqref{eq:s_q_limits}, we have
		$$a_\xi a_\eta^* - a_\eta^* a_\xi = a_\eta^* M_{\frac{q}{S(q)} - 1} a_\xi + \langle \xi, \eta \rangle M_{\frac{1}{q}} \in \mathcal I_\infty,$$
		which shows that the creation and annihilation operators commute inside the quotient algebra $\mathcal A_\infty/(\mathcal A_\infty \cap \mathcal I_\infty)$.
		As an aside, note that $\mathcal G_\infty/\mathcal I_\infty$ is not commutative since $[M_h,a_\xi] = 2M_ha_\xi \notin \mathcal I_\infty$ for $h(n) = (-1)^n$, but $\mathcal A_\infty/(\mathcal A_\infty \cap \mathcal I_\infty)$ is commutative since $\mathcal A_\infty$ is only generated by the creation an annihilation operators.
		Now, for any $a,b \in \mathcal A$ there are operator norm convergent sequences in $\mathcal A_\infty$ such that $a_i \to a$ and $b_i \to b$.
		Then the commutator $[a_i,b_i] \in \mathcal A_\infty \cap \mathcal I_\infty$ and $[a_i,b_i] \to [a,b]$ in operator norm, hence $[a,b] \in \mathcal A \cap \mathcal I$, proving that $\mathcal A/(\mathcal A \cap \mathcal I)$ is commutative.

		We have already seen in \cref{sec:fock_space_finite_dim} that $\Adj_{e^{\overline{\mathcal D} t}}$ preserves $\mathcal A$.
		The fact that it preserves $\mathcal G$ and $\mathcal I$ follows similarly from unitarity and the relations
		$$\Adj_{e^{\overline{\mathcal D} t}}(a_\xi^*) = a_{e^{Lt}\xi}^* \quad \quad \Adj_{e^{\overline{\mathcal D} t}}(a_\xi) = a_{e^{Lt}\xi}, \quad \Adj_{e^{\overline{\mathcal D} t}}(M_h) = M_h,$$
		which hold for every $\xi \in \mathfrak H_0$ and $h \in c_0(\mathbb N)$.
	\end{proof}

	The fact that $\Adj_{e^{\overline{\mathcal D} t}}$ is an isometric automorphism that fixes $\mathcal I$ as a set implies that it induces an isometric isomorphism on the quotient $C^*$-algebra that we denote $\alpha^t$.
	Since $\mathcal A/(\mathcal A \cap \mathcal I)$ is commutative, we obtain a compact Hausdorff space from its spectrum $\Omega = \sigma\left( \mathcal A/(\mathcal A \cap \mathcal I) \right)$, a space of continuous functions $C(\Omega) \cong \mathcal A/(\mathcal A \cap \mathcal I)$, a space of regular functions $C^\infty(\Omega) :=\{  a+(\mathcal A \cap \mathcal I) :\; a \in \mathcal A_\infty\} \subset C(\Omega)$, and a one-parameter family of automorphisms $\alpha^t \subset \Aut(C(\Omega))$ inducing a flow $\tilde{\Phi}^t \colon \Omega \to \Omega$.

	If \cref{asmp:normalized_kernel} holds additionally, the states $\Psi_x \colon \mathcal A \to \mathbb C$ also satisfy $\Psi_x(\mathcal A \cap \mathcal I)=0$ due to \cref{prop:state_embedding}.
	This means $\Psi$ also defines a map $\Psi \colon X \to \Omega$.
	Moreover, since $[\mathcal D, M_h]=0$,
	\begin{align*}
		[\mathcal D, a^*_{\xi_1} \cdots a_{\xi_k}^* M_h a_{\eta_1} \cdots a_{\eta_l}] & = a^*_{L(\xi_1)} \cdots a_{\xi_k}^* M_h a_{\eta_1} \cdots a_{\eta_l}                                                                                                     \\
		                                                                              & \quad +a^*_{\xi_1} \cdot a_{L(\xi_2)}^* \cdots a_{\xi_k}^* M_h a_{\eta_1} \cdots a_{\eta_l}+ \cdots + a^*_{\xi_1} \cdots a_{L(\xi_k)}^* M_h a_{\eta_1} \cdots a_{\eta_l} \\
		                                                                              & \quad +a^*_{\xi_1} \cdots a_{\xi_k}^* M_h a_{L(\eta_1)} \cdots a_{\eta_l} + \cdots + a^*_{\xi_1} \cdots a_{\xi_k}^* M_h a_{\eta_1} \cdots a_{L(\eta_l)},
	\end{align*}
	and the commutator $[\mathcal D, \cdot]$ preserves $\mathcal A_\infty$ and $\mathcal A_\infty \cap \mathcal I$.
	Therefore $[\mathcal D,\cdot]$ induces a derivation on the quotient algebra $C^\infty(\Omega)$ that we will denote by $\delta \colon C^\infty(\Omega) \to C^\infty(\Omega)$.
	Furthermore, since
	$$\left.\frac{d}{dt} \right|_{t=0} \Adj_{e^{\mathcal D t}} (a) = [\mathcal D, a]$$
	we also have
	$$\left.\frac{d}{dt} \right|_{t=0} \alpha^t (a) = \delta(a).$$

\end{appendices}

\printbibliography

\end{document}